\documentclass[a4paper]{article}

\usepackage[utf8]{inputenc}
\usepackage[T1]{fontenc}
\usepackage{geometry}
\usepackage{lmodern}
\usepackage{microtype}

\usepackage{amsmath}
\usepackage{amssymb}
\usepackage{amsthm}
\usepackage{bbm}
\usepackage{bm}
\usepackage{mathtools}
\usepackage{mathrsfs}
\usepackage{multicol}
\usepackage{xfrac}

\usepackage{thmtools}
\usepackage{thm-restate}

\usepackage{caption}
\usepackage{subcaption}

\usepackage{csquotes}
\usepackage{graphicx}
\usepackage{todonotes}

\usepackage{academicons}
\definecolor{orcidlogocol}{HTML}{A6CE39}

\usepackage{tikz}
\usetikzlibrary{arrows.meta,
                calc,
                decorations,
                decorations.pathreplacing,
                positioning,
                shapes}

\usepackage[backend=biber,
            style=ext-numeric,
            maxbibnames=99,
            giveninits=true]{biblatex}
\DeclareFieldFormat
  [article,inbook,incollection,inproceedings,patent,thesis,unpublished]
  {titlecase:title}{\MakeSentenceCase*{#1}}

\usepackage{xpatch}

\makeatletter
\patchcmd\blx@bblinput{\blx@blxinit}
                      {\blx@blxinit
                      }{}{\fail}
\makeatother

\usepackage{hyperref}
\usepackage{xurl}
\hypersetup{breaklinks=true}
\usepackage[nameinlink]{cleveref}

\theoremstyle{plain}
\newtheorem{theorem}{Theorem}
\newtheorem{lemma}[theorem]{Lemma}
\newtheorem{proposition}[theorem]{Proposition}

\newtheorem{corollary}[theorem]{Corollary}

\theoremstyle{definition}
\newtheorem{definition}[theorem]{Definition}
\newtheorem{example}[theorem]{Example}
\newtheorem{remark}[theorem]{Remark}

\tikzset{vertex/.style={circle, fill=black, inner sep=2pt}
}

\newcommand{\multiGraphs}{\mathcal{M}}
\newcommand{\multiGraphsOf}[1]{\mathcal{M}^{#1}}
\newcommand{\graphs}{\mathcal{G}}
\newcommand{\graphsOf}[1]{\mathcal{G}^{#1}}

\newcommand{\adjGraphs}{\mathcal{A}}
\newcommand{\jNeighbors}{\mathcal{N}}
\newcommand{\adjNeiGraphs}{\mathcal{F}}

\newcommand{\OperatorFrom}[1]{T_{#1}}
\newcommand{\OperatorFromTo}[2]{\OperatorFrom{#1 \rightarrow #2}}
\newcommand{\OperatorN}{\OperatorFrom{\bm{N}}}

\newcommand{\OperatorAW}{\OperatorFromTo{\bm{A}}{W}}
\newcommand{\OperatorFW}{\OperatorFromTo{\bm{F}}{W}}

\newcommand{\OperatorFamilyFrom}[1]{\TT_{#1}}
\newcommand{\OperatorFamilyFromTo}[2]{\OperatorFamilyFrom{#1 \to #2}}

\makeatletter
\newcommand{\notsovast}{\bBigg@{2.5}}
\newcommand{\vast}{\bBigg@{3.5}}
\newcommand{\Vast}{\bBigg@{5}}
\makeatother

\newcommand{\Lk}[1]{\mathsf{L}^{#1}_{\operatorname{iso}}}
\newcommand{\Lkk}{\Lk{k}}

\newcommand{\cutDist}{\delta_\square}

\newcommand{\OperatorANuW}{\OperatorFromTo{\bm{A}}{\nu^k_W}}
\newcommand{\OperatorNNuW}{\OperatorFromTo{\bm{N}}{\nu^k_W}}
\newcommand{\OperatorPNuW}{\OperatorFromTo{\pi}{\nu^k_W}}

\newcommand{\kFamilyOperatorsShort}{\TT}
\newcommand{\TT}{\mathbb{T}}

\newcommand{\naturalKFamilyOperatorsShortU}{\mathbb{T}_U^k}
\newcommand{\naturalKFamilyOperatorsShortW}{\mathbb{T}_W^k}

\newcommand{\term}{\mathbb{F}}
\newcommand{\graphOfTerm}{[\![\term]\!]}
\newcommand{\graphOf}[1]{[\![#1]\!]}
\newcommand{\homFunOf}[1]{f_{#1}}

\newcommand{\homFunction}[2]{\homFunOf{#1 \to #2}}
\newcommand{\homFunctionFW}{\homFunction{\bm{F}}{W}}
\newcommand{\homFunctionT}[1]{\homFunction{\term}{#1}}
\newcommand{\homFunctionOp}[1]{\homFunction{#1}{\kFamilyOperatorsShort}}
\newcommand{\homFunctionTOp}{\homFunctionT{\kFamilyOperatorsShort}}

\newcommand{\jsection}{{\bar{x}[\cdot/j]}}
\newcommand{\jnsection}{{\jsection \circ p_j}}

\newcommand{\normI}[1]{\lVert #1 \rVert_\infty}\newcommand{\normT}[1]{\lVert #1 \rVert_2}

\newcommand{\identityGraph}{{\bm{1}^k}}
\newcommand{\identityGraphOf}[1]{{\bm{1}^{#1}}}
\newcommand{\permutationGraph}{\bm{P}_\pi}
\newcommand{\jNeighbor}[1]{\bm{N}^k_{#1}}
\newcommand{\jNeighborj}{\jNeighbor{j}}
\newcommand{\jIntroduce}[1]{\bm{I}^k_{#1}}
\newcommand{\jIntroducej}{\jIntroduce{j}}
\newcommand{\jForget}[1]{\bm{F}^k_{#1}}
\newcommand{\jForgetj}{\jForget{j}}
\newcommand{\ijAdjacencyij}{\bm{A}^k_{ij}}
\newcommand{\adjacencyGraphOf}[2]{\bm{A}^{#1}_{#2}}
\newcommand{\introduceGraphOf}[2]{\bm{I}^{#1}_{#2}}
\newcommand{\forgetGraphOf}[2]{\bm{F}^{#1}_{#2}}
\newcommand{\neighborGraphOf}[2]{\bm{N}^{#1}_{#2}}
\newcommand{\blProd}{\cdot}
\newcommand{\treeClosure}[1]{\langle #1 \rangle_{\circ, \blProd}}
\newcommand{\pathClosure}[1]{\langle #1 \rangle_{\circ}}
\newcommand{\treeClosureK}{\treeClosure{\adjNeiGraphs^k}}

\newcommand{\iPk}{P^k}
\newcommand{\Mk}{\mathbb{M}^k}
\newcommand{\MM}{\mathbb{M}}
\newcommand{\MOf}[1]{\mathbb{M}^{#1}}
\newcommand{\MInfty}{\mathbb{M}^\infty}
\newcommand{\Pk}{\mathbb{P}^k}

\newcommand{\OWLk}{\mathbb{WL}^k}
\newcommand{\WLk}{\mathbb{WL}^k}

\newcommand{\OWLOf}[1]{\mathbb{WL}^{#1}}

\newcommand{\Tk}{\mathcal{T}^k}

\newcommand{\Borel}{\mathcal{B}}

\newcommand{\Topology}{\mathcal{T}}
\newcommand{\subAlg}{\mathcal{C}}
\newcommand{\subAlgC}{\mathcal{C}}
\newcommand{\subAlgD}{\mathcal{D}}
\newcommand{\subAlgs}{\Theta(\Borel, \mu)}
\newcommand{\subAlgsk}{\Theta({\Borel^{\otimes k}},{\mu^{\otimes k}})}

\newcommand{\kSubAlg}{\subAlg^k}
\newcommand{\kInftySubAlg}{\subAlg^k}

\newcommand{\LL}{\mathcal{L}}
\newcommand{\LTwo}{L^2(X, \mu)}
\newcommand{\LTwoLong}{L^2(X, \Borel, \mu)}
\newcommand{\LTwoY}{L^2(Y, \nu)}

\newcommand{\LTwoOne}{L^2(X_1, \mu_1)}
\newcommand{\LTwoTwo}{L^2(X_2, \mu_2)}

\newcommand{\LTwoSub}{L^2(X, \subAlg, \mu)}
\newcommand{\LTwoSubD}{L^2(X, \subAlgD, \mu)}
\newcommand{\LTwoQuo}{L^2(X/\subAlg, \mu/\subAlg)}
\newcommand{\LTwoQuoD}{L^2(X/\subAlgD, \mu/\subAlgD)}

\newcommand{\LTwoProdK}{L^2(X^k, \mu^{\otimes k})}
\newcommand{\LTwoProdKLong}{L^2(X^k, \Borel^{\otimes k}, \mu^{\otimes k})}
\newcommand{\LTwoProdKOf}[1]{L^2(X^k, #1, \mu^{\otimes k})}

\newcommand{\LInftyProdK}{L^\infty(X^k, \mu^{\otimes k})}
\newcommand{\LTwoProdKM}{L^2(X^{k-1}, \mu^{\otimes k-1})}
\newcommand{\LTwoProdI}{L^2(X^i, \mu^{\otimes i})}
\newcommand{\LTwoProdIM}{L^2(X^{i-1}, \mu^{\otimes i - 1})}
\newcommand{\LTwoProdKQuo}{L^2(X^k/\subAlg, \mu^{\otimes k}/\subAlg)}
\newcommand{\LTwoProdKQuoOf}[1]{L^2(X^k/#1, \mu^{\otimes k}/#1)}
\newcommand{\LTwoProdKQuoD}{L^2(X^k/\subAlgD, \mu^{\otimes k}/\subAlgD)}
\newcommand{\LTwoProdL}{L^2(X^\ell, \mu^{\otimes \ell})}
\newcommand{\LLTwoProdL}{\LL^2(X^\ell, \mu^{\otimes \ell})}

\newcommand{\LLInftyProdL}{\LL^\infty(X^\ell, \mu^{\otimes \ell})}

\newcommand{\LTwoMkNu}{L^2(\Mk, \nu)}
\newcommand{\LInftyMkNu}{L^\infty(\Mk, \nu)}
\newcommand{\LLInftyMkNu}{\LL^\infty(\Mk, \nu)}

\newcommand{\LTwoMkNuKW}{L^2(\Mk, \nu^k_W)}
\newcommand{\LInftyMkNuKW}{L^\infty(\Mk, \nu^k_W)}
\newcommand{\dmu}{\,d\mu}
\newcommand{\dmuk}{\,d\mu^{\otimes k}}

\newcommand{\stepDownOf}[1]{#1\!\!\downarrow\!{}}
\newcommand{\stepDown}{\stepDownOf{S}}

\newcommand{\LInftyTwo}{L^\infty(X_2, \mu_2)}

\newcommand{\LInftyY}{L^\infty(Y, \nu)}

\newcommand{\E}{\mathbb{E}}
\newcommand{\ExpVal}[1]{\mathbb{E}(-\mid #1)}
\newcommand{\ExpValSub}{\ExpVal{\subAlg}}

\newcommand{\measOne}{\mathscr{M}_{\le 1}}
\newcommand{\probMeas}{\mathscr{P}}

\newcommand{\allOne}{\boldsymbol{1}}

\renewcommand{\hom}{\mathsf{hom}}

\newcommand{\N}{\mathbb{N}}
\newcommand{\NInfty}{\mathbb{N} \cup \Set{\infty}}
\newcommand{\R}{\mathbb{R}}

\newcommand{\es}{\mathsf{e}}

\newcommand{\Set}[1]{\left\{#1\right\}}
\newcommand{\MSet}[1]{\{\!\!\{#1\}\!\!\}}
\renewcommand{\epsilon}{\varepsilon}
\renewcommand{\emptyset}{\varnothing}

\newcommand\restr[2]{{\left.\kern-\nulldelimiterspace #1 \vphantom{\big|} \right|_{#2} }}

\DeclareMathOperator{\colref}{\mathsf{cr}}
\DeclareMathOperator{\atp}{\mathsf{atp}}
\DeclareMathOperator{\wl}{\mathsf{wl}}
\DeclareMathOperator{\owl}{\mathsf{owl}}

\newcommand{\owlk}{\owl^k}
\newcommand{\invSpaceSmall}{\hspace{-4pt}}
\newcommand{\invSpace}{\hspace{-7pt}}

\newcommand{\simpleS}{\mathsf{s}}
\newcommand{\adjNeiGraphOf}[2]{\bm{S}^{#1}_{#2}}
\newcommand{\jAdjNei}{\adjNeiGraphOf{k}{j, V}}

\newcommand{\simpleAdjNeisK}{\mathcal{F}^{\simpleS k}}
\newcommand{\simpleAdjNeisKP}{\mathcal{F}^{\simpleS k + 1}}
\newcommand{\simpleTreeClosureK}{\treeClosure{\mathcal{F}^{\simpleS k}}}

\newcommand{\naturalSKFam}{\mathbb{T}_W^{\simpleS k}}
\newcommand{\naturalSKFamU}{\mathbb{T}_U^{\simpleS k}}
\newcommand{\skSubAlg}{\subAlg^{\simpleS k}}
\newcommand{\Msk}{\mathbb{M}^{\simpleS k}}
\newcommand{\Psk}{\mathbb{P}^{\simpleS k}}
\newcommand{\iPsk}{P^{\simpleS k}}
\newcommand{\isk}{\owl^{\simpleS k}}
\newcommand{\nusk}{\nu^{\simpleS k}}
\newcommand{\nusInfty}{\nu^{\simpleS  \infty}}
\newcommand{\acle}{\preccurlyeq}

\newcommand{\LInftyMskNu}{L^\infty(\Msk, \nu)}
\newcommand{\OperatorSNuW}{\OperatorFromTo{\bm{S}}{\nusk_W}}
\newcommand{\RskW}{R^{\simpleS  k}_W}
\newcommand{\Rsk}{R^{\simpleS  k}}
\newcommand{\LInftyMskNuskW}{L^\infty(\Msk, \nusk_W)}
\newcommand{\LTwoMskNuskW}{L^2(\Msk, \nusk_W)}
\newcommand{\TTsk}{\TT}
\newcommand{\LTwoMskNu}{L^2(\Msk, \nu)}
\newcommand{\Tsk}{\mathcal{T}^{\simpleS k}}
\newcommand{\MsInfty}{\mathbb{M}^{\simpleS \infty}}

\newcommand{\nonoSimpleS}{\mathsf{ns}}

\newcommand{\nonoSimplesK}{\mathcal{F}^{\nonoSimpleS k}}
\newcommand{\nonoSimplesKOf}[1]{\mathcal{F}^{\nonoSimpleS {#1}}}
\newcommand{\nonoSimpleTreeClosureK}{\treeClosure{\nonoSimplesK}}

\begin{document}

\title{Weisfeiler-Leman Indistinguishability of Graphons}
\author{Jan Böker\\RWTH Aachen University, Aachen, Germany\\\href{mailto:boeker@informatik.rwth-aachen.de}{boeker@informatik.rwth-aachen.de}\\ORCID: \href{https://orcid.org/0000-0003-4584-121X}{0000-0003-4584-121X}}
\maketitle

\begin{abstract}
    The color refinement algorithm is mainly known as a
    heuristic method for graph isomorphism testing.
    It has surprising but natural characterizations
    in terms of, for example,
    homomorphism counts from trees and
    solutions to a system of linear equations.
    Greb\'ik and Rocha (2022)
    have recently shown how color refinement and
    notions that characterize it
    generalize to graphons,
    which emerged as limit objects in the theory of dense graph limits.
    In particular, they show that these characterizations are still equivalent
    in the graphon case.
    The $k$-dimensional Weisfeiler-Leman algorithm ($k$-WL) is a more powerful
    variant of color refinement that colors $k$-tuples instead of single
    vertices, where the terms
    $1$-WL and color refinement are often used
    interchangeably since they compute equivalent colorings.
    We show how to adapt the result of Greb\'ik and Rocha to
    $k$-WL or, in other words,
    how $k$-WL and its characterizations
    generalize to graphons.
    In particular, we obtain characterizations in terms of
    homomorphism densities from \textit{multi}graphs of bounded treewidth
    and linear equations.
    We give a simple example that
    parallel edges make a difference
    in the more general case of graphons,
    which means that, there,
    the equivalence between $1$-WL and color refinement
    does not hold anymore.
    We also show how this equivalence can be recovered
    by defining a variant of $k$-WL that
    corresponds to
    homomorphism densities from simple graphs of bounded treewidth.
\end{abstract}

\section{Introduction}

\textit{Color refinement} is a polynomial-time algorithm
best known as an efficient heuristic for graph isomorphism testing
even though it has more applications, e.g.,
as \text{graph kernels} in machine learning \cite{GroheEtAl2021}.
It iteratively computes a coloring of the vertices of a simple graph,
and two graphs are called \textit{indistinguishable by color refinement}
if the resulting multisets of colors match.
For isomorphic graphs, the resulting multisets of colors are necessarily
the same,
but there are non-isomorphic graphs that still produce the same
multisets of colors.
The \textit{$k$-dimensional Weisfeiler-Leman algorithm ($k$-WL)} is a generalization
of color refinement that colors $k$-dimensional tuples of vertices
instead of single vertices.
This again yields a heuristic for graph isomorphism testing
by comparing the resulting multisets of colors, where
two graphs are called \textit{indistinguishable by $k$-WL}
if the resulting multisets of colors match.
Starting from $1$-WL, which is equivalent to color refinement,
it yields a hierarchy of ever-more-powerful polynomial-time algorithms,
none of which actually decides graph isomorphism \cite{cai_optimal_1992}.

There are various seemingly unrelated characterizations
of indistinguishability by color refinement, for example,
by homomorphism counts from trees \cite{Dvorak2010, Dell2018}
or by rational solutions to a certain system of linear equations
encoding graph isomorphism called \textit{fractional isomorphisms} \cite{Tinhofer1986, Tinhofer1991}.
Similar characterizations exist for the $k$-dimensional
Weisfeiler-Leman algorithm
in terms of homomorphism counts from graphs of bounded treewidth
\cite{Dvorak2010, Dell2018}
and Sherali-Adams relaxations
of the system of linear equations encoding graph isomorphism
\cite{Dell2018, ImmermanLander1990, AtseriasManeva2013, GroheOtto2015}.
\citeauthor{GrebikRocha2021} recently
investigated the graphon counterpart of fractional isomorphism
by first providing graphon counterparts of
the most important notions used as characterizations for fractional isomorphism of graphs
and then proving that they are all equivalent \cite{GrebikRocha2021}.
\textit{Graphons} emerged as limit objects
for sequences of graphs in the theory of dense graph limits
developed by
Borgs, Chayes, Lov{\'a}sz, S{\'o}s, Szegedy, and Vesztergombi
\cite{lovasz_limits_2006, borgs_convergent_2008, borgs_convergent_2012}.
In this theory, \textit{homomorphism densities} play a crucial role
as a starting point for a notion of convergence of a sequence of graphs
and lead to the \textit{cut distance}, a (pseudo-)metric on graphons.
The book of \citeauthor{Lovasz2012} \cite{Lovasz2012}
provides a comprehensive overview.

In this paper, we provide graphon counterparts
of the most important notions used as characterizations for
$k$-WL indistinguishability
and also prove their equivalence.
As \citeauthor{GrebikRocha2021} stress,
for fractional isomorphism of graphons,
both defining the corresponding notions
and proving their equivalence
turned out to be surprisingly difficult.
This is no different in the case of $k$-WL indistinguishability,
where somewhat unsurprisingly, even more difficulties
and technical hurdles arise.
In particular, it turns out that there is no clear single generalization
of $k$-WL indistinguishability to graphons
but multiple non-equivalent variants.
The arguably most interesting characterization we obtain is in terms
of homomorphism densities:
The starting point of \citeauthor{GrebikRocha2021}
was to call two graphons $U$ and $W$
\textit{fractionally isomorphic} if
the homomorphism density $t(T, U)$ of $T$ in $U$ equals
the homomorphism density of $T$ in $W$
for every (finite simple) tree $T$.
Based on this and the characterizations of $k$-WL for graphs,
it is only natural to ask what kind of similarity
$U$ and $W$ have to satisfy for $t(F, U) = t(F, W)$
to hold for every (finite simple) graph $F$ of treewidth at most $k$.
While we give an answer to this,
we also show that a much more elegant characterization
and direct correspondence to the usual definition of $k$-WL is obtained
if we require $t(F, U) = t(F, W)$
to hold for every \emph{multi}graph $F$ of treewidth at most $k$ instead.

\subsection{Finite Graphs}
\label{sec:finiteGraphs}

In this section, we give a brief description
of color refinement and $k$-WL for finite graphs.
Moreover, we briefly present the notions characterizing them
that are relevant to us.
A \textit{(finite simple) graph} is a pair $G = (V, E)$,
where $V$ is a set of \textit{vertices}
and $E \subseteq \binom{V}{2}$ a set of \textit{edges}.
We usually write $V(G) \coloneqq V$ and $E(G) \coloneqq E$.
The initial coloring of color refinement for the vertices of a graph $G$
is defined by simply letting $\colref_{G,0}(v) \coloneqq 1$
for every vertex $v \in V(G)$.
Then, for every $n \ge 0$, let
\begin{equation*}
    \colref_{G,n+1}(v) \coloneqq (\colref_{G,n}(v), \MSet{\colref_{G,n}(w) \mid wv \in E(G)})
\end{equation*}
for every $v \in V(G)$, where
$\MSet{\cdot}$ is used to denote a \textit{multiset}.
Hence, the new color of a vertex $v$ is determined by aggregating
the colors of all neighbors of $v$, and in particular,
two vertices $u$ and $v$ get different colors if they have a different number
of neighbors of some color $c$.
Every coloring $\colref_{G,n}$ induces a partition of $V(G)$, and
after a finite number of steps, the coloring we obtain is \textit{stable},
i.e., the next  and all further colorings induce the same partition.
For graphs $G$ and $H$, we say that
\textit{$G$ and $H$ are indistinguishable by color refinement}
if
$\MSet{\colref_{G,n}(v) \mid v \in V(G)}
= \MSet{\colref_{H,n}(v) \mid v \in V(H)}$
holds for every $n \ge 0$.

The notions characterizing indistinguishability by color refinement
that are important for us are
tree homomorphisms, fractional isomorphisms, and stable partitions.
A \textit{homomorphism} from a graph $F$ to a graph $G$
is a mapping $h \colon V(F) \to V(G)$
such that $uv \in E(F)$ implies $h(u)h(v) \in E(G)$.
The number of homomorphisms from $F$ to $G$ is denoted
by $\hom(F, G)$, and
$t(F, G) \coloneqq \hom(F, G) / \lvert V(G) \rvert^{\lvert V(F) \rvert}$
is the \textit{homomorphism density of $F$ in $G$}.
Then, a result of \citeauthor{Dvorak2010} states
two graphs $G$ and $H$ are not distinguished by color refinement
if and only if the number of homomorphisms $\hom(T, G)$ from $T$ to $G$
equals the corresponding number $\hom(T, H)$ from $T$ to $H$
for every tree $T$ \cite{Dvorak2010}, see also \cite{Dell2018}.
An older result due to Tinhofer \cite{Tinhofer1986, Tinhofer1991}
states that $G$ and $H$ are not distinguished by color refinement
if and only if they are \textit{fractionally isomorphic}, i.e.,
there is a doubly stochastic matrix $X$ such that $AX = XB$,
where $A$ and $B$ are the adjacency matrices of $G$ and $H$, respectively.
A characterization that is more closely related
to the color refinement algorithm itself
is given by \textit{stable} partitions of the vertex set $V(G)$
of a graph $G$,
which are partitions where
all vertices in the same class have the same number of neighbors
in every other class.
One can show that the partition induced by the colors of color refinement
is the \textit{coarsest stable partition}
and that graphs $G$ and $H$ are fractionally isomorphic
if and only if their
coarsest stable partitions have the same parameters, i.e.,
there is a bijection between the partitions
that preserves the size of every class $C$
and the numbers of neighbors a vertex in $C$
has in some other class $D$ \cite{Tinhofer1986}.
This, in turn, is equivalent to there being
\textit{some} stable partitions of $G$ and $H$
with the same parameters \cite{RamanaEtAl1994}.

The \textit{$k$-dimensional Weisfeiler-Leman algorithm ($k$-WL)} is a variant of
color refinement that colors $k$-tuples of vertices instead of single vertices;
here and also throughout the paper, $k$ is an integer with $k \ge 1$.
See \cite{cai_optimal_1992} for an overview of the history of $k$-WL.
Usually, no distinction is made between $1$-WL and color refinement
as they, in some sense, compute equivalent colorings.
However, when
stating the formal definition of $k$-WL, it is important to note
that already for graphs there actually are two non-equivalent definitions
to be found in the literature.
Following Grohe \cite{Grohe2021},
we refer to these distinct definitions as \textit{(non-oblivious) $k$-WL} and \textit{oblivious $k$-WL}.
Both $k$-WL and oblivious $k$-WL
operate on $k$-tuples of vertices, but in terms of expressive power,
$k$-WL is equivalent to oblivious $k+1$-WL
in the sense that they distinguish the same graphs.
In this paper, we nearly always consider oblivious $k$-WL
and only briefly define non-oblivious $k$-WL at the end of the paper
in \Cref{sec:nonObliviousSimpleWL}.
However, to avoid any confusion, we nevertheless continue to explicitly use the
term oblivious $k$-WL from here on and use the term $k$-WL only
for non-oblivious $k$-WL.

Let $G$ be a graph.
To define oblivious $k$-WL, we first have to define
the \textit{atomic type} $\atp_G(\bar{v})$ of a tuple $\bar{v} = (v_1, \dots, v_k) \in V(G)^k$
of vertices of $G$, which is the $k \times k$-matrix $A$ with entries
$A_{ij} = 2$ if $v_i = v_j$, $A_{ij} = 1$ if $v_i v_j \in E(G)$, and
$A_{ij} = 0$ otherwise.
We let $\owl^k_{G,0}(\bar{v}) \coloneqq \atp_G(\bar{v})$ for every $\bar{v} \in V(G)^k$,
and then
for every  $n \ge 0$, we define
\begin{equation}
    \owl^k_{G, n+1}(\bar{v}) \coloneqq \left(\owl^k_{G, n}(\bar{v}), \big(\MSet{\owl^k_{G, n}(\bar{v}[w/j] \mid w \in V(G))}\big)_{j \in [k]}\right) \label{eq:obliviouskWL}
\end{equation}
for every $\bar{v} \in V(G)^k$.
Here, $\bar{v}[w/j]$ denotes the $k$-tuple obtained from $\bar{v}$
by replacing the $j$th component by $w$;
this $k$-tuple is usually called a \textit{$j$-neighbor}
of $\bar{v}$.
Hence, the new color of a tuple $\bar{v}$ is determined by aggregating
the colors of all $j$-neighbors of $\bar{v}$ for every $j \in [k]$, and in particular,
two tuples $\bar{u}$ and $\bar{v}$ get different colors if, for some $j \in [k]$,
they have a different number of $j$-neighbors of some color $c$.
We say that \textit{oblivious $k$-WL does not distinguish graphs $G$ and $H$}
if
$\MSet{\owl^k_{G,n}(\bar{v}) \mid \bar{v} \in V(G)^k}
    = \MSet{\owl^k_{H,n}(\bar{v}) \mid \bar{v} \in V(H)^k}$
for every $n \ge 0$.

The previously described notions
that characterize indistinguishability by color
refinement generalize to oblivious $k$-WL:
First of all,
oblivious $(k+1)$-WL does not distinguish graphs $G$ and $H$
if and only if the number of homomorphisms $\hom(F, G)$
from $F$ to $G$ is equal to the corresponding number $\hom(F, H)$ from $F$ to $H$
for every graph $F$
of treewidth at most $k$ \cite{Dvorak2010, Dell2018}.
A system of linear equations $\Lkk(G, H)$, which
is closely related to the Sherali-Adams relaxations
of the system of linear equations encoding graph isomorphism,
generalizes the concept of fractional isomorphisms:
oblivious $k$-WL does not distinguish $G$ and $H$ if and only if
$\Lkk(G, H)$ has a non-negative real solution,
cf.\  \cite{Dell2018} and also \cite{ImmermanLander1990, AtseriasManeva2013, GroheOtto2015}.
The precise formulation of $\Lkk(G, H)$ is given in \Cref{sec:opHierarchies}.
Stable partitions of the vertex set $V(G)$ of a graph $G$
can be generalized to \emph{stable partitions of $V(G)^k$},
where tuples with different atomic types are in different classes
and, for every $j \in [k]$, all tuples in the same class have the same number of
$j$-neighbors in every other class.
One can again show that the coloring computed by $k$-WL on $G$
induces the \textit{coarsest stable
partition of $V(G)^k$}
and two graphs $G$ and $H$
are not distinguished by $k$-WL if and only if
the coarsest stable partitions of $V(G)^k$ and $V(H)^k$
have the same parameters, which again is equivalent
to there being some stable partitions with the same parameters.
See for example \cite{GroheOtto2015}, where this is implicitly treated.

\begin{remark}
    \citeauthor{ImmermanLander1990} \cite{ImmermanLander1990}
    first showed that
    fractional isomorphism of graphs can also be seen from the perspective of logic;
    it corresponds to equivalence
    in the logic $\mathsf{C}^2$, the $2$-variable fragment
    of first-order logic with counting quantifiers.
    More generally, indistinguishability
    by oblivious $k$-WL corresponds to equivalence in $\mathsf{C}^{k}$,
    the $k$-variable fragment
    of first-order logic with counting quantifiers \cite{cai_optimal_1992}.
    However, this perspective
    does not play a further role in this paper.
\end{remark}

\subsection{Graphons and Homomorphism Densities}
\label{sec:graphonsIntroduction}

\textit{Graphons} emerged in the theory of graph limits as
limit objects of sequences of dense graphs;
we refer to the book of \citeauthor{Lovasz2012} \cite{Lovasz2012}
for a comprehensive treatment of this topic.
Formally, a graphon is a symmetric Borel- or Lebesgue-measurable
(this usually does not make a difference) function
$W \colon [0,1] \times [0,1] \to [0,1]$,
although it can be useful
to consider more general underlying spaces
than the unit interval with the Lebesgue measure.
A graph $G$ can be viewed as a graphon $W_G$
by partitioning $[0,1]$ into intervals $I_1, \dots, I_n$
of the same size---one for each vertex---and
setting $W_G(x, y)$ either to $1$ or $0$ for all $x \in I_i,\, y \in I_j$
depending on whether $ij$ is an edge in $G$ or not.
Similarly, a vertex- and edge-weighted graph $H$
with edge weights in $[0,1]$
can be viewed as graphon $W_H$, cf.\ \cite[Section $7.1$]{Lovasz2012}.
This allows one to restore statements about graphs
and weighted graphs from statements about graphons, e.g., the equivalence
between the notions characterizing fractional isomorphism of graphs
from the results of \citeauthor{GrebikRocha2021}.

We follow Greb\'ik and Rocha, and
throughout the whole paper,
let $(X, \Borel)$ denote
a \textit{standard Borel space}
and $\mu$ a \textit{Borel probability measure} on $X$;
using this as the underlying space for graphons
has the advantage that we later can consider \textit{quotient spaces}.
We think of $(X, \Borel, \mu)$ as \textit{atom free}, i.e., that
there is no singleton set of positive measure,
but do not formally require it.
A \textit{kernel} is a ($\Borel \otimes \Borel$)-measurable
map $W \colon X \times X \to [0,1]$, and
a symmetric kernel is called a \textit{graphon}.
An important way to view a kernel $W \colon X \times X \to [0,1]$
as an \textit{(bounded linear) operator} is by defining the
\textit{kernel operator} $T_W \colon \LTwo \to \LTwo$ by setting
\begin{equation}
    (T_W f) (x) \coloneqq \int_X W(x,y) f(y) \dmu(y)
    \label{eq:kernelOperator}
\end{equation}
for every $f \in \LTwo$ and every $x \in X$.
It is a well-defined Hilbert-Schmidt operator\cite[Section~$7.5$]{Lovasz2012},
and if $W$ is a graphon, then $T_W$ is self-adjoint,
i.e., its \textit{Hilbert adjoint} $T_W^*$ satisfies $T_W^* = T_W$;
in general,
the Hilbert adjoint of an operator
$S \colon \LTwo \to \LTwoY$,
where
$(X, \Borel)$ and $(Y, \mathcal{D})$ are standard Borel spaces
with Borel probability measures $\mu$ and $\nu$ on $X$ and $Y$, respectively,
is the unique operator
$S^* \colon \LTwoY \to \LTwo$
satisfying $\langle S f, g \rangle = \langle f, S^* g \rangle$
for all $f \in \LTwo, g \in \LTwoY$.

The \textit{homomorphism density of a graph $F$ in a graphon $W \colon X \times X \to [0,1]$} is
\begin{equation}
    t(F, W) \coloneqq \int_{X^{V(F)}} \prod_{ij \in E(F)} W(x_i, x_j) \, d\mu^{\otimes V(F)}(\bar{x}),
    \label{eq:homDensity}
\end{equation}
where $\bar{x}$ denotes the vector of all variables $x_i$ for $i \in V(F)$.
This coincides with the previous definition for graphs,
i.e.,
for the graphon $W_G$ obtained from a graph $G$
as  described above,
we have $t(F, G) = t(F, W_G)$ \cite[$(7.2)$]{Lovasz2012}.
Two graphons $U, W \colon X \times X \to [0,1]$
are called \textit{weakly isomorphic}
if $t(F,U) = t(F, W)$ for every simple graph $F$.
This is the usual notion of isomorphism used for graphons
and has various characterizations.
For example, two graphons $U$ and $W$ are weakly isomorphic if and only if
their \textit{cut distance} $\delta_\square(U, W)$ is zero,
cf.\ \cite[Section $10.7$]{Lovasz2012} for this result
and the definition of the cut distance.
The definition of weak isomorphism via homomorphism
densities is robust in the following sense:
A \textit{multigraph} is a tuple $G = (V, E)$
where $V$ is set of vertices and
$E$ is a \textit{multiset} of edges from $\binom{V}{2}$.
The definition
of the homomorphism density $t(F, W)$ of $F$ in
a graphon $W \colon X \times X \to [0,1]$
in Equation (\ref{eq:homDensity})
extends to the case in which $F$ is a multigraph,
where we slightly abuse notation and assume that
each factor $W(x_i, x_j)$ occurs as often in the product
$\prod_{ij \in E(F)} W(x_i, x_j)$ as $ij$ is contained in $E(F)$.
Then, two graphons $U$ and $W$ are weakly isomorphic
if and only if $t(F, U) = t(F, W)$ for every multigraph $F$
\cite[Corollary $10.36$]{Lovasz2012}, i.e.,
the definition of weak isomorphism
remains unchanged if we use multigraphs instead of simple graphs.

Coming from the definition of weak isomorphism,
the first definition in the paper of
\citeauthor{GrebikRocha2021} is to call
two graphons $U$ and $W$ \emph{fractionally isomorphic}
if $t(T, U) = t(T, W)$ holds for every tree $T$.
We follow suit with the following definition,
where the treewidth of a multigraph is defined
analogously to the case of simple graphs,
i.e., the edge multiplicities are not taken into account
and parallel edges count as a single edge.
\begin{definition}
    \label{def:kWLIndistinguishability}
    Two graphons $U, W \colon X \times X \to [0,1]$
    are called \emph{$k$-WL indistinguishable}
    if $t(T, U) = t(T, W)$ holds for every multigraph of treewidth at most $k$,
    and $U$ and $W$ are called \emph{simply $k$-WL indistinguishable}
    if $t(T, U) = t(T, W)$ holds for every simple graph of treewidth
    at most $k$.
\end{definition}

In these terms, two graphons $U$ and $W$ are weakly isomorphic
if and only if they are $k$-WL indistinguishable for every $k$,
which again is equivalent to them being
simply $k$-WL indistinguishable for every $k \ge 1$.
However, while $k$-WL indistinguishability clearly implies
simple $k$-WL indistinguishability, the converse does not hold in general.
To illustrate this,
let $E_\ell$ denote the multigraph that consists of two vertices
connected by $\ell$ parallel edges. Then, the homomorphism densities
of $E_\ell$ in a graphon $W \colon X \times X \to [0,1]$ for $\ell \ge 0$
are precisely the moments of $W$, i.e., we have
\begin{equation*}
    t(E_\ell, W) = \int_{X \times X} W(x,y)^\ell d (\mu \times \mu)(x,y).
\end{equation*}
Hence, for graphons $U$ and $W$ to be $1$-WL indistinguishable,
it is already necessary that they have the same moments,
while this is not required for $U$ and $W$ to be fractionally isomorphic.
For example, the constant graphon $W_c \colon X \times X \to [0,1], (x,y) \mapsto c$
for $c \in [0,1]$
is fractionally isomorphic to any $c$-regular graphon, i.e.,
a graphon $W \colon X \times X \to [0,1]$
satisfying $\int_X W(x,y) d \mu(y) = c$ for every $x \in X$,
but these may not have the same moments.
A concrete example is given by the weighted graphs in
\Cref{fig:parallelEdgesCounterexample}.

\begin{figure}
    \centering
    \begin{tikzpicture}
		\node[vertex] (E1) {};
        \node[vertex, below left = 1cm and 0.55cm of E1] (E2) {};
		\node[vertex, below right = 1cm and 0.55cm of E1] (E3) {};
        \path[draw, thick] (E1) edge node[xshift = -6pt, yshift = 4pt] {$\frac{2}{3}$} (E2);
        \path[draw, thick] (E1) edge node[xshift = 8pt, yshift = 4pt] {$\frac{2}{3}$} (E3);
        \path[draw, thick] (E2) edge node[xshift = 0pt, yshift = -8pt] {$\frac{2}{3}$} (E3);
        \path[every loop/.style={in = 120, out = 60}, draw, thick]
            (E1) edge[loop] node[above] {$\frac{2}{3}$} (E1);
        \path[every loop/.style={in = 240, out = 180}, draw, thick]
            (E2) edge[loop] node[left] {$\frac{2}{3}$} (E2);
        \path[every loop/.style={in = 360, out = 300}, draw, thick]
            (E3) edge[loop] node[right] {$\frac{2}{3}$} (E3);

		\node[vertex, right = 5cm of E1] (v1) {};
        \node[vertex, below left = 1cm and 0.55cm of v1] (v2) {};
		\node[vertex, below right = 1cm and 0.55cm of v1] (v3) {};
        \path[draw, thick] (v1) edge (v2);
        \path[draw, thick] (v1) edge (v3);
        \path[draw, thick] (v2) edge (v3);
    \end{tikzpicture}
    \caption{Two fractionally isomorphic weighted graphs that are $1$-WL distinguishable.}
    \label{fig:parallelEdgesCounterexample}
\end{figure}

Since $1$-WL indistinguishability is a more restrictive notion
than fractional isomorphism,
\Cref{def:kWLIndistinguishability} also
introduces simple $k$-WL indistinguishability.
Simple $1$-WL indistinguishability is just fractional isomorphism
since, for any class of graphs closed under connected components,
the homomorphism densities of the connected components determine
the homomorphism densities of all graphs in that class,
cf.\ \cite[$(7.6)$]{Lovasz2012}.
However counter-intuitive it may seem at this point,
the characterizations we obtain for
$k$-WL indistinguishability are much more natural than
these for simple $k$-WL indistinguishability.
In particular, the adaption of oblivious $(k+1)$-WL
to graphons corresponds to $k$-WL indistinguishability
and not simple $k$-WL indistinguishability.

As a last remark, we note that for $\{0,1\}$-graphons, i.e.,
graphons that only take the values $0$ and $1$,
parallel edges do not make a difference for homomorphism
densities since powers of $1$
are just $1$.
Hence, for $\{0,1\}$-graphons---of which graphs, or more precisely
graphons obtained from graphs, are a special case---,
the two notions of $k$-WL indistinguishability
and simple $k$-WL indistinguishability coincide.
In particular, $\{0,1\}$-graphons are $1$-WL indistinguishable
if and only if they are fractionally isomorphic.
Hence, while the terms color refinement and $1$-WL
are usually used synonymously in the literature,
it is important to not confuse these concepts
as they differ in the more general case of graphons.

\subsection{Fractional Isomorphism of Graphons}
\label{sec:fracIsoGraphons}

In this section, we try to give a more formal but still high-level overview
over the notions introduced by \citeauthor{GrebikRocha2021}
that all characterize fractional isomorphism of graphons.
We deem this essential for understanding our results.

Stating the color-refinement algorithm
for graphons
requires more formalism than in the case of graphs.
Greb\'ik and Rocha first define the standard Borel space $\MM$
of \textit{iterated degree measures},
which can be seen as the space of colors used by color refinement;
its elements are sequences $\alpha = (\alpha_0, \alpha_1, \alpha_2, \dots)$
of colors after $0,1,2, \dots$ refinement rounds.
For a graphon $W \colon X \times X \to [0,1]$,
they define a measurable function
$\colref_{W} \colon X \to \MM$
(denoted $i_W$ in their work)
mapping every $x \in X$ to such a sequence
$(\alpha_0, \alpha_1, \alpha_2, \dots)$.
Then, the \textit{distribution on iterated degree measures (DIDM)} $\nu_W$
is a probability measure on $\MM$ defined as
the \textit{push-forward of $\mu$ via $\colref_W$}, i.e., by
$\nu_W(A) \coloneqq \mu(\colref_W^{-1}(A))$ for every $A \in \Borel(\MM)$.
Intuitively, this is the distribution of all colors
assigned to the points of the graphon $W$
and corresponds to the multiset of all colors
used in the definition of color-refinement indistinguishability of graphs.

\citeauthor{GrebikRocha2021} show that
the graphon analogue to
stable partitions of the vertex set of a graph
are \emph{sub-$\sigma$-algebras} that satisfy certain properties.
To gain some intuition, consider the sub-$\sigma$-algebra $\{\emptyset, X\}$
of $\Borel$: in some way, it corresponds
to the partition of the vertex set of a graph that
consists of a single class containing all vertices.
Formally,
\citeauthor{GrebikRocha2021} consider \textit{$\mu$-relatively complete}
sub-$\sigma$-algebras,
where a sub-$\sigma$-algebra $\subAlg \subseteq \Borel$ of $\Borel$
is called \textit{$\mu$-relatively complete}
if $Z \in \subAlg$
for all $Z \in \Borel,\, Z_0 \in \subAlg$ with $\mu(Z \triangle Z_0) = 0$.
The set of all $\mu$-relatively complete sub-$\sigma$-algebras
of $\Borel$ is denoted by $\subAlgs$.
Then, for our example,
the smallest $\mu$-relatively complete sub-$\sigma$-algebra that includes
$\Set{\emptyset, X}$ corresponds
to the partition of the vertex set of a graph that
consists of a single class containing all vertices.

Let $\subAlg \in \subAlgs$ be a $\mu$-relatively complete
sub-$\sigma$-algebra $\subAlg$ of $\Borel$.
If we let $\LTwo \coloneqq \LTwoLong$ denote the Hilbert space
of all measurable real-valued functions on $X$ with
$\lVert f \rVert_2 < \infty$
modulo equality $\mu$-almost everywhere,
we can then consider the subspace
$L^2(X, \subAlg, \mu)$
of $\LTwo$
consisting of all $\subAlg$-measurable functions
since $\subAlg$ is $\mu$-relatively complete.
Moreover, there is a quotient space corresponding to $\subAlg$, i.e.,
a standard Borel space $(X/\subAlg, \subAlg')$ with
a Borel probability measure $\mu / \subAlg$ on $X/\subAlg$.
Alternatively, the \textit{conditional expectation $\ExpVal{\subAlg}$}, i.e.,
the orthogonal projection onto $\LTwoSub$, yields a different but equivalent
perspective on quotient spaces.

Now, to connect sub-$\sigma$-algebras
to \emph{stable} partitions, for an operator $T \colon \LTwo \to \LTwo$,
a $\mu$-relatively complete sub-$\sigma$-algebra
$\subAlg \in \subAlgs$ is called \textit{$T$-invariant}
if $\LTwoSub$ is $T$-invariant, i.e.,
$T(\LTwoSub) \subseteq \LTwoSub$.
Then, for a graphon $W \colon X \times X \to [0,1]$,
the $\mu$-relatively complete sub-$\sigma$-algebra $\subAlg \in \subAlgs$
is called \textit{$W$-invariant} if it is $T_W$-invariant,
where we recall that $T_W$ is the operator defined by
\Cref{eq:kernelOperator}.
\citeauthor{GrebikRocha2021} show that
there is a minimum $W$-invariant $\mu$-relatively complete sub-$\sigma$-algebra
$\subAlg_W$ of $\Borel$ (denoted $\subAlg(W)$ in their work),
which corresponds to the coarsest stable partition of the vertex set of a graph.

Finally, for a graphon $W \colon X \times X \to [0,1]$
and a $W$-invariant $\mu$-relatively
complete sub-$\sigma$-algebra $\subAlg \in \subAlgs$,
\citeauthor{GrebikRocha2021} define the
\textit{quotient graphon}
$W / \subAlg$ on the space $X / \subAlg \times X / \subAlg$.
Then, for graphons $U, W \colon X \times X \to [0,1]$,
saying that the two quotient graphons $U/\subAlg_U$
and $W / \subAlg_W$
are isomorphic corresponds
to saying that two coarsest stable partitions have the same parameters.
Alternatively to the quotient graphon $W / \subAlg$, one can also
consider $W_\subAlg \coloneqq \E(W \mid \subAlg \times \subAlg)$,
the conditional expectation of $W$ given $\subAlg \times \subAlg$.
Intuitively, the difference is that $W_\subAlg$ is obtained by
simply averaging over the color classes of $\subAlg$,
while $W/\subAlg$ is obtained by first averaging over the color classes of $\subAlg$
and then identifying all elements of a color class.

\begin{remark}
    Greb\'ik and Rocha show that
    every DIDM $\nu$ defines a kernel $\MM \times \MM \to [0,1]$
    and that,
    for a graphon $W \colon X \times X \to [0,1]$ and its DIDM $\nu_W$,
    this kernel on $\MM \times \MM$ is actually isomorphic to $W/\subAlg_W$.
    Intuitively, this can be viewed as a canonical representation of $W$
    on the space of all colors.
\end{remark}

For standard Borel spaces $(X, \Borel)$ and $(Y, \mathcal{D})$
with Borel probability measures $\mu$ and $\nu$ on $X$ and $Y$, respectively,
an operator $S \colon \LTwo \to \LTwoY$ is called a \textit{Markov operator}
if $S f \ge 0$ for every $f \in \LTwo$ with $f \ge 0$, $S \allOne_X = \allOne_Y$,
and $S^* \allOne_Y = \allOne_X$.
Here, $\allOne_X$ and $\allOne_Y$ denote the
all-one functions on $X$ and $Y$, respectively.
The Markov operator $S$ is called a \textit{Markov embedding} if it is an isometry,
i.e., $\normT{S f} = \normT{f}$ for every $f \in \LTwo$,
and a \textit{Markov isomorphism} if it is a surjective Markov embedding.
Markov operators are simply the infinite-dimensional
analogue to doubly stochastic matrices
and yield the graphon analogue to fractional isomorphisms.
The main result of Greb\'ik and Rocha then is the following \Cref{th:colRefGraphons}.

\begin{theorem}[{\cite{GrebikRocha2021}}]
    \label{th:colRefGraphons}
    Let $U, W \colon X \times X \to [0,1]$ be graphons.
    The following are equivalent:
    \begin{enumerate}
        \itemsep0em
        \item
            $t(T, U) = t(T, W)$ for every tree $T$.
            \label{th:colRefGraphons:hom}
        \item $\nu_U = \nu_W$.
            \label{th:colRefGraphons:DIDM}
        \item
            $W / \subAlg_W$ and $U / \subAlg_U$ are isomorphic.
            \label{th:colRefGraphons:minInvariant}
        \item
            There is a Markov operator
            $S \colon \LTwo \to \LTwo$
            such that
            $T_U \circ S = S \circ T_W$.
            \label{th:colRefGraphons:markov}
        \item
            There are $U$- and $W$-invariant $\mu$-relatively complete sub-$\sigma$-algebras
            $\subAlg$ and $\subAlgD$, respectively,
            such that $U_\subAlg$ and $W_\subAlgD$ are weakly isomorphic.
            \label{th:colRefGraphons:invariant}
    \end{enumerate}
\end{theorem}

Recall the various
notions characterizing fractional isomorphism of graphs
presented in \Cref{sec:finiteGraphs}.
Characterization~$(\ref{th:colRefGraphons:hom})$
corresponds to homomorphism numbers of trees and is the definition of
fractional isomorphism of graphons,
Characterization~$(\ref{th:colRefGraphons:DIDM})$
corresponds to color refinement not distinguishing two graphs,
Characterization~$(\ref{th:colRefGraphons:minInvariant})$
corresponds to the coarsest stable partitions of two graphs
having the same parameters,
Characterization~$(\ref{th:colRefGraphons:markov})$
generalizes fractional isomorphisms, and
Characterization~$(\ref{th:colRefGraphons:invariant})$
corresponds to some stable partitions of two graphs
having the same parameters.
We remark that
there is a subtle difference in the way
Greb\'ik and Rocha phrase
Characterization~(\ref{th:colRefGraphons:minInvariant})
and (\ref{th:colRefGraphons:invariant}):
the former uses quotient spaces and the stronger notion of isomorphism
while the latter uses
conditional expectation and weak isomorphism.

\subsection{Weisfeiler-Leman Indistinguishability of Graphons}
\label{sec:kWLGraphons}

We continue in the vein of \Cref{sec:fracIsoGraphons}
and give an overview of the notions characterizing $k$-WL
indistinguishability of graphons before stating our main result,
\Cref{th:WLForGraphons}.
To generalize the oblivious $k$-WL to graphons,
we first define the standard Borel space $\Mk$,
which may not be confused with the product of $k$ copies of $\MM$.
$\Mk$ is the $k$-dimensional analogue to $\MM$
and can again be seen as the space of colors used by oblivious $k$-WL.
Its elements $\alpha = (\alpha_0, \alpha_1, \alpha_2, \dots)$ are
sequences, which can be viewed as the colors
assigned to points of a graphon after
$0,1,2,\dots$ refinement rounds.
Based on the definition of oblivious $k$-WL for graphs,
we define the measurable function $\owlk_W \colon X^k \to \Mk$
mapping a $k$-tuple $\bar{x} \in X^k$ to a sequence
$(\alpha_0, \alpha_1, \alpha_2, \dots)$.
In particular, $\alpha_0$ corresponds to the atomic type of a tuple of vertices
and contains the values $W(x_i, x_j)$ for $ij \in \binom{[k]}{2}$.
This further explains the difference between $1$-WL indistinguishability
and fractional isomorphism:
Already the first step of oblivious $2$-WL,
which corresponds to $1$-WL indistinguishability,
completely determines the distribution of values attained
by a graphon.
In contrast, these distributions do not have to be equal for graphons
to be fractionally isomorphic, cf.\
\Cref{sec:graphonsIntroduction} and \Cref{fig:parallelEdgesCounterexample}.

To get an intuition of how the refinement step of oblivious $k$-WL adapts to
graphons, recall how oblivious $k$-WL computes the new color
$\owl^k_{G, n+1}(\bar{v})$ of a tuple $\bar{v} \in V(G)^k$ for a graph $G$
in Equation~(\ref{eq:obliviouskWL}):
$\owl^k_{G, n+1}(\bar{v})$ is a tuple consisting of the old
color $\owl^k_{G, n}(\bar{v})$ of $\bar{v}$ and, for every $j \in [k]$,
the multiset
\begin{equation*}
    \MSet{\owl^k_{G, n+1}(\bar{v}[w/j] \mid w \in V(G))}
\end{equation*}
of colors of all $j$-neighbors of $\bar{v}$.
For a graphon $W$, this multiset becomes the probability measure
\begin{equation*}
    A \mapsto \mu(\{y \in X \mid \owlk_{W,n}(\bar{x}[y/j]) \in A\}),
\end{equation*}
where $A$ is a set of \enquote{colors} used in the $n$th refinement step,
i.e., we determine the mass of the $j$-neighbors of $\bar{x}$
having a color in $A$.
Compiling these probability measures for every $j \in [k]$
into a single tuple then yields the new color $\owlk_{W, n+1}(\bar{x})$ of $\bar{x}$.
The sequence of all these colors $\owlk_{W, n}(\bar{x})$
for $n = 0, 1, 2, \dots$ then yields the mapping $\owlk_W \colon X^k \to \Mk$,
and we define the
\textit{$k$-WL distribution ($k$-WLD)}
$\nu^k_W$ as the push-forward of the product measure $\mu^{\otimes k}$
via $\owlk_W$.
Then, $\nu^k_W$
is a probability measure on $\Mk$
corresponding to the multiset of colors
computed by oblivious $k$-WL on a graph.

The central idea for getting
from fractional isomorphism to $k$-WL indistinguishability
is to replace the single operator $T_W \colon \LTwo \to \LTwo$
of a graphon $W \colon X \times X \to [0,1]$
by a family $\naturalKFamilyOperatorsShortW$ of operators
on the product space $\LTwoProdK \coloneqq \LTwoProdKLong$.
This family $\naturalKFamilyOperatorsShortW$ is
indexed by a set $\adjNeiGraphs^k$ of \textit{bi-labeled graphs}:
a bi-labeled graph $\bm{G}$ is a triple $(G, \bm{a}, \bm{b})$,
where $G$ is a multigraph and
$\bm{a} \in V(G)^k$, $\bm{b} \in V(G)^\ell$ for $k,\ell \ge 0$ are tuples 
of vertices
such that both the entries of $\bm{a}$ and
the entries of $\bm{b}$ are pairwise distinct;
$\bm{a}$ and $\bm{b}$ may however overlap.
The set~$\adjNeiGraphs^k$ is carefully chosen such that
its bi-labeled graphs, together with specific operations,
serve as building blocks
to construct precisely the graphs of treewidth at most $k-1$.
It contains two types of bi-labeled graphs:
\emph{adjacency graphs} and \emph{$j$-neighbor graphs},
where intuitively, adjacency graphs insert an edge into a bag of a tree decomposition
and $j$-neighbor graphs move from one bag
of a tree decomposition to another by replacing a vertex
by a fresh one.
For a simple example,
consider \Cref{fig:fourcycle},
where $k = 3$ and a tree decomposition of the cycle $C_4$
is dissected into a sequence of bi-labeled graphs.
Here, $\adjacencyGraphOf{3}{12}$ and $\adjacencyGraphOf{3}{23}$
are specific instances of adjacency graphs,
while $\neighborGraphOf{3}{2}$ is a $j$-neighbor graph.
By \enquote{gluing} the output vertices of one bi-labeled
graph to the input vertices of the next bi-labeled graph
in the depicted order, we obtain $\mathbf{C}_4$,
a bi-labeled variant of $C_4$ with input labels
on the vertices from the upper bag and
output labels on the vertices from the lower bag:
going bottom up,
$\adjacencyGraphOf{3}{12}$ and $\adjacencyGraphOf{3}{23}$
first insert the edges $v_1 v_4$ and $v_4 v_3$,
then $\neighborGraphOf{3}{2}$ replaces
$v_4$ by $v_2$, and finally,
$\adjacencyGraphOf{3}{12}$ and $\adjacencyGraphOf{3}{23}$
insert the edges $v_1 v_2$ and $v_2 v_3$.

\begin{figure}
    \centering
    \begin{tikzpicture}
        \node[vertex, label={90:$v_2$}] (V2) {};
        \node[vertex, label={180:$v_1$}, below left = 0.7cm and 0.7cm of V2] (V1) {};
        \node[vertex, label={0:$v_3$}, below right = 0.7cm and 0.7cm of V2] (V3) {};
        \node[vertex, label={270:$v_4$}, below right = 0.7cm and 0.7cm of V1] (V4) {};
        \path[draw, thick]
            (V1) edge (V2)
            (V2) edge (V3)
            (V1) edge (V4)
            (V4) edge (V3);

        \node[below = 3.0cm of V2, yshift = -0.1cm] (BAG1) {$\{v_1,\, v_2,\, v_3\}$};
        \node[below = of BAG1] (BAG2) {$\{v_1,\, v_4,\, v_3\}$};
        \path[draw, thick] (BAG1) edge (BAG2);

        \node[vertex, label={90:$a_1$}, label={270:$b_1$}, right = 5cm of V2] (A1) {};
        \node[vertex, label={90:$a_2$}, label={270:$b_2$}, right = of A1] (A2) {};
        \node[vertex, label={90:$a_3$}, label={270:$b_3$}, right = of A2] (A3) {};
        \node[left = 0.3cm of A1, scale = 1.0] (ALabel){$\adjacencyGraphOf{3}{12}\colon$};
        \path[draw, thick] (A1) edge (A2);

        \node[vertex, label={90:$a_1$}, label={270:$b_1$}, below = 1.1cm of A1] (B1) {};
        \node[vertex, label={90:$a_2$}, label={270:$b_2$}, right = of B1] (B2) {};
        \node[vertex, label={90:$a_3$}, label={270:$b_3$}, right = of B2] (B3) {};
        \node[left = 0.3cm of B1, scale = 1.0] (BLabel){$\adjacencyGraphOf{3}{23}\colon$};
        \path[draw, thick] (B2) edge (B3);

        \node[vertex, label={90:$a_1$}, label={270:$b_1$}, below = 1.4cm of B1] (N1) {};
        \node[vertex, fill=white, draw=white, right = of N1] (N2g) {};
        \node[vertex, label={90:$a_2$}, above = 0.1cm of N2g] (N2a) {};
        \node[vertex, label={270:$b_2$}, below = 0.1cm of N2g] (N2b) {};
        \node[vertex, label={90:$a_3$}, label={270:$b_3$}, right = of N2g] (N3) {};
        \node[left = 0.3cm of N1, scale = 1.0] (NLabel){$\neighborGraphOf{3}{2}\colon$};

        \node[vertex, label={90:$a_1$}, label={270:$b_1$}, below = 1.4cm of N1] (C1) {};
        \node[vertex, label={90:$a_2$}, label={270:$b_2$}, right = of C1] (C2) {};
        \node[vertex, label={90:$a_3$}, label={270:$b_3$}, right = of C2] (C3) {};
        \node[left = 0.3cm of C1, scale = 1.0] (CLabel){$\adjacencyGraphOf{3}{12}\colon$};
        \path[draw, thick] (C1) edge (C2);

        \node[vertex, label={90:$a_1$}, label={270:$b_1$}, below = 1.1cm of C1] (D1) {};
        \node[vertex, label={90:$a_2$}, label={270:$b_2$}, right = of D1] (D2) {};
        \node[vertex, label={90:$a_3$}, label={270:$b_3$}, right = of D2] (D3) {};
        \node[left = 0.3cm of D1, scale = 1.0] (DLabel){$\adjacencyGraphOf{3}{23}\colon$};
        \path[draw, thick] (D2) edge (D3);
    \end{tikzpicture}
    \caption{The cycle $C_4$ on four vertices with a tree decomposition.
             On the right, this tree-decomposed graph is written as
             a sequence of bi-labeled graphs.}
    \label{fig:exampleBiLabeledGraphs}
    \label{fig:fourcycle}
\end{figure}

Every bi-labeled graph $\bm{F} \in \adjNeiGraphs^k$
together with a graphon $W \colon X \times X \to [0,1]$ defines
a \textit{graphon operator} $\OperatorFromTo{\bm{F}}{W}$
on $\LTwoProdK$.
Then, $\naturalKFamilyOperatorsShortW
\coloneqq (\OperatorFromTo{\bm{F}}{W})_{\bm{F} \in \adjNeiGraphs^k}$ denotes
the family of all these operators for bi-labeled graphs in $\adjNeiGraphs^k$.
The graphon operators of adjacency graphs just
multiply a given function $f(\bar{x})$ by the value $W(x_i, x_j)$
for fixed $i,j$.
The graphon operator of a $j$-neighbor graph, on the other hand,
averages in the direction $j$, i.e., integrates the given function $f(\bar{x})$
over the $j$th component.
As an example,
the graphon operators of the bi-labeled graphs from
\Cref{fig:exampleBiLabeledGraphs}
for a graphon $W \colon X \times X \to [0,1]$
are given by
$(\OperatorFromTo{\adjacencyGraphOf{3}{12}}{W} f)(x_1, x_2, x_3)
    \coloneqq W(x_1, x_2) \cdot f(x_1, x_2, x_3)$,
$(\OperatorFromTo{\adjacencyGraphOf{3}{23}}{W} f)(x_1, x_2, x_3)
    \coloneqq W(x_2, x_3) \cdot f(x_1, x_2, x_3)$, and
\begin{align*}
(\OperatorFromTo{\neighborGraphOf{3}{2}}{W} f)(x_1, x_2, x_3)
        &\coloneqq \int_X f(x_1, y, x_3) \,d\mu(y)
\end{align*}
for all $x_1, x_2, x_3 \in X$.
At this point, the reader might already note the connection
of these graphon operators to the description
of oblivious $k$-WL for graphons given above as, intuitively,
the graphon operators of adjacency graphs are used in the initial coloring
and the graphon operators of $j$-neighbors are used in the refinement steps.

We will see that the composition of the graphon operators
corresponding to the sequence of bi-labeled graphs in \Cref{fig:exampleBiLabeledGraphs},
\begin{equation*}
    \OperatorFromTo{\adjacencyGraphOf{3}{12}}{W}
        \circ \OperatorFromTo{\adjacencyGraphOf{3}{23}}{W}
        \circ \OperatorFromTo{\neighborGraphOf{3}{2}}{W}
        \circ\OperatorFromTo{\adjacencyGraphOf{3}{12}}{W}
        \circ \OperatorFromTo{\adjacencyGraphOf{3}{23}}{W},
\end{equation*}
is precisely the graphon operator
$\OperatorFromTo{\mathbf{C}_4}{W}$ of $\mathbf{C}_4$,
i.e., the bi-labeled variant of
$C_4$ obtained by gluing together this sequence
of bi-labeled graphs.
Furthermore, one can verify that
\begin{equation*}
    \int_{X^3}
    \OperatorFromTo{\mathbf{C}_4}{W} \bm{1}_{X^3} \,d\mu^{\otimes 3}
    = t(C_4, W),
\end{equation*}
where $\bm{1}_{X^3}$ is the all-one function on $X^3$, i.e.,
the homomorphism density of $C_4$ in $W$ is determined by
the operator $\OperatorFromTo{\mathbf{C}_4}{W}$.
In greater generality, the homomorphism density of a bi-labeled graph in a graphon
can always be expressed via the graphon operator.

Since we are now working with operators on $\LTwoProdKLong$,
we consider $\mu^{\otimes k}$-relatively complete sub-$\sigma$-algebra
of $\Borel^{\otimes k}$ for an analogue to partitions of $V(G)^k$ for some graph $G$.
For a graphon $W \colon X \times X \to [0,1]$,
a $\mu^{\otimes k}$-relatively complete sub-$\sigma$-algebra
$\subAlg \in \subAlgsk$ of $\Borel^{\otimes k}$
is called \emph{$W$-invariant}
if it is \textit{$\naturalKFamilyOperatorsShortW$-invariant}, i.e.,
$T$-invariant for every operator $T$ in the family
$\naturalKFamilyOperatorsShortW$.
In the case $k = 1$, this conflicts with the earlier
definition of Greb\'ik and Rocha, but
it will always be clear from the context what we mean.
We show that there is a
minimum $W$-invariant $\mu^{\otimes k}$-relatively complete sub-$\sigma$-algebra
$\subAlg^k_W$ of $\Borel^{\otimes k}$.
The partitions of $V(G)^k$ induced by the colors of oblivious $k$-WL
are invariant under permutations, i.e., reordering the vertices of a tuple
yields a tuple in the same class.
Similarly, we show that $\subAlg^k_W$ is \emph{permutation invariant}
and also define the notion of
\textit{permutation-invariant} operators, i.e.,
operators where
a reordering of the $k$ components of $X^k$ yields the same operator.
This reflects the fact that,
in the system $\Lkk$ of linear equations characterizing
oblivious $k$-WL,
variables are indexed by sets and not by tuples.

We cannot give a meaningful definition of the
quotient of a graphon $W \colon X \times X \to [0,1]$ w.r.t.\
a $\mu^{\otimes k}$-relatively complete sub-$\sigma$-algebra $\subAlg \in \subAlgsk$
if $k \ge 2$.
Instead, we consider quotients of operators.
Intuitively,
for an operator $T$ on $\LTwoProdK$,
its \emph{quotient operator w.r.t.\ $\subAlg$}
on $\LTwoProdKQuo$, denoted by $T/\subAlg$,
is defined
by going from $\LTwoProdKQuo$ to $\LTwoProdK$, applying $T$,
and then going back to $\LTwoProdKQuo$.
Again, a different but equivalent definition
can also be given via conditional expectations
by letting $T_\subAlg \coloneqq \ExpValSub \circ T \circ \ExpValSub$.
Then, we can consider the families $\naturalKFamilyOperatorsShortW / \subAlg$
and $(\naturalKFamilyOperatorsShortW)_\subAlg$ of quotient operators
w.r.t.\ $\subAlg$.

We now state our main theorem,
\Cref{th:WLForGraphons}.
As mentioned before, it is based on oblivious $k$-WL,
so there is a mismatch between the $k$ in the treewidth,
i.e., the $k$ in $k$-WL indistinguishability,
and the other characterizations.
We note that, since
there are no quotient graphons involved in \Cref{th:kWLGraphons},
we also do not obtain
a canonical representation of a graphon $W \colon X \times X \to [0,1]$
as a graphon $\Mk \times \Mk \to [0,1]$ (or as multiple such graphons).
Instead, we define canonical representations of the operators in
$\naturalKFamilyOperatorsShortW$ on the space $\LTwoMkNuKW$ by hand.

\begin{restatable}{theorem}{maintheorem}
    \label{th:WLForGraphons}
    \label{th:kWLGraphons}
    Let $k \ge 1$
    and $U, W \colon X \times X \to [0,1]$ be graphons.
    The following are equivalent:
    \begin{enumerate}
        \item
            $t(F, U) = t(F, W)$ for every multigraph of treewidth at most $k-1$.
            \label{th:WLForGraphons:homomorphisms}
        \item $\nu^k_U = \nu^k_W$.
            \label{th:WLForGraphons:DIDM}
        \item
            There is a (permutation-inv.) Markov iso.\
            $R \colon \LTwoProdKQuoOf{\kSubAlg_W} \to \LTwoProdKQuoOf{\kSubAlg_U}$
            such that
            $\naturalKFamilyOperatorsShortU/\kSubAlg_U\circ R= R \circ \naturalKFamilyOperatorsShortW/\kSubAlg_W$.
            \label{th:WLForGraphons:minInvSubalgebra}
        \item
            There is a (permutation-inv.) Markov operator
            $S \colon \LTwoProdK \to \LTwoProdK$
            such that
            $\naturalKFamilyOperatorsShortU \circ S = S \circ \naturalKFamilyOperatorsShortW$.
            \label{th:WLForGraphons:markovOperator}
        \item
            There are
            $\mu^{\otimes k}$-relatively complete sub-$\sigma$-algebras $\subAlgC$ and $\subAlgD$
            of $\Borel^{\otimes k}$ that are $U$-invariant and $W$-invariant, respectively,
            and a Markov iso.\
            $R \colon \LTwoProdKQuoD \to \LTwoProdKQuo$
            such that
            $\naturalKFamilyOperatorsShortU/\subAlg \circ R= R \circ \naturalKFamilyOperatorsShortW/\subAlgD$.
            \label{th:WLForGraphons:markovIsomorphism}
    \end{enumerate}
\end{restatable}

The characterizations of \Cref{th:WLForGraphons}
are listed in the same order as
these in \Cref{th:colRefGraphons}.
Recall the various
notions characterizing $k$-WL indistinguishability of graphs
presented in \Cref{sec:finiteGraphs}.
Characterization~$(\ref{th:WLForGraphons:homomorphisms})$
corresponds to homomorphism numbers of graphs
of treewidth at most $k-1$ and is the definition of
$(k-1)$-WL indistinguishability of graphons.
We note that,
as in the case of simple graphs,
one could always assume the multigraphs in
Characterization (\ref{th:WLForGraphons:homomorphisms})
to be connected, cf.\ \cite[$(7.6)$]{Lovasz2012}.
For example, in the case $k = 2$, it
could equivalently be phrased in terms of
homomorphism densities of trees with
parallel edges.
Characterization~$(\ref{th:WLForGraphons:DIDM})$
states that the $k$-WLDs of the graphons are the same and
corresponds to oblivious $k$-WL not distinguishing two graphs.
Characterization~$(\ref{th:WLForGraphons:minInvSubalgebra})$
--$(\ref{th:WLForGraphons:markovIsomorphism})$
look very similar, but have a different focus:
Characterization~$(\ref{th:WLForGraphons:markovOperator})$
generalizes (non-negative real) solutions to the
system $\Lkk(G, H)$ of linear equations
by stating
that there is a Markov operator on the product space $\LTwoProdK$
that intertwines all operators in the families
$\naturalKFamilyOperatorsShortU$ and
$\naturalKFamilyOperatorsShortW$ simultaneously.
Permutation invariance can be left out without changing the equivalence
to the other characterizations, i.e.,
if there is a (not necessarily permutation-invariant)
Markov operator $S$ satisfying Characterization~(\ref{th:WLForGraphons:markovOperator}),
then there also is a permutation-invariant one.
Characterization~$(\ref{th:WLForGraphons:minInvSubalgebra})$ and
$(\ref{th:WLForGraphons:markovIsomorphism})$ on the other hand,
correspond to the (coarsest) stable partitions of vertex-tuples of two graphs
having the same parameters.
We note that there is a one-to-one correspondence between Markov isomorphisms
and measure-preserving almost bijections,
cf.\ \cite[Theorem E.$3$]{GrebikRocha2021}, but for the ease of presentation,
we stick to Markov isomorphisms.

\subsection{Overview and Further Remarks}
\label{sec:overview}
\label{sec:conclusions}

In \Cref{sec:preliminaries}, the preliminaries, we collect
some basics we need: we briefly visit
product spaces, Markov operators, and quotient spaces
before defining quotient operators.
\Cref{sec:graphons} formally introduces bi-labeled graphs
and graphon operators,
which are the key to both stating and proving \Cref{th:WLForGraphons}.
In particular, we define the set $\adjNeiGraphs^k$
of bi-labeled graphs from which we
are able to construct precisely the multigraphs of treewidth $k-1$,
and then, for a graphon $W$,
the family of graphon operators $\naturalKFamilyOperatorsShortW$.
\Cref{sec:WL} is the main section of this paper containing
the formal definitions of all notions in \Cref{th:kWLGraphons}
and, of course, its proof, for which we follow the
structure of Greb\'ik and Rocha \cite{GrebikRocha2021}.
Let us give a brief overview of the proof and how
the set $\adjNeiGraphs^k$ and the corresponding family of graphon operators
$\naturalKFamilyOperatorsShortW$ are used in it:
\begin{itemize}
    \item (\ref{th:WLForGraphons:homomorphisms}) $\implies$ (\ref{th:WLForGraphons:DIDM}):
        We use the set $\adjNeiGraphs^k$ of bi-labeled graphs
        to construct expressions that correspond to precisely the tree-decomposed graphs
        of treewidth at most $k-1$ (\Cref{sec:bilabeledgraphs}).
        This allows us to define a
        set $\Tk \subseteq C(\Mk, \R)$ of functions on $\Mk$ that corresponds
        to homomorphism densities of these graphs (\Cref{sec:WLIndistinguishability}).
        The Stone-Weierstrass Theorem yields that $\Tk$ is dense in $C(\Mk, \R)$,
        which implies that $k$-WLDs are determined by homomorphism densities.
    \item (\ref{th:WLForGraphons:DIDM}) $\implies$ (\ref{th:WLForGraphons:minInvSubalgebra}):
        We show that a $k$-WLD $\nu$ defines a family $\TT_\nu$
        of operators on $\LTwoMkNu$ (\Cref{sec:operatorsOnMeasures}).
        Then, we proceed to show that, for every graphon $W$, the family $\TT_{\nu^k_W}$
        is in some sense isomorphic to the family
        $\naturalKFamilyOperatorsShortW/\kSubAlg_W$
        of quotient operators of $\naturalKFamilyOperatorsShortW$.
    \item (\ref{th:WLForGraphons:minInvSubalgebra}) $\implies$ (\ref{th:WLForGraphons:markovOperator}):
        We combine facts we establish on quotient operators (\Cref{sec:quotientOperators})
        together with the fact that $\kSubAlg_U$ and $\kSubAlg_W$
        are $\naturalKFamilyOperatorsShortU$-
        and $\naturalKFamilyOperatorsShortW$-invariant, respectively.
    \item (\ref{th:WLForGraphons:markovOperator}) $\implies$ (\ref{th:WLForGraphons:markovIsomorphism}):
        This is a refined variant of the argument by
        \citeauthor{GrebikRocha2021} in \cite{GrebikRocha2021},
        which uses the Mean Ergodic Theorem
        for Hilbert spaces to Markov operators
        \cite[Theorem $8.6$, Example $13.24$]{EisnerEtAl2015}.
        We have condensed this argument into a standalone lemma, \Cref{th:markovOperatorToIsomorphism}.
    \item (\ref{th:WLForGraphons:markovIsomorphism}) $\implies$ (\ref{th:WLForGraphons:homomorphisms}):
        We again use the fact that the set $\adjNeiGraphs^k$ of bi-labeled graphs
        allows to construct expressions that correspond to precisely the tree-decomposed graphs
        of treewidth at most $k-1$ (\Cref{sec:bilabeledgraphs}).
        We use this to show that homomorphism densities of
        graphs of treewidth at most $k-1$ in a graphon
        $W$ can be expressed purely by expressions built from the operators in
        $\naturalKFamilyOperatorsShortW$ (\Cref{sec:graphonOperators}).
        Then, we use that Markov embeddings
        are compatible with point-wise products of functions,
        which for us means that intertwining Markov embeddings
        preserve homomorphism densities (\Cref{le:markovEmbeddingPreservesHomomorphisms}).
\end{itemize}
In \Cref{sec:simpleWL}, we show how
\Cref{th:kWLGraphons} can be modified to obtain
a characterization of simple $k$-WL indistinguishability, i.e.,
indistinguishability w.r.t.\ homomorphism densities
of simple graphs of treewidth at most $k$
instead of multigraphs.
However, the corresponding analogue to \Cref{th:WLForGraphons} obtained this way
is less elegant and has an artificial touch to it.
The reason for this is that the set of bi-labeled graphs one uses
instead of $\adjNeiGraphs^k$
is not closed under \emph{transposition}, which implies that
the corresponding family of operators is not closed under taking
Hilbert adjoints.
Most of the proofs in \Cref{sec:simpleWL}
are left out as they are mostly analogous to the ones
in \Cref{sec:WL}.

The original goal of this work was to define a \textit{$k$-WL distance}
of graphons
and to prove that it yields the same topology as treewidth-$k$
homomorphism densities, cf.\ \cite{Boeker2021},
where the result of Greb\'ik and Rocha is used to
prove such a result for the \textit{tree distance},
which is based on the characterization of fractional isomorphism
via Markov operators.
However, the approach in~\cite{Boeker2021} does not go well together
with \Cref{th:kWLGraphons} as
multigraph homomorphism densities define a non-compact topology
that is different
from the one obtained by the cut distance, cf.\
\cite[Exercise $10.26$]{Lovasz2012} or \cite[Lemma C.$2$]{Janson2013}.
Moreover, the characterization of simple $k$-WL indistinguishability
via Markov operators is also not well-suited for this
as the corresponding family of operators is not closed
under Hilbert adjoints.
Hence, it remains an open problem to define such a distance.

A different open problem is given by
the contemporaneous work of \citeauthor{grohe_homomorphism_2022}
\cite{grohe_homomorphism_2022}:
They do not focus on a specific set of bi-labeled graphs,
but use bi-labeled graphs to give a unified framework
to characterize graphs in terms of homomorphism numbers.
In particular, they obtain a characterization
of homomorphism numbers from graphs of bounded pathwidth.
It would be interesting to see if their framework, or at least
their work on graphs of bounded pathwidth, generalizes to graphons.

\section{Preliminaries}
\label{sec:preliminaries}

In this section, we briefly collect some
facts that we use throughout the paper.
\Cref{sec:standardBorelSpaces} concerns product spaces;
for a more complete reference,
we refer to \cite{Dudley2002, Billingsley1995}.
The definitions and results regarding
Markov operators in \Cref{sec:markovOperators} are taken from
\cite{EisnerEtAl2015}.
The treatment of quotient spaces in \Cref{subsec:algebras} is based on that of
\citeauthor{GrebikRocha2021} \cite{GrebikRocha2021}.
We then use these quotient spaces to define \textit{quotient operators}
in \Cref{sec:quotientOperators}.

\subsection{Product Spaces}
\label{sec:standardBorelSpaces}

Recall that, throughout the whole paper,
$(X, \Borel)$ denotes a \textit{standard Borel space},
i.e., $\Borel$ is the Borel $\sigma$-algebra of a Polish space,
and $\mu$ a \textit{Borel probability measure} on $X$.
We often consider the space $(X^k, \Borel^{\otimes k}, \mu^{\otimes k})$
with the product $\sigma$-algebra $\Borel^{\otimes k}$ of $\Borel$
and the product measure $\mu^{\otimes k}$ of $\mu$
for $k \ge 1$.
The product of a countable family of standard Borel spaces is again
a standard Borel space \cite[Section $12$.B]{Kechris1995}.
Moreover, for a countable family of standard Borel spaces,
its product $\sigma$-algebra is actually equal
to the Borel $\sigma$-algebra of the product topology
of the underlying Polish spaces
as Polish spaces are second countable \cite[Section $11$.A]{Kechris1995}.
Hence, the product space $(X^k, \Borel^{\otimes k})$
is again a standard Borel space and $\Borel^{\otimes k}$
is equal to the Borel $\sigma$-algebra
of the product topology of the Polish space underlying $(X, \Borel)$.
For simplicity, we identify the products $X \times X \times X$
and $(X \times X) \times X$ in the usual way.
Then, also $\Borel \otimes \Borel \otimes \Borel = (\Borel \otimes \Borel) \otimes \Borel$
and $\mu \otimes \mu \otimes \mu = (\mu \otimes \mu) \otimes \mu$
\cite[Section 18]{Billingsley1995}.
We treat higher-order products in the same way.

We often use the Tonelli-Fubini theorem,
cf.\ \cite[Theorem $4.4.5$]{Dudley2002} and also \cite[Theorem $18.3$]{Billingsley1995},
which states that, for $\sigma$-finite measure spaces
$(X, \mathcal{S}, \mu)$ and $(Y, \mathcal{T}, \nu)$
and a non-negative function $f$ on $X \times Y$ that
is measurable for $\mathcal{S} \otimes \mathcal{T}$,
we have
\begin{equation*}
    \int_{X \times Y} f \, d (\mu \times \nu)
    = \int_{X} \int_{Y} f(x, y) \, d \nu(y) \, d \mu(x)
    = \int_{Y} \int_{X} f(x, y) \, d \mu(x) \, d \nu(y).
\end{equation*}
In particular, the functions $x \mapsto \int_Y f(x,y) \, d \nu(y)$
and $y \mapsto \int_X f(x,y) \, d \mu(x)$ are measurable for $\mathcal{S}$
and $\mathcal{T}$, respectively.
If $f$ is not necessarily non-negative but integrable with respect to $\mu \times \nu$,
then the same equations hold and the aforementioned functions
are measurable on sets $X'$ and $Y'$ with $\mu(X \setminus X') = 0$
and $\nu(Y \setminus Y') = 0$, respectively.

\subsection{Markov Operators}
\label{sec:markovOperators}

In general,
for a measure space $(X, \mathcal{S}, \mu)$ and $1 \le p \le \infty$,
the space $\mathcal{L}^p(X, \mu) \coloneqq \mathcal{L}^p(X, \mathcal{S}, \mu)$
consists of all measurable real-valued functions on $X$ with
$\lVert f \rVert_p < \infty$,
and $L^p(X, \mu) \coloneqq L^p(X, \mathcal{S}, \mu)$
is obtained from $\mathcal{L}^p(X, \mu)$
by identifying functions that are equal
$\mu$-almost everywhere.
The space $L^2(X, \mu)$ plays a special role among these spaces
as it is a Hilbert space with the inner product
given by $\langle f, g \rangle \coloneqq \int_X f g \dmu$.
Besides $L^2(X, \mu)$, the space $L^\infty(X, \mu)$ also plays an important
role in this paper.
Note that, if $\mu$ is a probability measure,
then we have
$\lVert f \rVert_2 \le \lVert f \rVert_\infty$
and, in particular,
the inclusion $L^\infty(X, \mu) \subseteq L^2(X, \mu)$ holds.

Given two normed linear spaces $(X, \lVert \cdot \rVert)$ and $(Y, \lvert \cdot \rvert)$,
a function $T \colon X \to Y$
is called
a (bounded linear) \textit{operator} if it is Lipschitz and linear.
If $(X, \lVert \cdot \rVert) = (Y, \lvert \cdot \rvert)$,
then we just say that $T$ is an \textit{operator on $X$.}
The \textit{operator norm of $T$}
is given by $\lVert T \rVert \coloneqq \sup \{\lvert T(x) \rvert \mid \lVert x \lVert \le 1\} < \infty$,
and if $\lVert T \rVert \le 1$, then $T$ is called a \textit{contraction}.
For probability spaces $(X, \mathcal{S}, \mu)$ and $(Y, \mathcal{T}, \nu)$
and an operator $T \colon L^2(X, \mu) \to L^2(Y, \nu)$,
we call $T$ an $L^\infty$-contraction
if its restriction to $L^\infty(X, \mu)$
yields a well-defined contraction $L^\infty(X, \mu) \to L^\infty(Y, \nu)$.
To clearly distinguish this from $T$ being a contraction
$L^2(X, \mu) \to L^2(Y, \nu)$, we sometimes use the term $L^2$-contraction
for this.
Observe that the composition of two contractions yields a contraction,
and in particular, the composition of $L^2$- and $L^\infty$-
contractions yields a $L^2$- and a $L^\infty$-contraction, respectively.

For measure spaces $(X, \mathcal{S}, \mu)$ and $(Y, \mathcal{T}, \nu)$,
the \textit{Hilbert adjoint}
of an operator $T \colon L^2(X, \mu) \to L^2(Y, \nu)$
is the unique operator
$T^* \colon L^2(Y, \nu) \to L^2(X, \mu)$
satisfying $\langle T f, g \rangle = \langle f, T^* g \rangle$
for all $f \in L^2(X, \mu), g \in L^2(Y, \nu)$.
For standard Borel spaces $(X, \Borel)$ and $(Y, \mathcal{D})$
with Borel probability measures $\mu$ and $\nu$ on $X$ and $Y$, respectively,
an operator $S \colon \LTwo \to \LTwoY$ is called a \textit{Markov operator}
if $S f \ge 0$ for every $f \in \LTwo$ with $f \ge 0$, $S \allOne_X = \allOne_Y$,
and $S^* \allOne_Y = \allOne_X$.
Markov operators are both $L^2$- and $L^\infty$-contractions
\cite[Theorem $13.2$ b)]{EisnerEtAl2015}.
A Markov operator is called a \textit{Markov embedding} if it is an isometry.
For example, the \textit{Koopman operator}
$T_\varphi \colon \LTwo \to \LTwo$
of a measure-preserving measurable map $\varphi \colon X \to X$,
defined by
$T_\varphi f \coloneqq f \circ \varphi$
for every $f \in \LTwo$,
is a Markov embedding \cite[Example $13.1$]{EisnerEtAl2015}.
A \textit{Markov isomorphism} is a surjective Markov embedding.
Note that every Markov isomorphism $S$
satisfies $S^{-1} = S^*$ \cite[Corollary $13.14$]{EisnerEtAl2015}.
Moreover, there is a one-to-one correspondence between Markov isomorphisms
and measure-preserving almost bijections,
cf.\ \cite[Theorem E.$3$]{GrebikRocha2021}.
See \cite{EisnerEtAl2015} for a thorough treatment of Markov operators.
There, the results are stated for complex $L^p$-spaces,
but this usually does not make a difference
by the positivity of Markov operators, cf.\ \cite[Lemma $7.5$]{EisnerEtAl2015}.

\subsection{Quotient Spaces}
\label{subsec:algebras}

Recall that
a sub-$\sigma$-algebra $\subAlg \subseteq \Borel$ of $\Borel$
is called \textit{$\mu$-relatively complete}
if $Z \in \subAlg$
for all $Z \in \Borel, Z_0 \in \subAlg$ with $\mu(Z \triangle Z_0) = 0$.
Requiring
$Z \in \subAlg$ for every $Z \in \Borel$ with $\mu(Z) = 0$
instead would yield an equivalent definition.
The set of all $\mu$-relatively complete sub-$\sigma$-algebras
of $\Borel$ is denoted by $\subAlgs$
and clearly includes $\Borel$ itself.
For a non-empty $\Phi \subseteq \subAlgs$, we have
$\bigcap \Phi \coloneqq \bigcap_{\subAlg \in \Phi} \subAlg \in \subAlgs$
\cite[Claim $5.4$]{GrebikRocha2021}.
Hence, for a set $\mathcal{X} \subseteq \Borel$, there is a smallest
$\mu$-relatively complete sub-$\sigma$-algebra including $\mathcal{X}$,
which we denote by $\langle \mathcal{X} \rangle$.
If $\subAlg \subseteq \Borel$ is a sub-$\sigma$-algebra,
then one can show that
$\langle \subAlg \rangle = \Set{A \triangle Z \mid A \in \subAlg, Z \in \Borel \text{ with } \mu(Z) = 0}$.
Given $\subAlg \in \subAlgs$, we let $\LTwoSub \subseteq \LTwo$
denote the subset of all functions that are $\subAlg$-measurable.
It is a standard fact that, for $\subAlg \in \subAlgs$,
the linear hull of $\Set{\allOne_A}_{A \in \mathcal{C}}$
is dense in $\LTwoSub$.
The \emph{conditional expectation}
$\ExpVal{\subAlg}$ is the orthogonal projection onto the closed
linear subspace $\LTwoSub$ of $\LTwo$.

\begin{proposition}[{Conditional Expectation, \cite[Section $34$]{Billingsley1995}}]
    \label{cl:conditionalExpectation}
    Let $\subAlg \in \subAlgs$.
    Then, $\LTwoSub$ is a closed linear subspace of $\LTwo$ and
    there is a self-adjoint operator
    $\ExpVal{\subAlg} \colon \LTwo \to \LTwo$
    such that
    \begin{enumerate}
        \item $\ExpVal{\subAlg}$ is the orthogonal projection onto $\LTwoSub$,
        \item $\int_A f \dmu = \int_A  \E(f \mid \subAlg) \dmu$ for every $A \in \subAlg$ and every $f \in \LTwo$, and
        \item $\int_X f \cdot \E(g \mid \subAlg) \dmu = \int_X \E(f \mid \subAlg) \cdot g \dmu$
            for all $f,g \in \LTwo$.
    \end{enumerate}
\end{proposition}

Given a measure space $(X, \mathcal{S}, \mu)$, a measurable space $(Y, \mathcal{T})$,
and a measurable function $g \colon X \to Y$,
the \textit{push-forward} $g_* \mu$ is the measure on $Y$ defined
by $g_* \mu (A) \coloneqq \mu(g^{-1}(A))$ for every $A \in \mathcal{T}$.
For a measurable function $f \colon Y \to [-\infty,\infty]$, we then have
$\int_Y f \, d (g_* \mu) = \int_X f \circ g \, d \mu$ \cite[Theorem $4.1.11$]{Dudley2002}.
The following proposition then guarantees the existence of a quotient space
of $(X, \Borel, \mu)$ w.r.t.\
a $\mu$-relatively complete sub-$\sigma$-algebra $\subAlg \in \subAlgs$.

\begin{proposition}[{\cite[Theorem E.$1$]{GrebikRocha2021}}]
    \label{th:quotientSpaces}
    Let $\subAlg \in \subAlgs$.
    There is a standard Borel space $(X/\subAlg, \subAlg')$,
    a Borel probability measure $\mu/\subAlg$ on $X/\subAlg$,
    a measurable surjection $q_\subAlg \colon X \to X/\subAlg$,
    and Markov operators
    $S_\subAlg \colon \LTwo \to \LTwoQuo$
    and
    $I_\subAlg \colon \LTwoQuo \to \LTwo$
    such that
    \begin{multicols}{2}
    \begin{enumerate}
        \itemsep0em
        \item $I_\subAlg$ is the Koopman operator of $q_\subAlg$,
        \item $\mu/\subAlg$ is the push-forward of $\mu$ via $q_\subAlg$,
        \item $S_\subAlg^* = I_\subAlg$,\label{th:quotientSpaces:adjoint}
        \item $S_\subAlg \circ \ExpVal{\subAlg} = S_\subAlg$,\label{th:quotientSpaces:SExpVal}
        \item $I_\subAlg$ is an isometry onto $\LTwoSub$, \label{th:quotientSpaces:isometry}
        \item $I_\subAlg \circ S_\subAlg = \ExpVal{\subAlg}$, and\label{th:quotientSpaces:ExpVal}
        \item $S_\subAlg \circ I_\subAlg$ is the identity on $\LTwoQuo$.
    \end{enumerate}
    \end{multicols}
\end{proposition}

\Cref{co:quotientSpaceUnique} is a technical result
that intuitively states that
the quotient space $(X/\subAlg, \subAlg')$ is unique
and the same as $\LTwoSub$
up to sets of measure zero.
\begin{proposition}[{\cite[Corollary E.$2$]{GrebikRocha2021}}]
    \label{co:quotientSpaceUnique}
    Let $(X, \Borel)$ and $(Y, \mathcal{D})$ be standard Borel spaces.
    Let $\mu$ be a Borel probability measure on $X$
    and $f \colon X \to Y$ be a measurable function.
    Let $\subAlg \in \subAlgs$
    be the minimum $\mu$-relatively complete sub-$\sigma$-algebra
    that makes $f$ measurable.
    Then, for every $g_0 \in \LTwoSub$, there is a measurable map
    $g_1 \colon Y \to \R$ such that $g_0(x) = (g_1 \circ f)(x)$
    for $\mu$-almost every $x \in X$.
\end{proposition}

\subsection{Quotient Operators}
\label{sec:quotientOperators}

For $\subAlg \in \subAlgs$
and an operator $T \colon \LTwo \to \LTwo$,
we use the conditional expectation to define
the operators $T_\subAlg \colon \LTwo \to \LTwo$
and
$T/\subAlg \colon \LTwoQuo \to \LTwoQuo$ by
\begin{align*}
    &T_\subAlg \coloneqq \ExpVal{\subAlg} \circ T \circ \ExpVal{\subAlg}&
    &\text{and}&
    &T/\subAlg \coloneqq S_\subAlg \circ T \circ I_\subAlg,&
\end{align*}
respectively.
These definitions reflect the same concept of a quotient operator
via different languages.
The following lemma states some basic properties and shows how
both definitions are related.

\begin{lemma}
    \label{le:expOperator}
    \label{le:quoOperator}
    Let $\subAlg \in \subAlgs$
    and $T \colon \LTwo \to \LTwo$ be an operator.
    Then,
    \begin{enumerate}
        \itemsep0em
        \item
            $(T_\subAlg)^* = (T^*)_\subAlg$ and
            $(T / \subAlg)^* = T^* / \subAlg$,
            \label{le:expOperator:adjoint}
            \label{le:quoOperator:adjoint}
        \item
            if $T$ is self-adjoint, then so are $T_\subAlg$ and $T/\subAlg$,
            \label{le:expOperator:selfAdjoint}
            \label{le:quoOperator:selfAdjoint}
        \item $I_\subAlg \circ T / \subAlg = T_\subAlg \circ I_\subAlg$, \label{le:quoOperator:expOperatorI}
        \item $T/\subAlg \circ S_\subAlg = S_\subAlg \circ T_\subAlg$,\label{le:quoOperator:expOperator}
        \item
            if $\subAlg$ is $T$-invariant,
            then $T_\subAlg = T \circ \ExpValSub$
            and
            $I_\subAlg \circ T / \subAlg = T \circ I_\subAlg$, and
            \label{le:expOperator:invariant}
            \label{le:quoOperator:invariantI}
        \item
            if $T$ is self-adjoint and $\subAlg$ is $T$-invariant,
            then $T/\subAlg \circ S_\subAlg = S_\subAlg \circ T$.\label{le:quoOperator:invariant}
    \end{enumerate}
\end{lemma}
\begin{proof}
    For (\ref{le:expOperator:adjoint}),
    we have
    $(T_\subAlg)^*= \ExpVal{\subAlg}^* \circ T^* \circ \ExpVal{\subAlg}^*= \ExpVal{\subAlg} \circ T^* \circ \ExpVal{\subAlg}= (T^*)_\subAlg$
    by \Cref{cl:conditionalExpectation}
    and
    $(T/\subAlg)^*= I_\subAlg^* \circ T^* \circ S_\subAlg^*= S_\subAlg \circ T^* \circ I_\subAlg = T^* /\subAlg$
    by (\ref{th:quotientSpaces:adjoint}) of \Cref{th:quotientSpaces}.
    This also immediately yields (\ref{le:expOperator:selfAdjoint}).
    For (\ref{le:quoOperator:expOperatorI}), we have
    \begin{equation*}
        I_\subAlg \circ T / \subAlg = I_\subAlg \circ S_\subAlg \circ T \circ I_\subAlg = \ExpVal{\subAlg} \circ T \circ I_\subAlg = \ExpVal{\subAlg} \circ T \circ \ExpVal{\subAlg} \circ I_\subAlg
        = T_\subAlg \circ I_\subAlg
    \end{equation*}
    by (\ref{th:quotientSpaces:ExpVal}) and (\ref{th:quotientSpaces:SExpVal}) of \Cref{th:quotientSpaces} and \Cref{cl:conditionalExpectation}.
    For (\ref{le:quoOperator:expOperator}), we have
    \begin{equation*}
        T/\subAlg \circ S_\subAlg = S_\subAlg \circ T \circ I_\subAlg \circ S_\subAlg = S_\subAlg \circ \ExpVal{\subAlg} \circ T \circ \ExpVal{\subAlg}= S_\subAlg \circ T_\subAlg
    \end{equation*}
    by (\ref{th:quotientSpaces:SExpVal}) and (\ref{th:quotientSpaces:ExpVal}) of \Cref{th:quotientSpaces}.
    For (\ref{le:expOperator:invariant}),
    assume that $\subAlg$ is $T$-invariant.
    By \Cref{cl:conditionalExpectation},
    the expectation $\ExpValSub$ is the orthogonal projection
    onto $\LTwoSub$.
    Hence,
    $(T \circ \ExpValSub)(\LTwo) = T(\LTwoSub) \subseteq \LTwoSub$
    and, as $\ExpValSub$ is the identity on $\LTwoSub$,
    the first claim $T_\subAlg = T \circ \ExpValSub$ follows.
    Then, continuing with (\ref{le:quoOperator:expOperatorI}), we get
    $I_\subAlg \circ T / \subAlg = T_\subAlg \circ I_\subAlg = T \circ \ExpVal{\subAlg} \circ I_\subAlg = T \circ I_\subAlg$
    by
    (\ref{th:quotientSpaces:SExpVal}) of \Cref{th:quotientSpaces}
    and \Cref{cl:conditionalExpectation}.
    Now, (\ref{le:quoOperator:invariant})
    follows from (\ref{le:quoOperator:selfAdjoint}),
    (\ref{le:quoOperator:invariantI}),
    and (\ref{th:quotientSpaces:adjoint}) of \Cref{th:quotientSpaces}.
\end{proof}

The following lemma is an application of the Mean Ergodic Theorem
for Hilbert spaces to Markov operators
\cite[Theorem $8.6$, Example $13.24$]{EisnerEtAl2015}
and is the essence of the proof
of the direction
\enquote{(\ref{th:colRefGraphons:markov})
$\implies$
(\ref{th:colRefGraphons:invariant})}
of \Cref{th:colRefGraphons}
by Greb\'ik and Rocha \cite{GrebikRocha2021}.
\begin{lemma}
    \label{th:markovOperatorToIsomorphism}
    Let $S \colon \LTwo \to \LTwo$ be a Markov operator.
    There are $\subAlgC, \subAlgD \in \subAlgs$
    with
    \begin{enumerate}
        \itemsep0em
        \item $\LTwoSub = \{f \in \LTwo \mid (S \circ S^*) f = f\}$,\label{th:markovOperatorToIsomorphism:subAlgC}
        \item $\LTwoSubD = \{f \in \LTwo \mid (S^* \circ S) f = f\}$,
        \item $\ExpValSub \circ S = S \circ \ExpVal{\subAlgD}$,\label{th:markovOperatorToIsomorphism:commutesExpVal}
\item $R \coloneqq S_\subAlg \circ S \circ I_\subAlgD \colon \LTwoQuoD \to \LTwoQuo$ is a Markov isomorphism, and\label{th:markovOperatorToIsomorphism:isomorphism}
\item
            for all operators $T_1, T_2 \colon \LTwo \to \LTwo$
            with $T_1 \circ S = S \circ T_2$ and $S^* \circ T_1 = T_2 \circ S^*$,
\begin{enumerate}
                \vspace{-3pt}
                \itemsep0em
                \item $\subAlg$ is $T_1$-invariant,\label{th:markovOperatorToIsomorphism:T1Inv}
                \item $\subAlgD$ is $T_2$-invariant, and\label{th:markovOperatorToIsomorphism:T2Inv}
                \item $T_1/\subAlg \circ R = R \circ T_2/\subAlgD$.\label{th:markovOperatorToIsomorphism:RCommutes}
            \end{enumerate}\label{th:markovOperatorToIsomorphism:commute}
\end{enumerate}
\end{lemma}
\begin{proof}
    The proof of the existence of $\subAlg, \subAlgD \in \subAlgs$
    satisfying (\ref{th:markovOperatorToIsomorphism:subAlgC}) to (\ref{th:markovOperatorToIsomorphism:isomorphism}) uses the Mean Ergodic Theorem
    and is identical to the the proof of Theorem $1.2$, ($4$) $\implies$ ($5$),
    in \cite{GrebikRocha2021};
    we leave it out here.
    To prove (\ref{th:markovOperatorToIsomorphism:commute}),
    let $T_1, T_2 \colon \LTwo \to \LTwo$
    be bounded linear operators satisfying
    $T_1 \circ S = S \circ T_2$
    and $S^* \circ T_1 = T_2 \circ S^*$.
    We get
    $T_1 \circ (S \circ S^*)= S \circ T_2 \circ S^*= (S \circ S^*) \circ T_1$.
    Then, for $f \in \LTwoSub$, we have $(S \circ S^*)f = f$ by
    (\ref{th:markovOperatorToIsomorphism:subAlgC})
    and get
    $T_1 f= (T_1 \circ S \circ S^*) f= (S \circ S^* \circ T_1) f= (S \circ S^*) (T_1 f)$,
    which, again by (\ref{th:markovOperatorToIsomorphism:subAlgC}),
    implies $T_1 f \in \LTwoSub$.
    Therefore, $\subAlg$ is $T_1$-invariant, which proves (\ref{th:markovOperatorToIsomorphism:T1Inv}).
    Analogously, we get that $T_2 \circ (S^* \circ S) = (S^* \circ S) \circ T_2$
    and that $\subAlgD$ is $T_2$ invariant,
    which proves (\ref{th:markovOperatorToIsomorphism:T2Inv}).
    Now, we use (\ref{th:markovOperatorToIsomorphism:commutesExpVal})
    and the $T_2$-invariance of $\subAlgD$ to obtain
    to obtain
    \begin{align*}
        T_1/\subAlg \circ R= S_\subAlg \circ T_1 \circ I_\subAlg \circ S_\subAlg \circ S \circ I_\subAlgD
        &= S_\subAlg \circ T_1 \circ \ExpValSub \circ S \circ I_\subAlgD \tag{\Cref{th:quotientSpaces} (\ref{th:quotientSpaces:ExpVal})}\\
        &= S_\subAlg \circ T_1 \circ S \circ \ExpVal{\subAlgD} \circ I_\subAlgD \tag{(\ref{th:markovOperatorToIsomorphism:commutesExpVal})}\\
        &= S_\subAlg \circ T_1 \circ S \circ I_\subAlgD \tag{\Cref{th:quotientSpaces} (\ref{th:quotientSpaces:adjoint}) and (\ref{th:quotientSpaces:SExpVal})}\\
        &= S_\subAlg \circ S \circ T_2 \circ I_\subAlgD\\
        &= S_\subAlg \circ S \circ I_\subAlgD \circ T_2/\subAlgD \tag{\Cref{le:quoOperator} (\ref{le:quoOperator:invariantI})}\\
        &= R \circ T_2/\subAlgD.
    \end{align*}
\end{proof}

\subsection{Permutation Invariance}
\label{subsec:permutations}

Let $k \ge 1$ and consider $\LTwoProdK$.
Every permutation $\pi \colon [k] \to [k]$
induces a measure-preserving measurable map $\pi \colon X^k \to X^k$
by setting
$\pi(x_1, \dots, x_k) \coloneqq (x_{\pi(1)}, \dots, x_{\pi(k)})$
for all $x_1, \dots, x_k \in X$,
which allows us to consider
its Koopman operator $T_\pi$
on $\LTwoProdK$.
Clearly, the adjoint of $T_\pi$ is given by $T_{\pi^{-1}}$.
We call a $\mu^{\otimes k}$-relatively complete sub-$\sigma$-algebra
$\subAlg \in \subAlgsk$ \textit{permutation invariant}
if $\subAlg$ is $T_{\pi}$-invariant for every permutation $\pi \colon [k] \to [k]$.
It is easy to see that this is the case if and only if
$\pi(\subAlg) \subseteq \subAlg$
for every permutation $\pi \colon [k] \to [k]$,
which again is equivalent to
$\pi(\subAlg) = \subAlg$
for every permutation $\pi \colon [k] \to [k]$.
A trivial example of such a permutation-invariant
sub-$\sigma$-algebra is $\Borel^{\otimes k}$
itself.

For $\subAlg, \subAlgD \in \subAlgsk$,
an operator $T \colon \LTwoProdKQuo \to \LTwoProdKQuoD$
is called \textit{permutation invariant} if
$T_\pi / \subAlgD \circ T = T \circ T_\pi / \subAlg$
for every permutation $\pi \colon [k] \to [k]$.
For the special case $\subAlg = \subAlgD = \Borel^{\otimes k}$,
this means that an operator $T$ on $\LTwoProdK$
is permutation invariant if
$T_{\pi} \circ T = T \circ T_{\pi}$
for every permutation $\pi \colon [k] \to [k]$.
Of course, this notion
depends on the underlying space $(X, \Borel, \mu)$, i.e.,
if we consider $(X^k, \Borel^{\otimes k}, \mu^{\otimes k})$
as the underlying space, then
all these operators mentioned before are trivially permutation invariant.
However, since the intended underlying space is always clear from the context,
we just use the term permutation invariant.
It is not hard to prove that,
if $\subAlg \in \subAlgsk$ is permutation invariant,
then so are $S_\subAlg$ and $I_\subAlg$, i.e.,
$T_\pi / \subAlg \circ S_\subAlg = S_\subAlg \circ T_\pi$
and
$T_\pi \circ I_\subAlg = I_\subAlg \circ T_\pi / \subAlg$
for every permutation $\pi \colon [k] \to [k]$.

\section{Graphon Operators}
\label{sec:graphons}

In this section, we present the main
ingredient to \Cref{th:kWLGraphons}.
The key insight to go from color refinement
to $k$-WL is, for a graphon $W$, to replace the operator $T_W$
on $\LTwo$
by a family $\naturalKFamilyOperatorsShortW$
of operators
on the product space $\LTwoProdK$.
This idea is somewhat already present
in the work of Grohe and Otto \cite[Section $5.1$]{GroheOtto2015},
where they define a family of graphs and consider a matrix
that is a fractional isomorphism between all these graphs simultaneously.
The graphon setting will show that the step of defining these graphs
for the sake of them having the right adjacency matrix
is rather artificial and only works in the setting of (finite-dimensional)
matrices: the operators we define are not integral operators
defined by a graphon.

The family $\naturalKFamilyOperatorsShortW$
we define is closely related to oblivious $k$-WL
and tree decompositions, or more precisely, tree-decomposed graphs.
In \Cref{sec:bilabeledgraphs},
we follow the approach of \cite{MacinskaRoberson2020}
of using a set of \textit{bi-labeled graphs} as building blocks
that are then glued together to form larger graphs.
From our set $\adjNeiGraphs^k$ of bi-labeled graphs,
we obtain precisely the multigraphs
of treewidth at most $k-1$.
In \Cref{sec:graphonOperators},
we adapt the concept of \textit{homomorphism matrices}
of bi-labeled graphs from \cite{MacinskaRoberson2020}
by defining the \textit{graphon operator}
of a bi-labeled graph and a graphon.
The graphon operators of the bi-labeled graphs in $\adjNeiGraphs^k$
and a graphon $W$
then yield the family
$\naturalKFamilyOperatorsShortW$.
We show
how this family is related to homomorphisms:
on the level of bi-labeled graphs, we obtain
all multigraphs of treewidth at most $k-1$,
while we obtain all \textit{homomorphism functions}
of multigraphs of treewidth at most $k-1$
on the operator level.

\subsection{Bi-Labeled Graphs}
\label{sec:bilabeledgraphs}

A \textit{bi-labeled graph} $\bm{G}$ is a triple $(G, \bm{a}, \bm{b})$,
where $G$ is a multigraph and
$\bm{a} \in V(G)^k$, $\bm{b} \in V(G)^\ell$ for $k,\ell \ge 0$ are tuples 
of vertices
such that both the entries of $\bm{a}$ and
the entries of $\bm{b}$ are pairwise distinct;
$\bm{a}$ and $\bm{b}$ may however overlap.
When there is no fear of ambiguity, we sometimes just use the term \textit{graph}
to refer to a bi-labeled graph.
The multigraph $G$ is called the \textit{underlying graph} of $\bm{G}$,
and the tuples $\bm{a}$ and $\bm{b}$ are called
the tuples of \textit{input} and \textit{output vertices},
respectively.
That is, a bi-labeled graph is a multigraph where
additionally \textit{input}
and \textit{output labels} are assigned to the vertices
with every vertex having at most one label of each type.
Note that one usually does not require that every
vertex has at most one label of each type,
cf.\ \cite{MacinskaRoberson2020},
but this is needed to ensure that
graphon operators are well defined;
the precise reason for this will be seen later
when graphon operators are defined.

\begin{figure}
    \centering
    \begin{tikzpicture}
		\node[vertex, label={90:$a_1$}, label={270:$b_1$}] (N1) {};
        \node[vertex, label={90:$a_2$}, label={270:$b_2$}, right = of N1] (N2) {};
        \node[vertex, fill=white, draw=white, right = of N2] (N3g) {};
		\node[vertex, label={90:$a_3$}, above = 0.1cm of N3g] (N3a) {};
		\node[vertex, label={270:$b_3$}, below = 0.1cm of N3g] (N3b) {};

		\node[vertex, label={90:$a_1$}, label={270:$b_1$}, right = 1.5cm of N3g] (A1) {};
        \node[vertex, label={90:$a_2$}, label={270:$b_2$}, right = of A1] (A2) {};
        \node[vertex, label={90:$a_3$}, label={270:$b_3$}, right = of A2] (A3) {};
        \path[draw, thick] (A1) edge (A2);
        \path[draw, thick] (A2) edge (A3);
        \path[draw, thick] (A1) edge[bend right = 15] (A3);
        \node[] (circ) at ($(N3g)!0.5!(A1)$) {$\circ$};

		\node[vertex, label={90:$a_1$}, label={270:$b_1$}, right = 1.5cm of A3] (R1) {};
        \node[vertex, label={90:$a_2$}, label={270:$b_2$}, right = of R1] (R2) {};
        \node[vertex, fill=white, draw=white, right = of R2] (R3g) {};
		\node[vertex, label={90:$a_3$}, above = 0.1cm of R3g] (R3a) {};
		\node[vertex, label={270:$b_3$}, below = 0.1cm of R3g] (R3b) {};
        \path[draw, thick] (R1) edge (R2);
        \path[draw, thick] (R2) edge (R3b);
        \path[draw, thick] (R1) edge (R3b);
        \node[] (equals) at ($(A3)!0.5!(R1)$) {$=$};
    \end{tikzpicture}
    \caption{Composition of bi-labeled graphs.}
    \label{fig:composition}
\end{figure}

Two bi-labeled graphs $\bm{G} = (G, \bm{a}, \bm{b})$ and $\bm{G'} = (G', \bm{a'}, \bm{b'})$
are \textit{isomorphic} if there is an isomorphism $\varphi \colon V(G) \to V(G')$
from $G$ to $G'$ such that
$\varphi(\bm{a}) = \bm{a'}$ and $\varphi(\bm{b}) = \bm{b'}$.
For $k, \ell \ge 0$,
let $\multiGraphs^{k, \ell}$ denote the set of all (isomorphism types of) bi-labeled graphs
with $k$ input and $\ell$ output vertices,
and let $\graphs^{k, \ell} \subseteq \multiGraphs^{k, \ell}$ be the subset
whose underlying graphs are simple.
Let $\multiGraphs \coloneqq \cup_{k, \ell \ge 0} \multiGraphs^{k, \ell}$
and $\graphs \coloneqq \cup_{k, \ell \ge 0} \graphs^{k, \ell}$.

The \textit{transpose} of a bi-labeled graph
$\bm{G} = (G, \bm{a}, \bm{b}) \in \multiGraphs^{k, \ell}$
is the bi-labeled graph
$\bm{G}^* \coloneqq (G, \bm{b}, \bm{a}) \in \multiGraphs^{\ell, k}$,
and $\bm{G}$ is called \textit{symmetric} if $\bm{G}^* = \bm{G}$.
The \textit{composition} of two bi-labeled graphs
$\bm{F_1} = (F_1, \bm{a_1}, \bm{b_1}) \in \multiGraphs^{k, m}$
and $\bm{F_2} = (F_2, \bm{a_2}, \bm{b_2}) \in \multiGraphs^{m, \ell}$
is the bi-labeled graph
$\bm{F_1} \circ \bm{F_2} \coloneqq (F, \bm{a_1}, \bm{b_2}) \in \multiGraphs^{k, \ell}$,
where $F$ is obtained from the disjoint union
of $F_1$ and $F_2$ by identifying vertices $b_{1,i}$ and $a_{2,i}$
for every $i \in [m]$.
An example is given in \Cref{fig:composition}.
The \textit{Schur product} of two bi-labeled graphs
without output labels
$\bm{F_1} = (F_1, \bm{a_1}, ()),
\bm{F_2} = (F_2, \bm{a_2}, ()) \in \multiGraphs^{k, 0}$
is the bi-labeled graph
$\bm{F_1} \blProd \bm{F_2} \coloneqq (F, \bm{a_1}, ()) \in \multiGraphs^{k, 0}$,
where $F$ is obtained from the disjoint union
of $F_1$ and $F_2$ by identifying vertices $a_{1,i}$ and $a_{2,i}$
for every $i \in [m]$.
One usually defines the Schur product for general
bi-labeled graphs in $\multiGraphs^{k, \ell}$ by also identifying output vertices,
cf.\ \cite{MacinskaRoberson2020}.
This, however, can result in vertices with
multiple input or output labels, which we do not allow
by our definition of a bi-labeled graph as remarked earlier.
Both the composition
and the Schur product of bi-labeled graph may introduce parallel
edges, cf.\ \Cref{fig:parallelEdges}, which means that
the set $\graphs$ is neither closed under composition nor under Schur products.

\newcommand{\drawAdjacencyGraph}[4]{
    \node[vertex, label={90:$a_1$}, #2] (#1L) {};
    \node[vertex, label={90:$a_2$}, right = of #1L, #3] (#1R) {};
    \path[draw, thick] (#1L) edge (#1R);
}

\newcommand{\drawParallelEdges}[2]{
    \node[vertex, label={90:$a_1$}, #2] (#1L) {};
    \node[vertex, label={90:$a_2$}, right = of #1L] (#1R) {};
    \path[draw, thick] (#1L) edge[bend left = 30] (#1R);
    \path[draw, thick] (#1L) edge[bend right = 30] (#1R);
}

\begin{figure}
    \centering
    \begin{tikzpicture}
\drawAdjacencyGraph{E3}{label={270:$b_1$}}{label={270:$b_2$}}{$\bm{A}_{2, 12} \circ \identityGraph\colon$}
        \drawAdjacencyGraph{E4}{right = 1.0cm of E3R}{}{$\bm{A}_{2, 12} \circ \identityGraph\colon$}
        \node[] (circ) at ($(E3R)!0.5!(E4L)$) {$\circ$};
        \drawParallelEdges{P2}{right = 1.5cm of E4R}
        \node[] (equals2) at ($(E4R)!0.5!(P2L)$) {$=$};

        \drawAdjacencyGraph{E1}{right = 1.5cm of P2R}{}{$\bm{A}_{2, 12} \circ \identityGraph\colon$}
        \drawAdjacencyGraph{E2}{right = 1.0cm of E1R}{}{$\bm{A}_{2, 12} \circ \identityGraph\colon$}
        \node[] (cdot) at ($(E1R)!0.5!(E2L)$) {$\cdot$};
\node[] (equals) at ($(P2R)!0.5!(E1L)$) {$=$};
    \end{tikzpicture}
    \caption{Both composition and the Schur product may introduce parallel edges.}
    \label{fig:parallelEdges}
\end{figure}

\textit{Treewidth} is a graph parameter that measures
how \enquote{tree-like} a graph is.
Too see how the concept is related to bi-labeled graphs,
let us first recall the standard definition of treewidth
via \textit{tree decompositions}.
Formally, a \textit{tree decomposition} of a multigraph $G$ is a pair $(T, \beta)$,
where $T$ is a tree and $\beta \colon V(T) \to 2^{V(G)}$ such that,
\begin{enumerate}
    \item
        for every $v \in V(G)$,
        the set $\{t \mid v \in \beta(t)\}$ is non-empty and connected and,
    \item
        for every $uv \in E(G)$,
        there is a $t \in V(T)$ such that $u,v \in \beta(t)$.
\end{enumerate}
For every $t \in V(T)$, the set $\beta(t)$ is called the \textit{bag} at $t$.
The \textit{width} of the tree decomposition $(T, \beta)$ is
$\max \{\lvert \beta(t) \rvert \mid t \in V(T)\} - 1$.
The \textit{treewidth} $\operatorname{tw}(G)$ of a multigraph $G$ is the minimum of the widths
of all tree decompositions of $G$.
Note that treewidth is usually defined for simple graphs and not for multigraphs,
but for us,
ignoring the edge multiplicities like in the previous definition
yields just the right notion for multigraphs.
For the sake of completeness,
note that \textit{path decompositions} and \textit{pathwidth}
of a multigraph $G$ can be defined analogously
by only considering tree decomposition $(T, \beta)$
where $T$ is a path.

General tree decompositions are impractical to work with,
and we rather use the following restricted form of a tree decomposition:
First, a \textit{rooted tree decomposition}
is a triple $(T, r, \beta)$ where
$(T, \beta)$ is a tree decomposition of $G$
and $r \in V(T)$ a vertex of $T$, which we view as the root of $T$.
Then,
a \textit{nice tree decomposition} of a multigraph $G$
is a rooted tree decomposition $(T, r, \beta)$
such that
\begin{enumerate}
    \item
        $\beta(r) = \emptyset$ and $\beta(t) = \emptyset$ for every
        leaf $t$ of $(T, r)$ and
    \item
        every internal node $s \in V(T)$ of $T$ is of one of the following three types:

        \;\textbf{Introduce node:}
        $s$ has exactly one child $t$ with $\beta(s) = \beta(t) \cup {v}$
        for some $v \in V(G) \setminus \beta(t)$.

        \;\textbf{Forget node:}
        $s$ has exactly one child $t$ with $\beta(s) \cup {v} = \beta(t)$
        for some $v \in V(G) \setminus \beta(s)$.

        \;\textbf{Join node:}
        $s$ has exactly two children $t_1, t_2$ with $\beta(s) = \beta(t_1) = \beta(t_2)$.
\end{enumerate}
The \textit{width} of $(T, r, \beta)$ is the width of $(T, \beta)$.
Nice tree decompositions are attractive from an algorithmic point of view
because of their simplified structure, which allows one to specify
dynamic-programming algorithms by a simple case distinction
based on the node type.
We do not design such an algorithm here but use
nice tree decompositions to obtain a simple set of bi-labeled graphs
that serve as building blocks for all graphs of treewidth at most $k$;
nice tree decompositions do not pose a restriction since
every graph $G$ with treewidth $k$
has a nice tree decomposition of width $k$.

\begin{lemma}[{\cite[Lemma $13.1.2$]{kloks_treewidth_1994}}]
    Every graph $G$ with treewidth $k$
    has a nice tree decomposition of width $k$.
\end{lemma}

\begin{figure}
    \centering
    \begin{tikzpicture}
		\node[vertex, label={90:$a_1$}, label={270:$b_1$}] (I1) {};
		\node[vertex, label={90:$a_2$}, right = of I1] (I2) {};
		\node[vertex, label={90:$a_3$}, label={270:$b_2$}, right = of I2] (I3) {};
        \node[left = 0.5cm of I1, scale = 1.0] (ILabel){$\introduceGraphOf{3}{2}\colon$};

		\node[vertex, label={90:$a_1$}, label={270:$b_1$}, below = 2.0cm of I1] (F1) {};
		\node[vertex, label={270:$b_2$}, right = of F1] (F2) {};
		\node[vertex, label={90:$a_2$}, label={270:$b_3$}, right = of F2] (F3) {};
        \node[left = 0.5cm of F1, scale = 1.0] (FLabel){$\forgetGraphOf{3}{2}\colon$};

		\node[vertex, label={90:$a_1$}, label={270:$b_1$}, below = 2.0cm of F1] (N1) {};
        \node[vertex, fill=white, draw=white, right = of N1] (N2g) {};
		\node[vertex, label={90:$a_2$}, above = 0.1cm of N2g] (N2a) {};
		\node[vertex, label={270:$b_2$}, below = 0.1cm of N2g] (N2b) {};
		\node[vertex, label={90:$a_3$}, label={270:$b_3$}, right = of N2g] (N3) {};
        \node[left = 0.5cm of N1, scale = 1.0] (NLabel){$\neighborGraphOf{3}{2}\colon$};

		\node[vertex, label={90:$a_1$}, label={270:$b_1$}] (A1) at ($(I3)!0.5!(F3) + (4,0)$) {};
		\node[vertex, label={90:$a_2$}, label={270:$b_2$}, right = of A1] (A2) {};
		\node[vertex, label={90:$a_3$}, label={270:$b_3$}, right = of A2] (A3) {};
        \node[left = 0.5cm of A1, scale = 1.0] (ALabel){$\adjacencyGraphOf{3}{12}\colon$};
        \path[draw, thick] (A1) edge (A2);

\node[vertex, label={90:$a_1$}, below = 2.0cm of A1] (11) {};
		\node[vertex, label={90:$a_2$}, right = of 11] (12) {};
        \node[vertex, label={90:$a_3$}, right = of 12] (13) {};
        \node[left = 0.5cm of 11, scale = 1.0] (1Label){$\identityGraphOf{3}\colon$};

    \end{tikzpicture}
    \caption{The bi-labeled graphs $\introduceGraphOf{3}{2}$, $\forgetGraphOf{3}{2}$, $\neighborGraphOf{3}{2}$, $\adjacencyGraphOf{3}{12}$, and $\identityGraphOf{3}$.}
\end{figure}

We now want to view a bi-labeled graph $\bm{G}$ that is decomposed by a
nice tree decomposition as a term built
from atomic terms, where these atomic terms
act as building blocks for the decomposition
by providing elementary operations like
adding an edge to a bag
or moving between bags of the decomposition.
Such a term can then be evaluated to obtain
$\bm{G}$, and in the next section, we show
that by defining \emph{graphon operators} for each atomic term,
we can alternatively evaluate this expression
to a function describing
the homomorphism density of $\bm{G}$ in a graphon.
The following definition gives us this set $\adjNeiGraphs^k$
of building blocks.

\begin{definition}
    \label{def:differentgraphs}
    Let $k \ge 1$.
    Define
    \begin{enumerate}
        \itemsep0em
        \item
            the \textit{$ij$-adjacency graph}
                $\ijAdjacencyij \coloneqq (([k], \Set{ij}), (1, \dots, k), (1, \dots, k)) \in \graphs^{k,k}$
            for $i \neq j \in [k]$,
        \item
            the \textit{$j$-introduce graph}
            $\jIntroducej \coloneqq (([k], \emptyset), (1, \dots, k),(1, \dots, j-1, j+1, \dots, k)) \in \graphsOf{k, k-1}$,
        \item
            the \textit{$j$-forget graph}
            $\jForgetj \coloneqq {\jIntroducej}^* \in \graphsOf{k-1, k}$, and
        \item
            the \textit{$j$-neighbor graph}
            $\jNeighborj \coloneqq \jIntroducej \circ \jForgetj \in \graphsOf{k,k}$ for $j \in [k]$, and finally,
        \item the \textit{all-one graph}
            $\identityGraph \coloneqq (([k], \emptyset), (1, \dots, k), ()) \in \graphs^{k,0}$.
    \end{enumerate}
    Let $\adjGraphs^k \coloneqq \Set{\ijAdjacencyij \mid i \neq j \in [k]} \subseteq \graphs^{k,k}$
    and $\jNeighbors^k \coloneqq \Set{\jNeighbor{j} \mid j \in [k]} \subseteq \graphsOf{k,k}$
    be the sets of all adjacency graphs and all neighbor graphs, respectively.
    Finally, let
    $\adjNeiGraphs^k \coloneqq \adjGraphs^k \cup \jNeighbors^k$.
\end{definition}

The set $\adjNeiGraphs^k$ of adjacency and neighbor graphs
together with $\identityGraph$
suffices to construct
essentially every graph of treewidth at most $k$.
Let us first illustrate this with an
example: Consider the tree-decomposed graph in
\Cref{fig:treeDecompTerm},
which can be translated to
the language of bi-labeled graphs as the expression
\begin{equation}
    (\neighborGraphOf{3}{3}
        \circ \adjacencyGraphOf{3}{12} \circ \adjacencyGraphOf{3}{13}
        \circ \neighborGraphOf{3}{1}
        \circ \adjacencyGraphOf{3}{12} \circ \adjacencyGraphOf{3}{13}
        \circ \identityGraphOf{3})
    \blProd
    (\neighborGraphOf{3}{1}
        \circ \adjacencyGraphOf{3}{13} \circ \adjacencyGraphOf{3}{23}
        \circ \neighborGraphOf{3}{3}
        \circ \adjacencyGraphOf{3}{13} \circ \adjacencyGraphOf{3}{23}
        \circ \identityGraphOf{3}).
    \label{eq:termExample}
\end{equation}
Consider the left subtree and the subexpression left of the Schur product.
The all-one graph $\identityGraphOf{3}$ corresponds
to the leaf containing $v_6$, $v_2$, and $v_4$ in the left subtree,
the adjacency graphs $\adjacencyGraphOf{3}{12}$
and $\adjacencyGraphOf{3}{13}$ insert
the edges $v_6 v_2$ and $v_6 v_4$, respectively,
and
the neighbor graph $\neighborGraphOf{3}{1}$ moves
to the new bag containing $v_1$, $v_2$, and $v_4$
as it replaces the first vertex $v_6$ by $v_1$.
Then, the adjacency graphs $\adjacencyGraphOf{3}{12}$
and $\adjacencyGraphOf{3}{13}$ insert
the edges $v_1 v_2$ and $v_1 v_4$, respectively,
before the neighbor graph $\neighborGraphOf{3}{3}$ moves
to the new bag containing $v_1$, $v_2$, and $v_3$
as it replaces the third vertex $v_4$ by $v_1$.
The right subexpression is constructed analogously
from the right subtree, and finally,
the Schur product corresponds to the join node
of the tree decomposition
and glues the two bi-labeled graphs obtained from the two subexpressions
together.
Then,
evaluating the expression in Equation~(\ref{eq:termExample}) yields the bi-labeled graph
in \Cref{fig:evaluatedTerm}, i.e.,
its underlying graph is the graph in \Cref{fig:treeDecompTerm},
its input labels are $v_1$, $v_2$, and $v_3$,
and it has no output labels.
The expression in (\ref{eq:termExample})
is an example of a \emph{term},
and we can formalize this view of tree-decomposed graphs
as expressions built from bi-labeled graphs
by composition and the Schur product.

\begin{figure}
    \centering
    \begin{tikzpicture}
        \node[vertex, label={90:$v_2$}] (V2) {};

        \node[vertex, label={180:$v_1$}, above left = 0.7cm and 0.7cm of V2] (V1) {};
        \node[vertex, label={180:$v_6$}, below left = 0.7cm and 0.7cm of V2] (V3) {};
        \node[vertex, label={180:$v_4$}, below left = 0.7cm and 0.7cm of V1] (V4) {};

        \node[vertex, label={0:$v_3$}, above right = 0.7cm and 0.7cm of V2] (V5) {};
        \node[vertex, label={0:$v_7$}, below right = 0.7cm and 0.7cm of V2] (V6) {};
        \node[vertex, label={0:$v_5$}, below right = 0.7cm and 0.7cm of V5] (V7) {};

        \path[draw, thick]
            (V1) edge (V2)
            (V2) edge (V3)
            (V1) edge (V4)
            (V4) edge (V3)
            (V5) edge (V2)
            (V2) edge (V6)
            (V5) edge (V7)
            (V7) edge (V6)
        ;

        \node[right = 5.0cm of V2, yshift = 1.8cm] (B1) {$\{v_1,\, v_2,\, v_3\}$};

        \node[below left = 0.7cm and -0.5cm of B1] (B2) {$\{v_1,\, v_2,\, v_3\}$};
        \node[below = 0.7cm of B2] (B3) {$\{v_1,\, v_2,\, v_4\}$};
        \node[below = 0.7cm of B3] (B4) {$\{v_6,\, v_2,\, v_4\}$};

        \node[below right = 0.7cm and -0.5cm of B1] (B5) {$\{v_1,\, v_2,\, v_3\}$};
        \node[below = 0.7cm of B5] (B6) {$\{v_5,\, v_2,\, v_3\}$};
        \node[below = 0.7cm of B6] (B7) {$\{v_5,\, v_2,\, v_7\}$};

        \path[draw, thick]
            (B1) edge (B2)
            (B2) edge (B3)
            (B3) edge (B4)
            (B1) edge (B5)
            (B5) edge (B6)
            (B6) edge (B7)
        ;
    \end{tikzpicture}
    \caption{A graph and a tree decomposition of it.}
    \label{fig:treeDecompTerm}
\end{figure}

\begin{figure}
    \centering
    \begin{tikzpicture}
        \node[vertex, label={90:$a_2$}] (L2) {};

        \node[vertex, label={90:$a_1$}, above left = 0.7cm and 0.7cm of L2] (L1) {};
        \node[vertex, below left = 0.7cm and 0.7cm of L2] (L3) {};
        \node[vertex, below left = 0.7cm and 0.7cm of L1] (L4) {};

        \node[vertex, label={90:$a_3$}, above right = 0.7cm and 0.7cm of L2] (L5) {};
        \node[vertex, below right = 0.7cm and 0.7cm of L2] (L6) {};
        \node[vertex, below right = 0.7cm and 0.7cm of L5] (L7) {};

        \path[draw, thick]
            (L1) edge (L2)
            (L2) edge (L3)
            (L1) edge (L4)
            (L4) edge (L3)
            (L5) edge (L2)
            (L2) edge (L6)
            (L5) edge (L7)
            (L7) edge (L6)
        ;
    \end{tikzpicture}
    \caption{The bi-labeled graph obtained by evaluation of (\ref{eq:termExample}).}
    \label{fig:evaluatedTerm}
\end{figure}

\begin{definition}
    Let $k \ge 1$ and $\mathcal{F} \subseteq \multiGraphs^{k,k}$
    be a set of bi-labeled graphs with $k$ input and $k$ output labels.
    The set $\treeClosure{\mathcal{F}}$ of \emph{$\mathcal{F}$-terms (terms)}
    is the smallest set of expressions such that
    \begin{enumerate}
        \item
            $\identityGraph \in \treeClosure{\mathcal{F}}$, \label{def:terms:identity}
        \item
            $\bm{F} \circ \term \in \treeClosure{\mathcal{F}}$
            for all $\bm{F} \in \mathcal{F}$, $\term \in \treeClosure{\mathcal{F}}$, and \label{def:terms:composition}
        \item
            $\term_1 \cdot \term_2 \in \treeClosure{\mathcal{F}}$
            for all $\term_1, \term_2 \in \treeClosure{\mathcal{F}}$. \label{def:terms:product}
    \end{enumerate}
    Similarly, let $\pathClosure{\mathcal{F}} \subseteq \treeClosure{\mathcal{F}}$
    be the smallest set of terms satisfying Conditions~(\ref{def:terms:identity}) and (\ref{def:terms:composition}).
    For a term $\term \in \treeClosure{\mathcal{F}}$,
    let $\graphOfTerm$
    denote
    the bi-labeled graph obtained from evaluating it.
\end{definition}

Note that, for a set $\mathcal{F} \subseteq \multiGraphs^{k,k}$
and a term $\term \in \treeClosure{\mathcal{F}}$, the bi-labeled graph $\graphOfTerm$ is well-defined
as we always have $\graphOfTerm \in \multiGraphs^{k, 0}$
since every term starts with $\identityGraph$.
As demonstrated before, for the specific set $\adjNeiGraphs^k$
of adjacency and neighbor graphs,
a term $\term \in \treeClosureK$
is essentially a tree-decomposed graph,
where the tree decomposition is rooted,
the multigraph being decomposed is the
bi-labeled graph underlying $\graphOfTerm$,
and the bag at the root is given by the input vertices
of $\graphOfTerm$.

\begin{theorem}
    \label{le:graphsOfTerms}
    The
    underlying graphs of the bi-labeled graphs obtained by
    evaluating the terms in
    $\pathClosure{\adjNeiGraphs^k}$
    and
    $\treeClosure{\adjNeiGraphs^k}$
    are, up to isolated vertices, precisely the multigraphs
    of pathwidth and treewidth at most $k-1$, respectively.
\end{theorem}
\begin{proof}
    The proof is quite simple and we only sketch it here.
    The main idea is the following:
    For a bi-labeled graph $\bm{G} \in \multiGraphs^{k, \ell}$,
    the composition $\ijAdjacencyij \circ \bm{G}$
    of the $ij$-adjacency graph $\ijAdjacencyij$
    with $\bm{G}$
    adds an edge between the vertices in $\bm{G}$
    with the $i$th and the $j$th input label.
    When viewing the $k$ input vertices of $\bm{G}$
    as the \enquote{current} bag of a tree decomposition,
    we can compose adjacency graphs with $\bm{G}$
    to add any edge between the vertices
    of this bag.
    Similarly,
    the composition $\jNeighborj \circ \bm{G}$
    of the $j$-neighbor graph $\jNeighborj$
    with $\bm{G}$ moves to a new bag obtained by replacing
    the vertex with the $j$th input label by a fresh vertex.
    In terms of nice tree decompositions, this corresponds to
    an introduce node that immediately follows on a forget node.

    Following the interpretation sketched above,
    a term $\term \in \treeClosure{\adjNeiGraphs^k}$
    can inductively be translated into
    a rooted tree decomposition
    for the underlying graph of $\graphOfTerm$
    of width $k-1$:
    For $\identityGraph$, we take a single leaf containing
    $k$ fresh vertices.
    For $\ijAdjacencyij \circ \term$,
    we just take the tree decomposition for $\term$.
    For $\jNeighborj \circ \term$,
    we take the tree decomposition for $\term$
    and connect a new node to its root and replace the
    vertex specified by the $j$th input label in $\graphOfTerm$
    by the fresh vertex introduced by $\jNeighborj$.
    Finally, for the Schur product $\term_1 \cdot \term_2$
    we take the tree decompositions of $\term_1$ and $\term_2$
    and turn them into a single tree decomposition
    by connecting their roots using a join node.

    For the converse direction,
    we inductively turn a nice tree decomposition $(T, r, \beta)$ of a graph $G$
    of width at most $k-1$
    into a term $\term \in \treeClosure{\adjNeiGraphs^k}$
    such that the underlying graph of $\graphOfTerm$
    is $G$ with some additional isolated vertices.
    This direction requires a bit more thought
    since we have to make sure that the set $\adjNeiGraphs^k$
    is not too restrictive.
    We do this by modifying the nice tree decomposition
    $(T, r, \beta)$ as follows, where we have to keep in mind
    that a term fixes an ordering of the vertices of the graph:
    First, pad the bag of every leaf to size $k$
    by adding $k$ fresh isolated vertices.
    At an introduce node,
    add a forget node below
    that removes one of the isolated vertices.
    At a forget node,
    add an introduce node directly above adding a fresh isolated vertex.
    At a join node, re-order the vertices in one of the terms
    such that the original vertices of $G$ are at the same positions
    in both terms
    and, then, identify every additional isolated vertex with
    the one at the same position in the other term.
    Then, on this modified tree decomposition, we can apply the inverse
    construction to the one described in the previous paragraph.
\end{proof}

The \textit{height} $h(\term)$ of a term $\term \in \treeClosureK$
is inductively defined by letting $h(\identityGraph) \coloneqq 0$,
$h(\bm{N} \circ \term) \coloneqq h(\term) + 1$
for all $\bm{N} \in \jNeighbors^k$, $\term \in \treeClosureK$,
$h(\bm{A} \circ \term) \coloneqq h(\term)$
for all $\bm{A} \in \adjGraphs^k$, $\term \in \treeClosureK$,
and
$h(\term_1 \cdot \term_2) \coloneqq  \max\Set{h(\term_1), h(\term_2)}$
for all $\term_1, \term_2 \in \treeClosureK$.
Then, the height of $\term$ corresponds to the height
of the tree of the tree decomposition when viewing $\term$ as a tree-decomposed graph.

\begin{remark}
    In terms of nice tree decompositions,
    a $j$-neighbor graph corresponds to
    an introduce node that immediately follows on a forget node.
    This is also nicely reflected in the definition of $\jNeighborj$:
    it is a $j$-introduce graph composed with a $j$-forget graph,
    which correspond to an introduce node and a forget node, respectively.
    \Cref{le:graphsOfTerms} and its proof
    would have been simplified if we included
    individual
    $j$-introduce and $j$-forget graphs,
    in $\adjNeiGraphs^k$:
    we would not have to deal with isolated vertices.
    However, just considering $j$-neighbor graphs
    instead  has the advantage that all bi-labeled graphs
    in $\adjNeiGraphs^k$
    have both $k$ input and $k$ output labels, which means
    that we can restrict ourselves to the single product space $\LTwoProdK$
    later on instead of having to deal with
    all product spaces $L^2(X^1, \mu^{\otimes 1}), \dots, \LTwoProdK$.
    For this very same reason, we use the all-one graph $\identityGraphOf{k}$
    instead of defining a \emph{leaf graph}
    as an empty graph and then
    using $k$ individual introduce graphs,
    each introducing a single vertex, to obtain a bag of size $k$.
    The downside of these restrictions is that we may have to add
    isolated vertices to a graph, cf.\ the statement and proof of
    \Cref{le:graphsOfTerms}.
    Since additional isolated vertices do not affect
    the homomorphism density of a graph in a graphon,
    this is perfectly fine for us.
    Moreover, it is also not a restriction that,
    in a $j$-neighbor graph, both the forgotten and introduced
    vertex use the $j$th label
    since we may just inductively re-order the vertices of the whole term
    afterwards to make sure that the newly introduced vertex
    has the desired label;
    this is also done in the proof of \Cref{le:graphsOfTerms}.
\end{remark}

In the proof of \Cref{le:graphsOfTerms}, we use that
we can re-label input vertices by inductively re-labeling a whole term.
This could have been simplified by
including \textit{permutation graphs} in $\adjNeiGraphs^k$:
for $k \ge 1$ and a permutation $\pi \colon [k] \to [k]$,
we define the \textit{permutation graph}
\begin{equation*}
    \permutationGraph \coloneqq (([k], \emptyset), (1, \dots, k), (\pi(1), \dots, \pi(k)) \in \graphs^{k,k}.
\end{equation*}
Moreover, for
a tuple $\bm{a} \in V(F)^k$ of vertices of a graph $F$,
let $\pi(\bm{a}) \coloneqq (a_{\pi(1)}, \dots, a_{\pi(k)})$.
Then, for a bi-labeled graph $(F, \bm{a}, \bm{b}) \in \multiGraphsOf{k, \ell}$,
we have
$\permutationGraph \circ (F, \bm{a}, \bm{b})
= (F, \pi^{-1}(\bm{a}), \bm{b})$
for every permutation $\pi \colon [k] \to [k]$
and
$(F, \bm{a}, \bm{b}) \circ \permutationGraph
= (F, \bm{a}, \pi(\bm{b}))$
for every permutation $\pi \colon [\ell] \to [\ell]$.
In order to keep the set $\adjNeiGraphs^k$ simple,
we do not include permutations graphs in it.
Nevertheless, they
come in handy later when proving that the operators
and sub-$\sigma$-algebras
we define are permutation invariant.

\begin{figure}
    \centering
    \begin{tikzpicture}
        \node[vertex, label={90:$a_{\pi(1)}$}, label={270:$b_1$}] (P1) {};
		\node[right = 0.5cm of P1] (Pd) {$\dots$};
        \node[vertex, label={90:$a_{\pi(k)}$}, label={270:$b_k$}, right = 0.5cm of Pd] (Pk) {};
        \node[left = 0.5cm of P1, scale = 1.0] (PLabel){$\bm{P}_{\pi}\colon$};

        \node[vertex, label={90:$a_1$}, label={270:$b_{\pi^{-1}(1)}$}, right = 2.0cm of Pk] (Pp1) {};
		\node[right = 0.5cm of Pp1] (Ppd) {$\dots$};
        \node[vertex, label={90:$a_k$}, label={270:$b_{\pi^{-1}(k)}$}, right = 0.5cm of Ppd] (Ppk) {};

        \node[] (equals) at ($(Pk)!0.5!(Pp1)$) {$=$};
    \end{tikzpicture}
    \caption{Two representations of the graph $\permutationGraph$. In the first representation, the vertices are sorted by their output labels, and in the second, by their their input labels.}
\end{figure}

\subsection{Graphon Operators}
\label{sec:graphonOperators}

\textit{Graphon operators} generalize the homomorphism density $t(F, W)$
of a multigraph $F$ in a graphon $W \colon X \times X \to [0,1]$
to bi-labeled graphs.
To this end,
let $\bm{F} = (F, \bm{a}, \bm{b})\in \multiGraphs^{k, \ell}$
be a bi-labeled graph.
To simplify notation, let $t(\bm{F}, W) \coloneqq t(F, W)$
denote the homomorphism density of the underlying graph of $\bm{F}$ in $W$,
i.e., we ignore both the input and output labels.
Now, let us first take the input labels of $\bm{F}$ into account,
that is, we view $\bm{F}$ as a multi-rooted multigraph
and the homomorphism density becomes
a function by not fixing
the vertices that have an input label.
Formally, the
\textit{homomorphism function} of
$\bm{F}$ in $W$
is the function $\homFunctionFW \colon X^k \to [0,1]$
defined by
\begin{equation}
    \homFunctionFW (x_{a_1}, \dots, x_{a_k})
    \coloneqq
    \int_{X^{V(F)} \setminus \bm{a}} \prod_{ij \in E(F)} W(x_i, x_j)
    \, d\mu^{\otimes V(F) \setminus \bm{a}}(\bar{x})
    \label{eq:homFunctionGraphon}
\end{equation}
for all $x_{a_1}, \dots, x_{a_k} \in X$.
We note that we again slightly abuse notation and assume that
each factor $W(x_i, x_j)$ occurs as often in the product
$\prod_{ij \in E(F)} W(x_i, x_j)$ as $ij$ is contained in $E(F)$.
The Tonelli-Fubini theorem immediately yields that
\begin{equation*}
    \langle \allOne_{X^k}, \homFunctionFW \rangle = t(\bm{F}, W).
\end{equation*}

Taking both input and output labels of $\bm{F}$ into account,
we obtain an operator $\OperatorFW$ instead of a function $\homFunctionFW$
by, intuitively, \enquote{gluing} a given function $f$
to the output vertices of $\bm{F}$ to obtain the function $\OperatorFW f$.
Formally, the \textit{$\bm{F}$-operator of $W$}
is the mapping
$\OperatorFromTo{\bm{F}}{W} \colon \LTwoProdL \to \LTwoProdK$
defined by
\begin{equation}
    (\OperatorFromTo{\bm{F}}{W} f) (x_{a_1}, \dots, x_{a_k})\coloneqq \int\limits_{X^{V(F) \setminus \bm{a}}}\prod_{ij \in E(F)} W(x_i, x_j) \cdot f(x_{b_1}, \ldots, x_{b_\ell})\, d\mu^{\otimes V(F) \setminus \bm{a}}(\bar{x})
    \label{eq:homOperatorDefinition}
\end{equation}
for every $f \in \LTwoProdL$
and all $x_{a_1}, \dots, x_{a_k} \in X$,
where we again slightly abuse notation w.r.t.\ the product over all $ij \in E(F)$.
The goal of this definition is that an application of $\OperatorFW$
to a homomorphism function $\homFunction{\bm{G}}{W}$
should yield the homomorphism function $\homFunction{\bm{F} \circ \bm{G}}{W}$;
as we will see in \Cref{le:graphOperations}, this is the case.
We note that
$\homFunctionFW = \OperatorFW \allOne_{X^\ell}$
as an element of $\LInftyProdK$ and, in particular,
\begin{equation*}
    \langle \allOne_{X^k}, \OperatorFW \allOne_{X^\ell} \rangle = t(\bm{F}, W).
    \label{le:homOperatorHomDensity}
\end{equation*}

The definition of $\OperatorFW$
only depends on the isomorphism type of $\bm{F}$,
i.e., isomorphic bi-labeled graphs
define the same operator.
\begin{lemma}
    \label{le:isomorphicGraphsOperator}
    Let $\bm{F} = (F, \bm{a}, \bm{b}), \bm{F'} = (F', \bm{a'}, \bm{b'})
    \in \multiGraphsOf{k, \ell}$
    be isomorphic bi-labeled graphs and
    $W \colon X \times X \to [0,1]$ be a graphon.
    Then,
    $\OperatorFromTo{\bm{F}}{W} =  \OperatorFromTo{\bm{F'}}{W}$.
\end{lemma}
\begin{proof}
    Let $\varphi \colon V(F) \to V(F')$ be an isomorphism
    from $F$ to $F'$ such that
    $\varphi(\bm{a}) = \bm{a'}$ and $\varphi(\bm{b}) = \bm{b'}$.
    Note that
    $\prod_{ij \in E(F')} W(x_i, x_j) = \prod_{ij \in E(F)} W(x_{\varphi(i)}, x_{\varphi(j)})$
    for $\mu^{\otimes V(F')}$-almost every $\bar{x} \in X^{V(F')}$.
    Hence, by substituting variable $x_{\varphi(i)}$ by $x_i$
    for every $i \in V(F)$
    in (\ref{eq:homOperatorDefinition}),
    we immediately get
    $\OperatorFromTo{\bm{F}}{W} =  \OperatorFromTo{\bm{F'}}{W}$.
\end{proof}
Moreover, if $\bm{F}$ does not have any edges,
then the definition of $\OperatorFromTo{\bm{F}}{W}$
is independent of $W$ and we just write $T_{\bm{F}}$.
We just have to be a bit careful since
$T_{\bm{F}}$ is still dependent on
the standard Borel space $(X, \Borel)$
and the Borel probability measure $\mu$.

\begin{example}
    \label{ex:graphonOperators}
    \begin{enumerate}
        \item
        Define $\bm{A} \coloneqq (([2], \Set{12}), (1), (2)) \in \graphs^{1,1}$
        to be the edge with one input and one output vertex.
        Let $W \colon X \times X \to [0,1]$ be a graphon
        and $f \in \LTwo$.
        Then,
        \begin{equation*}
            (\OperatorFromTo{\bm{A}}{W} f)(x_1)
            = \int_X W(x_1, x_2) \cdot f(x_2) \, d \mu(x_2)
            = (T_W f) (x_1)
        \end{equation*}
        for every $x_1 \in X$, i.e.,
        $\OperatorFromTo{\bm{A}}{W} = T_W$.
        \item
        Define $\bm{D} \coloneqq (([2], \Set{12}), (1), (1)) \in \graphs^{1,1}$
        to be the edge where one vertex has both an input
        and an output label.
        Let $W \colon X \times X \to [0,1]$ be a graphon
        and $f \in \LTwo$.
        Then,
        \begin{equation*}
            (\OperatorFromTo{\bm{D}}{W} f)(x_1)
            = \int_X W(x_1, x_2) \cdot f(x_1) \, d \mu(x_2)
            = f(x_1) \cdot \deg_W(x_1)
        \end{equation*}
        for every $x_1 \in X$,
        where $\deg_W(x) \coloneqq \int_X W(x, y) \, d \mu(y)$
        for every $x \in X$.
        \item
        Let $k \ge 1$ and $\pi \colon [k] \to [k]$ be a permutation.
        Since $\permutationGraph$ does not have any edges,
        $\OperatorFromTo{\permutationGraph}{W}$ is independent of
        a specific graphon $W \colon X \times X \to [0,1]$
        and we simply denote it by $\OperatorFrom{\permutationGraph}$
        as agreed on before.
        The operator $\OperatorFrom{\permutationGraph}$
        is equal to the Koopman operator $T_\pi$
        of the measure-preserving measurable map $X^k \to X^k$ induced by $\pi$:
        Let $f \in \LTwoProdK$.
        Then,
        \begin{equation*}
            (\OperatorFrom{\permutationGraph} f)(x_1, \dots, x_k)= f(x_{\pi(1)}, \dots, x_{\pi(k)})= (f \circ \pi)(x_1, \dots, x_k)= (T_\pi f)(x_1, \dots, x_k)
        \end{equation*}
        for all $x_1, \dots, x_k \in X$.
    \end{enumerate}
\end{example}

The Tonelli-Fubini theorem
and the Cauchy-Schwarz inequality allow us to
verify that \Cref{eq:homOperatorDefinition}
is indeed is a well-defined operator and, furthermore, a contraction,
where it is crucial
that we made the somewhat unusual assumption that, in a bi-labeled graph,
every vertex has at most one label of each type.
The intuitive reason for this is that the diagonal
$(X^k, \Borel^{\otimes k}, \mu^{\otimes k})$
has measure zero (as long as our standard Borel space is atom free),
a problem which one does not face in the case of (finite-dimensional) matrices.
\begin{lemma}
    \label{le:homOperatorWellDefined}
    Let $\bm{F} \in \multiGraphsOf{k, \ell}$ be a bi-labeled graph
    and $W \colon X \times X \to [0,1]$ be a graphon.
    Then, $\OperatorFromTo{\bm{F}}{W} \colon \LTwoProdL \to \LTwoProdK$ is
    well-defined and an $L^2$- and $L^\infty$-contraction.
\end{lemma}
\begin{proof}
    Let $\bm{F} = (F, \bm{a}, \bm{b})$.
    For $f \in \LLTwoProdL$, define
    $S$ by
    \begin{equation*}
        (Sf) (\bar{x}) \coloneqq \prod_{ij \in E(F)} W(x_i, x_j) \cdot f(x_{b_1}, \ldots, x_{b_\ell})
    \end{equation*}
    for every $\bar{x} \in X^{V(F)}$.
    Then, $S$ is clearly linear, and
    we claim that
    \begin{equation*}
        \normT{S f} \le \normI{W}^{\es(F)} \cdot \normT{f} < \infty,
    \end{equation*}
    which not only implies that
    $S f \in \LL^2(X^{V(F)}, \mu^{\otimes V(F)})$
    but also that
    $S$ is a well-defined mapping
    $\LTwoProdL \to L^2(X^{V(F)}, \mu^{\otimes V(F)})$
    that additionally is an $L^2$-contraction.

    To prove the claim, we first note that
    it is easy to see
    that $\bar{x} \mapsto \prod_{ij \in E(F)} W(x_i, x_j)$
    is a function in $\LL^\infty(X^{V(F)}, \mu^{\otimes V(F)})$
    with $\normI{\cdot}$-norm
    at most $\normI{W}$ since $\bm{F}$ does not have loops, i.e.,
    $i \neq j$.
    Second,
    $\bar{x} \mapsto f(x_{b_1}, \ldots, x_{b_\ell})$ is a function
    in $\LL^2(X^{V(F)}, \mu^{\otimes V(F)})$
    and, by the Tonelli-Fubini theorem,
    its $\normT{\cdot}$-norm
    is $\normT{f}$, where
    it is important that the entries
    of $\bm{b}$ are pairwise distinct.
    Then, the claim easily follows.
    Similarly, for $f \in \LLInftyProdL$,
    we have $\normI{S f} \le \normI{W}^{\es(F)} \cdot \normI{f}$, i.e.,
    $S$ is also an $L^\infty$-contraction.

    Now, $\OperatorFromTo{\bm{F}}{W} f$ is the function
    obtained from $Sf$ by
    integrating out the variables $x_i$ for $i \in V(F) \setminus \bm{a}$.
    The Cauchy-Schwarz inequality together with the Tonelli-Fubini Theorem
    yields that this is an $L^2$-contraction, i.e.,
    $\normT{\OperatorFromTo{\bm{F}}{W} f} \le \normT{Sf}$.
    Moreover, this is trivially an $L^\infty$-contraction, i.e.,
    $\normI{\OperatorFromTo{\bm{F}}{W} f} \le \normI{Sf}$
    for $f \in \LLInftyProdL$.
    Since this operation is also linear, this finishes the proof.
\end{proof}

The operator $\OperatorFW$ was defined such that
the application
to a homomorphism function $\homFunction{\bm{G}}{W}$
yields the homomorphism function $\homFunction{\bm{F} \circ \bm{G}}{W}$.
The following lemma formalizes this by stating that the composition
of bi-labeled graphs corresponds to the composition of graphon operators.
Moreover, the analogous correspondence holds
between the transpose and the Hilbert adjoint and between
the Schur product and the point-wise product.

\begin{lemma}
    \label{le:graphOperations}
    \label{le:homOperatorAdjoint}
    Let $W \colon X \times X \to [0,1]$ be a graphon.
    Then,
    \begin{enumerate}
        \item
            $\OperatorFromTo{\bm{F}^*}{W} = \OperatorFromTo{\bm{F}}{W}^*$
            for every $\bm{F} \in \multiGraphs$,
            \label{le:graphOperations:transpose}
        \item
            if $\bm{F} \in \multiGraphs$ is symmetric, then $\OperatorFromTo{\bm{F}}{W}$ is self-adjoint,
            \label{le:graphOperations:symmetric}
        \item
            $\OperatorFromTo{\bm{F_1} \circ \bm{F_2}}{W}= \OperatorFromTo{\bm{F_1}}{W} \circ \OperatorFromTo{\bm{F_2}}{W}$
            for all $\bm{F_1} \in \multiGraphsOf{k, m}$,
            $\bm{F_2} \in \multiGraphsOf{m, \ell}$, $k, \ell, m \ge 0$, and
            \label{le:graphOperations:composition}
        \item
            $\OperatorFromTo{\bm{F_1} \blProd \bm{F_2}}{W} (c_1 \cdot c_2)
            = (\OperatorFromTo{\bm{F_1}}{W} c_1) \cdot
            (\OperatorFromTo{\bm{F_2}}{W} c_2)$
            for all $c_1, c_2 \in \R$, $\bm{F_1}, \bm{F_2} \in \multiGraphsOf{k, 0}$,
            $k \ge 0$.
            \label{le:graphOperations:product}
    \end{enumerate}
\end{lemma}
\begin{proof}
    (\ref{le:graphOperations:transpose}):
    We have
    \begin{align*}
        &\langle \OperatorFromTo{\bm{F}}{W} f, g \rangle\\
        = {}&
        \int\limits_{{X^{\bm{a}}}}
        \vast(
        \int\limits_{{X^{V(F) \setminus \bm{a}}}}\prod_{{ij \in E(F)}} \!\!\! W(x_i, x_j) \cdot f(x_{b_1}, \ldots, x_{b_\ell})\,d\mu^{\otimes V(F) \setminus \bm{a}}(\bar{x})\vast)\cdot g(x_{a_1}, \ldots, x_{a_k})\,d\mu^{\otimes \bm{a}}(\bar{x})\\
        = {}&
        \int\limits_{{X^{V(F)}}}\prod_{{ij \in E(F)}} \!\!\! W(x_i, x_j) \cdot f(x_{b_1}, \ldots, x_{b_\ell})\cdot g(x_{a_1}, \ldots, x_{a_k})\,d\mu^{\otimes V(F)}(\bar{x})\\
        = {}&
        \int\limits_{\mathclap{X^{\bm{b}}}}\;f(x_{b_1}, \ldots, x_{b_\ell}) \cdot \vast(
            \int\limits_{{X^{V(F) \setminus \bm{b}}}}\prod_{{ij \in E(F)}} \!\!\! W(x_i, x_j)\cdot g(x_{a_1}, \ldots, x_{a_k})\,d\mu^{\otimes V(F) \setminus \bm{b}}(\bar{x})\vast)\,d\mu^{\otimes \bm{b}}(\bar{x})\\
        = {}&\langle f, \OperatorFromTo{\bm{F}^*}{W} g \rangle
    \end{align*}
    for all $f \in \LTwoProdL, g \in \LTwoProdK$
    by the Tonelli-Fubini theorem,
    which is applicable since the product being integrated
    is a function in $L^1(X^{V(F)}, \mu^{\otimes V(F)})$
    by the Cauchy-Schwarz inequality.

    (\ref{le:graphOperations:symmetric}):
    By (\ref{le:graphOperations:transpose}), we have
    $\OperatorFromTo{\bm{F}}{W}^* = \OperatorFromTo{\bm{F}^*}{W}= \OperatorFromTo{\bm{F}}{W}$.

    (\ref{le:graphOperations:composition}):
    Let $\bm{F_1} = (F_1, \bm{a_1}, \bm{b_1})$,
    $\bm{F_2} = (F_2, \bm{a_2}, \bm{b_2})$, and
    $\bm{F_1} \circ \bm{F_2} = (F, \bm{a_1}, \bm{b_2})$.
    In the following, we identify vertices
    $b_{1,1}, \dots, b_{1,m}$
    with $a_{2,1}, \dots, a_{2,m}$.
    Note that the sets
    $V(F_1) \setminus \bm{a_1}$
    and $V(F_2) \setminus \bm{b_1} = V(F_2) \setminus \bm{a_2}$
    form a partition of $V(F_1 \circ F_2) \setminus \bm{a_1}$.
    Then, we have
    \begin{align*}
        &(\OperatorFromTo{\bm{F_1}}{W} (\OperatorFromTo{\bm{F_2}}{W} f)) (x_{a_{1,1}}, \dots, x_{a_{1,k}})\\
        = {}&
        \int\limits_{{X^{V(F_1) \setminus \bm{a_1}}}}\!\prod_{{ij \in E(F_1)}} \!\!\! W(x_i, x_j)\cdot (\OperatorFromTo{\bm{F_2}}{W} f)(x_{b_{1,1}}, \ldots, x_{b_{1,m}})\,d\mu^{\otimes V(F_1) \setminus \bm{a_1}}(\bar{x})\\
        = {}&
        \int\limits_{{X^{V(F_1) \setminus \bm{a_1}}}}\!\prod_{{ij \in E(F_1)}} \!\!\! W(x_i, x_j)\cdot \!
        \vast(
        \int\limits_{{X^{V(F_2) \setminus \bm{a_2}}}}\!\prod_{{ij \in E(F_2)}} \!\!\! W(x_i, x_j)
        \cdot f(x_{b_{2,1}}, \ldots, x_{b_{2,\ell}})\begin{aligned}[t]
            &\,d\mu^{\otimes V(F_2) \setminus \bm{a_2}}(\bar{x})\vast)\\[0.5em]
            &\;\,d\mu^{\otimes V(F_1) \setminus \bm{a_1}}(\bar{x})
        \end{aligned}\\[0.5em]
        = {}&
        \int\limits_{X^{V(F_1) \setminus \bm{a_1}}}\vast(\, \int\limits_{X^{V(F_2) \setminus \bm{a_2}}}\!\prod_{ij \in E(F)} \!\!\! W(x_i, x_j)\cdot f(x_{b_{2,1}}, \ldots, x_{b_{2,\ell}})\,d\mu^{\otimes V(F_2) \setminus \bm{a_2}}(\bar{x})\vast)\,d\mu^{\otimes V(F_1) \setminus \bm{a_1}}(\bar{x})\\
        = {}&
        \int\limits_{{X^{V(F) \setminus \bm{a_1}}}}\!\prod_{{ij \in E(F)}} \!\!\! W(x_i, x_j)\cdot f(x_{b_{2,1}}, \ldots, x_{b_{2,\ell}})\,d\mu^{\otimes V(F) \setminus \bm{a_1}}(\bar{x})
    \end{align*}
    for every $f \in \LTwoProdL$
    and $\mu^{\otimes \bm{a_1}}$-almost all $x_{a_{1,1}}, \dots, x_{a_{1,k}} \in X$
    by the Tonelli-Fubini theorem.

    (\ref{le:graphOperations:product}):
    Let $\bm{F_1} = (F_1, \bm{a_1}, ())$,
    $\bm{F_2} = (F_2, \bm{a_2}, ())$, and
    $\bm{F_1} \cdot \bm{F_2} = (F, \bm{a_1}, ())$.
    In the following, we identify vertices
    $a_{1,1}, \dots, a_{1,k}$
    with $a_{2,1}, \dots, a_{2,k}$.
    Then, we have
    \begin{align*}
        &\left((\OperatorFromTo{\bm{F_1}}{W} c_1) \cdot (\OperatorFromTo{\bm{F_2}}{W} c_2) \right) (x_{a_{1,1}}, \dots, x_{a_{1,k}})\\
        = {}&
        \begin{aligned}[t]
            &\vast( \int\limits_{{X^{V(F_1) \setminus \bm{a_1}}}}\prod_{{ij \in E(F_1)}} \!\!\! W(x_i, x_j) \cdot c_1\,d\mu^{\otimes V(F_1) \setminus \bm{a_1}}(\bar{x})\vast)\\
            &\;\cdot \vast(
                \int\limits_{{X^{V(F_2) \setminus \bm{a_2}}}}\prod_{{ij \in E(F_2)}} \!\!\! W(x_i, x_j) \cdot c_2\,d\mu^{\otimes V(F_2) \setminus \bm{a_2}}(\bar{x})\vast)
        \end{aligned}
        \\
= {}&
        \int\limits_{{X^{V(F_1) \setminus \bm{a_1}}}}\int\limits_{X^{V(F_2) \setminus \bm{a_2}}}\prod_{{ij \in E(F_1)}} \!\!\! W(x_i, x_j) \cdot \prod_{{ij \in E(F_2)}} \!\!\! W(x_i, x_j)
            \cdot
            c_1 \cdot c_2\begin{aligned}[t]
            &\,d\mu^{\otimes V(F_2) \setminus \bm{a_2}}(\bar{x})\\
            &\;\,d\mu^{\otimes V(F_1) \setminus \bm{a_1}}(\bar{x})
            \end{aligned}
        \\
        = {}&
        \int\limits_{{X^{V(F) \setminus \bm{a_1}}}}\prod_{{ij \in E(F)}}\!\!\! W(x_i, x_j)\cdot c_1 \cdot c_2\,d\mu^{\otimes V(F) \setminus \bm{a_1}}(\bar{x})\\
        = {}& \OperatorFromTo{\bm{F_1} \blProd \bm{F_2}}{W} (c_1 \cdot c_2)(x_{a_{1,1}}, \dots, x_{a_{1,k}})
    \end{align*}
    for all $c_1, c_2 \in \R$
    and $\mu^{\otimes \bm{a_1}}$-almost all $x_{a_{1,1}}, \dots, x_{a_{1,k}} \in X$
    by the Tonelli-Fubini theorem.
\end{proof}

For a set $\mathcal{F} \subseteq \multiGraphs^{k,k}$,
every graphon $W \colon X \times X \to [0,1]$
induces a family of $L^\infty$-contractions
$\OperatorFamilyFromTo{\mathcal{F}}{W} \coloneqq (\OperatorFW)_{\bm{F} \in \mathcal{F}}$
on $\LTwoProdK$,
cf.\ \Cref{le:homOperatorWellDefined}.
When handling such families of operators,
we often use notation like
$\OperatorFamilyFromTo{\mathcal{F}}{W} \circ T$
for an $L^\infty$-contraction $T$
or $\OperatorFamilyFromTo{\mathcal{F}}{W} / \subAlg$
for $\subAlg \in \subAlgsk$
to denote the family obtained by applying the operation
to every operator in the family;
for these examples, we obtain the families
$(\OperatorFW \circ T)_{\bm{F} \in \mathcal{F}}$
and
$(\OperatorFW / \subAlg)_{\bm{F} \in \mathcal{F}}$.
Moreover, if the
graphs in $\mathcal{F}$
do not have any edges,
we again abbreviate
$\OperatorFamilyFrom{\mathcal{F}} \coloneqq (\OperatorFrom{\bm{F}})_{\bm{F} \in \mathcal{F}}$.
Recall that $\adjNeiGraphs^k$
is the set of all neighbor and adjacency graphs with $k$ input and output labels.
Let us finally define the family
\begin{equation*}
    \naturalKFamilyOperatorsShortW \coloneqq \OperatorFamilyFromTo{\adjNeiGraphs^k}{W}
\end{equation*}
that replaces the single operator $T_W$
in \Cref{th:kWLGraphons}, our characterization of oblivious $k$-WL.

Let us explore the connection between
the family $\naturalKFamilyOperatorsShortW$
and treewidth $k-1$ homomorphism functions:
Recall that the terms in $\treeClosure{\adjNeiGraphs^k}$
correspond to the
tree-decomposed multigraphs of treewidth at most $k-1$ by \Cref{le:graphsOfTerms}.
Given such a term $\term \in \treeClosure{\adjNeiGraphs^k}$,
we can use the correspondence of bi-labeled graph operations
to their operator counterparts, cf.\
\Cref{le:graphOperations},
to inductively compute the homomorphism function $\homFunction{\graphOfTerm}{W}$
of $\graphOfTerm$ in a graphon $W$
using the operators $\naturalKFamilyOperatorsShortW$.
Hence,
the operators in $\naturalKFamilyOperatorsShortW$
yield
all homomorphism functions of multigraphs
of treewidth at most $k-1$ in $W$.
An important part of
the proof of \Cref{th:kWLGraphons}
consists of defining different families
of $L^\infty$-contractions indexed by $\adjNeiGraphs^k$
that we may use instead of $\naturalKFamilyOperatorsShortW$
and still yield the same homomorphism functions.
For example, we may replace $\naturalKFamilyOperatorsShortW$
by the quotient operators $\naturalKFamilyOperatorsShortW / \subAlg$
for an appropriate $\subAlg \in \subAlgsk$.
This leads to the following general definition, where we consider
families of operators on a space $\LTwo$ where $(X, \Borel, \mu)$
may not necessarily be a product space.
For example, it could be a quotient space.

\begin{definition}
    \label{def:homFunctionFamily}
    Let $(Y, \mathcal{D})$ be
    a standard Borel space
    and $\nu$ be a Borel probability measure on $Y$.
    Let $k \ge 1$
    and
    $\TT = (T_{\bm{F}})_{\bm{F} \in \mathcal{F}}$ be a family of $L^\infty$-contractions
    on $\LTwoY$
    indexed by a set $\mathcal{F} \subseteq \multiGraphs^{k,k}$.
    For every term $\term \in \treeClosure{\mathcal{F}}$,
    the \textit{homomorphism function} of $\term$ in $\TT$
    is the function
    $\homFunctionTOp \in \LInftyY$ with $\normI{\homFunctionTOp} \le 1$
    defined inductively by
    \begin{enumerate}
        \item
            $\homFunctionTOp \coloneqq \allOne_{Y}$
            for $\term = \identityGraph$,
        \item
            $\homFunctionTOp \coloneqq T_{\bm{F}} \homFunctionOp{\term'}$
            for $\term = \bm{F} \circ \term'$, where $\bm{F} \in \mathcal{F}$, and
        \item
            $\homFunctionTOp \coloneqq \homFunctionOp{\term_1} \cdot \homFunctionOp{\term_2}$
            for $\term = \term_1 \blProd \term_2$.
    \end{enumerate}
    Moreover, the \textit{homomorphism density} of $\term$ in $\TT$
    is defined as
$t(\term, \TT) \coloneqq \langle \allOne_Y, \homFunctionTOp \rangle$.
\end{definition}

As remarked above,
given a term $\term \in \treeClosure{\adjNeiGraphs^k}$,
we can use the correspondence of bi-labeled graph operations
to their operator counterparts
to inductively compute the homomorphism function $\homFunction{\graphOfTerm}{W}$
and, in particular, the homomorphism density $t(\graphOfTerm, W)$
of $\graphOfTerm$ in a graphon $W$
using the operators in $\naturalKFamilyOperatorsShortW$.

\begin{lemma}
    \label{le:operatorsPreserveHomomorphisms}
    Let $k \ge 1$.
    Let $W \colon X \times X \to [0,1]$ be a graphon.
    Then,
    \begin{align*}
        &\homFunctionT{\naturalKFamilyOperatorsShortW} = \homFunction{\graphOfTerm}{W}&
        &\text{and}&
        &t(\term, \naturalKFamilyOperatorsShortW) = t(\graphOfTerm, W)&
    \end{align*}
    for every $\term \in \treeClosureK$.
\end{lemma}
\begin{proof}
    $\naturalKFamilyOperatorsShortW$
    is a family of $L^\infty$-contractions
    on $\LTwoProdK$.
    We show that
    $\homFunctionT{\naturalKFamilyOperatorsShortW} = \homFunction{\graphOfTerm}{W}$
    by induction on $\term \in \treeClosureK$.
    Then,
    \begin{equation*}
        t(\term, \naturalKFamilyOperatorsShortW)
        = \langle \allOne_{X^k}, \homFunctionT{\naturalKFamilyOperatorsShortW} \rangle
        = \langle \allOne_{X^k}, \homFunction{\graphOfTerm}{W} \rangle
        = t(\graphOfTerm, W)
    \end{equation*}
    by definition of $t(\term, \naturalKFamilyOperatorsShortW)$
    and $\homFunction{\graphOfTerm}{W}$.

    For the induction basis $\term = \identityGraph$, we have
    $\homFunction{\term}{\naturalKFamilyOperatorsShortW} = \allOne_{X^k}$
    by \Cref{def:homFunctionFamily}
    and
    $\homFunction{\graphOfTerm}{W}
    = \OperatorFromTo{\identityGraph}{W} \allOne_{X^0}
    = \allOne_{X^k}$
    by the definition of
    $\OperatorFromTo{\identityGraph}{W}$.
    For the first case of the inductive step
    $\term = \bm{F} \circ \term'$,
    where $\bm{F} \in \adjNeiGraphs^k$
    and $\graphOf{\term'} \in \multiGraphs^{k, 0}$,
    we have
    \begin{align*}
        \homFunctionT{\naturalKFamilyOperatorsShortW}
        = \OperatorFW \homFunction{\term'}{\naturalKFamilyOperatorsShortW}
        = \OperatorFW \homFunction{\graphOf{\term'}}{W}
        &= \OperatorFW (\OperatorFromTo{\graphOf{\term'}}{W} \allOne_{X^0})\\
        &= \OperatorFromTo{\bm{F} \circ \graphOf{\term'}}{W} \allOne_{X^0}\\
        &= \OperatorFromTo{\graphOf{\term}}{W} \allOne_{X^0}\\
        &= \homFunction{\graphOfTerm}{W}
    \end{align*}
    by \Cref{def:homFunctionFamily},
    the induction hypothesis,
    the definition of $\OperatorFromTo{\graphOf{\term}}{W}$,
    and
    \Cref{le:graphOperations} (\ref{le:graphOperations:composition}).
    For the second case of the inductive step
    $\term = \term_1 \blProd \term_2$,
    where $\graphOf{\term}, \graphOf{\term_1}, \graphOf{\term_2} \in \multiGraphs^{k, 0}$.
    Then, we have
    \begin{align*}
        \homFunctionT{\naturalKFamilyOperatorsShortW}
        = \homFunction{\term_1}{\naturalKFamilyOperatorsShortW} \cdot \homFunction{\term_2}{\naturalKFamilyOperatorsShortW}
        = \homFunction{\graphOf{\term_1}}{W} \cdot \homFunction{\graphOf{\term_2}}{W}
        &= (\OperatorFromTo{\graphOf{\term_1}}{W} \allOne_{X^{0}}) \cdot (\OperatorFromTo{\graphOf{\term_2}}{W} \allOne_{X^{0}})\\
        &= \OperatorFromTo{\graphOf{\term_1} \cdot \graphOf{\term_2}}{W} \allOne_{X^{0}}\\
        &= \OperatorFromTo{\graphOf{\term}}{W} \allOne_{X^{0}}\\
        &= \homFunction{\graphOfTerm}{W}
    \end{align*}
    by \Cref{def:homFunctionFamily},
    the induction hypothesis,
    the definition of $\OperatorFromTo{\graphOf{\term}}{W}$, and
    \Cref{le:graphOperations} (\ref{le:graphOperations:product}).
\end{proof}

The following lemma gives a sufficient condition under
which
two families of $L^\infty$- contractions
yield the same homomorphism densities.
Recall that a Markov embedding
is a Markov operator that is an isometry.
Unlike Markov operators in general,
Markov embeddings are compatible with
point-wise products of functions,
cf.\ \cite[Theorem $13.9$, Remark $13.10$]{EisnerEtAl2015}.
This is crucial since we need the point-wise product
of functions to get from bounded pathwidth to bounded treewidth
homomorphism functions.

\begin{lemma}
    \label{le:markovEmbeddingPreservesHomomorphisms}
    Let $k \ge 1$.
    Let $(X_1, \Borel_1)$ and $(X_2, \Borel_2)$ be standard Borel spaces
    with Borel probability measures $\mu_1$ and $\mu_2$ on $X_1$ and $X_2$, respectively.
    Let $\TT_1$ and $\TT_2$ be families of $L^\infty$-contractions
    on $\LTwoOne$ and $\LTwoTwo$, respectively, indexed by $\adjNeiGraphs^k$.
    If $I \colon \LTwoTwo \to \LTwoOne$ is a Markov embedding
    such that $\TT_1 \circ I = I \circ \TT_2$, then
    \begin{align*}
        &I \homFunctionT{\TT_2} = \homFunctionT{\TT_1}&
        &\text{and}&
        &t(\term, \TT_1) = t(\term, \TT_2)&
    \end{align*}
    for every $\term \in \treeClosureK$.
\end{lemma}
\begin{proof}
    We show that
$I \homFunctionT{\TT_1} = \homFunctionT{\TT_2}$
by induction on $\term \in \treeClosureK$.
    Then, also
    \begin{equation*}
        t(\term, \TT_1)= \langle \allOne_{X_1}, \homFunctionT{\TT_1} \rangle = \langle \allOne_{X_1}, I \homFunctionT{\TT_2} \rangle = \langle I^* \allOne_{X_1}, \homFunctionT{\TT_2} \rangle = \langle \allOne_{X_2}, \homFunctionT{\TT_2} \rangle = t(\term, \TT_2).
    \end{equation*}
    For the induction basis $\term = \identityGraph$, we have
    \begin{equation*}
        I \homFunctionT{\TT_2}= I \allOne_{X_2}= \allOne_{X_1}= \homFunctionT{\TT_1}.
    \end{equation*}
    For $\term = \bm{F} \circ \term'$, where $\bm{F} \in \adjNeiGraphs^k$,
    we have
    \begin{equation*}
        I \homFunctionT{\TT_2}= (I \circ (\TT_2)_{\bm{F}}) \homFunction{\term'}{\TT_2}= ((\TT_1)_{\bm{F}} \circ I) \homFunction{\term'}{\TT_2}= (\TT_1)_{\bm{F}} \homFunction{\term'}{\TT_1}= \homFunctionT{\TT_1}
    \end{equation*}
    by the assumption and the induction hypothesis.
    Finally, for $\term = \term_1 \blProd \term_2$,
    we use that $I$ is a Markov embedding and, hence,
    satisfies $I(f \cdot g) = I f \cdot I g$
    for all $f, g \in \LInftyTwo$ \cite[Theorem $13.9$]{EisnerEtAl2015}.
    We have
    \begin{align*}
        I \homFunctionT{\TT_2}= I (\homFunction{\term_1}{\TT_2} \cdot \homFunction{\term_2}{\TT_2})= I \homFunction{\term_1}{\TT_2} \cdot I \homFunction{\term_2}{\TT_2}= \homFunction{\term_1}{\TT_1} \cdot \homFunction{\term_2}{\TT_1}= \homFunction{\term}{\TT_1}
    \end{align*}
    by the induction hypothesis.
\end{proof}

An important application of \Cref{le:markovEmbeddingPreservesHomomorphisms}
is to replace the family
$\naturalKFamilyOperatorsShortW$
by the quotient operators $\naturalKFamilyOperatorsShortW / \subAlg$
for an appropriate $\subAlg \in \subAlgsk$.
To this end, we call $\subAlg \in \subAlgsk$
\textit{$W$-invariant}
if $\subAlg$ is invariant for every operator in the family $\naturalKFamilyOperatorsShortW$,
i.e.,
$\subAlg$ is $\OperatorFromTo{\bm{F}}{W}$-invariant
for every $\bm{F} \in \adjNeiGraphs^k$,
which in turn means that
$\OperatorFromTo{\bm{F}}{W}(\LTwoProdKOf{\subAlg})
\subseteq \LTwoProdKOf{\subAlg}$
for every $\bm{F} \in \adjNeiGraphs^k$.
\begin{corollary}
    \label{le:invariantQuotientPreservesHomomorphisms}
    Let $k \ge 1$.
    Let $W \colon X \times X \to [0,1]$ be a graphon
    and $\subAlg \in \subAlgsk$ be $W$-invariant.
    Then,
    \begin{equation*}
t(\term, (\naturalKFamilyOperatorsShortW)_\subAlg)= t(\term, \naturalKFamilyOperatorsShortW \!/ \subAlg)= t(\term, \naturalKFamilyOperatorsShortW)= t(\graphOfTerm, W)
    \end{equation*}
    for every $\term \in \treeClosureK$.
\end{corollary}
\begin{proof}
    The last equality is just \Cref{le:operatorsPreserveHomomorphisms}.
    By \Cref{le:quoOperator} (\ref{le:quoOperator:expOperatorI}) and (\ref{le:quoOperator:invariantI}),
    we have
    $I_\subAlg \circ \naturalKFamilyOperatorsShortW \!/ \subAlg = (\naturalKFamilyOperatorsShortW)_\subAlg \circ I_\subAlg$
    and
    $I_\subAlg \circ \naturalKFamilyOperatorsShortW \!/ \subAlg = \naturalKFamilyOperatorsShortW \circ I_\subAlg$,
    respectively,
    where $I_\subAlg$ is a Markov embedding by
    \Cref{th:quotientSpaces} (\ref{th:quotientSpaces:isometry}),
    Therefore, \Cref{le:markovEmbeddingPreservesHomomorphisms}
    yields the first two equalities.
\end{proof}

\section{Weisfeiler-Leman and Graphons}
\label{sec:WL}

In \Cref{sec:minInv} to \Cref{sec:WLIndistinguishability},
we closely follow Greb\'ik and Rocha \cite{GrebikRocha2021}
to prove \Cref{th:kWLGraphons} and formally
define all notions appearing in it.
We start off by defining
the minimum $W$-invariant $\mu^{\otimes k}$-relatively complete
sub-$\sigma$-algebra $\subAlg^k_W$ of $\Borel^{\otimes k}$
for a graphon $W$
via the family of operators $\naturalKFamilyOperatorsShortW$.
in \Cref{sec:minInv}.
Then, in \Cref{subsec:degreeMeasures}, we
define the space $\Mk$,
i.e., the space of all colors used by oblivious $k$-WL,
and \textit{$k$-WL distributions},
which generalize multisets of colors.
In \Cref{subsec:mappingOwl},
we define the function $\owlk_W \colon X^k \to \Mk$
and the $k$-WL distribution $\nu^k_W$ for a graphon $W$.
In \Cref{sec:operatorsOnMeasures}, we deviate
from Greb\'ik and Rocha \cite{GrebikRocha2021}
by a larger margin:
They show that every
distribution on iterative degree measures $\nu$ defines a
graphon on the space $\MM$;
this graphon for $\nu_W$
is then isomorphic to the quotient graphon $W/\subAlg_W$.
Since the operators in $\naturalKFamilyOperatorsShortW$
are not integral operators defined by a graphon
(intuitively, these graphons would have to be non-zero
only on the diagonal, which has measure zero),
we take the different route of
showing that a $k$-WL distribution $\nu$
defines a family of operators $\TT_\nu$
on $\LTwoMkNu$;
the family $\TT_{\nu^k_W}$
then corresponds to $\naturalKFamilyOperatorsShortW$.
In \Cref{sec:WLIndistinguishability},
we define the set $\Tk$ of \emph{homomorphism functions} on $\Mk$
and use the Stone-Weierstrass Theorem to show that it
is dense in $C(\Mk)$
before we finally prove \Cref{th:kWLGraphons} in
\Cref{sec:mainProof}.
The remaining sections discuss some implications of \Cref{th:kWLGraphons}:
\Cref{sec:measureHierarchies} shows that one can combine
all $k$-WL distributions $\nu^1_W, \nu^2_W, \dots$
of a graphon $W$
into a single distribution to obtain a new characterization of
weak isomorphism.
\Cref{sec:opHierarchies} explains how
the characterization of
\Cref{th:kWLGraphons} using Markov operators
corresponds to the system $\Lkk$ of linear equations.

\subsection{The Minimum \texorpdfstring{$W$}{W}-Invariant Sub-\texorpdfstring{$\sigma$}{Sigma}-Algebra}
\label{sec:minInv}

For a family $\TT = (T_i)_{i \in I}$ of operators
$T_i \colon \LTwo \to \LTwo$, where $i \in I$,
and a sub-$\sigma$-algebra $\subAlg \in \subAlgs$,
define
\begin{equation*}
    \TT(\subAlg) \coloneqq \bigcap \Set{\subAlgD \in \subAlgs \mid \subAlgD \supseteq \subAlg \text{ and } T_i(\LTwoSub) \subseteq \LTwoSubD \text{ for every } i \in I}.
\end{equation*}
Then, $\TT(\subAlg) \in \subAlgs$, cf.\ \Cref{subsec:algebras},
and $\subAlg$ is called \textit{$\TT$-invariant} if
$\TT(\subAlg) \subseteq \subAlg$,
which is equivalent to requiring that
$\subAlg$ is $T_i$-invariant for every $i \in I$.
Note that this operation is monotonous, i.e.,
for all $\subAlg, \subAlgD \in \subAlgs$
with $\subAlg \subseteq \subAlgD$, we have
$\TT(\subAlg) \subseteq \TT(\subAlgD)$.
By definition, the family
$\naturalKFamilyOperatorsShortW$
consists of the operators from the two families
$\OperatorFamilyFromTo{\adjGraphs^k}{W}$
and
$\OperatorFamilyFrom{\jNeighbors^k}$.
The following definition uses these two individual
families to define the
sub-$\sigma$-algebra $\kInftySubAlg_W$ of $\Borel^{\otimes k}$.
Already at this point, one should notice the connection
to oblivious $k$-WL, cf.\ \Cref{sec:kWLGraphons}:
the operators in $\OperatorFamilyFromTo{\adjGraphs^k}{W}$
capture the concept of atomic types
while the operators in $\OperatorFamilyFrom{\jNeighbors^k}$
correspond to the refinement rounds via $j$-neighbors used in oblivious $k$-WL.

\begin{definition}
    Let $k \ge 1$
    and $W \colon X \times X \to [0,1]$ be a graphon.
    Define $\kSubAlg_{W, n} \in \subAlgsk$ for every $n \in \N$ by setting
    $\kSubAlg_{W,0} \coloneqq \OperatorFamilyFromTo{\adjGraphs^k}{W}(\left\langle \Set{\emptyset, X^k} \right\rangle)$,
    $\kSubAlg_{W, n + 1} \coloneqq \OperatorFamilyFrom{\jNeighbors^k}(\kSubAlg_{W, n}) \text{ for every } n \in \N$,
    and
    $\kInftySubAlg_W \coloneqq \kInftySubAlg_{W, \infty} \coloneqq \langle \bigcup_{n \in \N} \kSubAlg_{W, n} \rangle$.
\end{definition}

Verifying that $\kInftySubAlg_W$ is in fact the minimum
$W$-invariant
$\mu^{\otimes k}$-relatively complete sub-$\sigma$-algebra
of $\Borel^{\otimes k}$
is mostly analogous to \cite[Proposition $5.13$]{GrebikRocha2021}.
A difference is given by the operators in $\OperatorFamilyFromTo{\adjGraphs^k}{W}$,
which are \emph{multiplication operators}, i.e.,
they multiply their arguments with a fixed function. This implies that
a single initial application guarantees $\OperatorFamilyFromTo{\adjGraphs^k}{W}$-invariance
for all subsequent sub-$\sigma$-algebras in the sequence.
Moreover, we also verify that $\kInftySubAlg_W$ is permutation invariant, i.e.,
$\kInftySubAlg_W$ is $T_{\pi}$-invariant for every permutation $\pi \colon [k] \to [k]$.
\begin{lemma}
    \label{le:algebraSequence}
    \label{le:minimumTA}
    \label{le:kSubAlgInduction}
    Let $k \ge 1$
    and $W \colon X \times X \to [0,1]$ be a graphon.
    Then,
    \begin{enumerate}
        \item $\kSubAlg_{W,0} = \big\langle \bigcup_{\bm{A} \in \adjGraphs^k} \Set{(\OperatorFromTo{\bm{A}}{W} \allOne_{X^k})^{-1}(A) \mid A \in \Borel([0,1])} \big\rangle$,\label{le:algebraSequence:zeroAltDef}
        \item $\kSubAlg_{W,0}$ is the minimum $\OperatorFamilyFromTo{\adjGraphs^k}{W}$-invariant $\mu^{\otimes k}$-relatively complete sub-$\sigma$-algebra of $\Borel^{\otimes k}$,\label{le:algebraSequence:minimumTA}
        \item $\kSubAlg_{W, n+1} = \big\langle \kSubAlg_{W,n} \cup \bigcup_{\bm{N} \in \jNeighbors^k} \Set{(\OperatorN \allOne_A)^{-1}(B) \mid A \in \kSubAlg_{W,n}, B \in \Borel([0,1])} \big\rangle$ for every $n \in \N$,\label{le:algebraSequence:nAltDef}
        \item $\kSubAlg_{W,n}$ is $\OperatorFamilyFromTo{\adjGraphs^k}{W}$-invariant for every $n \in \N \cup \Set{\infty}$,\label{le:algebraSequence:TAInvariant}
        \item $\kInftySubAlg_W$ is the minimum $W$-invariant $\mu^{\otimes k}$-relatively complete sub-$\sigma$-algebra of $\Borel^{\otimes k}$, and \label{le:algebraSequence:TInvariant}
        \item $\kSubAlg_{W, n}$ is permutation invariant for every $n \in \N \cup \Set{\infty}$. \label{le:algebraSequence:permutation}
    \end{enumerate}
\end{lemma}
\begin{proof}
    (\ref{le:algebraSequence:zeroAltDef}) and (\ref{le:algebraSequence:minimumTA}):
    Let $\subAlg$ denote the minimum $\OperatorFamilyFromTo{\adjGraphs^k}{W}$-invariant $\mu^{\otimes k}$-relatively complete sub-$\sigma$-algebra of $\Borel^{\otimes k}$,
    and let $\subAlgD$ denote the $\mu^{\otimes k}$-relatively complete sub-$\sigma$-algebra of $\Borel^{\otimes k}$ on the right-hand side
    of the equality in (\ref{le:algebraSequence:zeroAltDef}).
    We prove that $\subAlg = \subAlgD = \kSubAlg_{W,0}$.
    We start by proving $\subAlg \subseteq \subAlgD$.
    Let $\bm{A} \in \adjGraphs^k$.
    The function $\OperatorFromTo{\bm{A}}{W} \allOne_{X^k}$
    is $\subAlgD$-measurable by definition of $\subAlgD$.
    Hence, for a $\subAlgD$-measurable function $g \in \LTwoProdK$,
    the product $(\OperatorFromTo{\bm{A}}{W} \allOne_{X^k}) \cdot g = \OperatorFromTo{\bm{A}}{W} g$
    is again $\subAlgD$-measurable, where the equality holds since $\OperatorFromTo{\bm{A}}{W}$
    is a multiplication operator.
    That is, $\subAlgD$ is $\OperatorFamilyFromTo{\adjGraphs^k}{W}$-invariant,
    which yields $\subAlg \subseteq \subAlgD$.
    For the inclusion $\subAlgD \subseteq \subAlg$
    on the other hand, $\allOne_{X^k}$ is trivially $\subAlg$-measurable
    and, since $\subAlg$ is
    $\OperatorFamilyFromTo{\adjGraphs^k}{W}$-invariant,
    the function $\OperatorFromTo{\bm{A}}{W} \allOne_{X^k}$
    is $\subAlg$-measurable for every $\bm{A} \in \adjGraphs^k$.
    Hence, $\subAlgD \subseteq \subAlg$.
    We have established $\subAlg = \subAlgD$
    and it remains to prove that these are also equal to $\kSubAlg_{W,0}$.
    We have $\left\langle \Set{\emptyset, X^k} \right\rangle \subseteq \subAlg$
    and, hence, $\kSubAlg_{W,0} = \OperatorFamilyFromTo{\adjGraphs^k}{W}(\left\langle \Set{\emptyset, X^k} \right\rangle) \subseteq \OperatorFamilyFromTo{\adjGraphs^k}{W}(\subAlg) \subseteq \subAlg$.
    On the other hand,
    for every $\bm{A} \in \adjGraphs^k$,
    the function $\OperatorFromTo{\bm{A}}{W} \allOne_{X^k}$
    is $\kSubAlg_{W,0}$-measurable.
    Hence, $\subAlgD \subseteq \kSubAlg_{W,0}$.

    (\ref{le:algebraSequence:nAltDef}):
    Let $\subAlgD$ denote the $\mu^{\otimes k}$-relatively complete sub-$\sigma$-algebra of $\Borel^{\otimes k}$ from (\ref{le:algebraSequence:nAltDef}),
    i.e.,
    $\subAlgD$ is the minimum
    $\mu^{\otimes k}$-relatively complete sub-$\sigma$-algebra of $\Borel^{\otimes k}$
    that contains $\kSubAlg_{W,n}$
    and makes the maps $\OperatorN \allOne_A$ for $\bm{N} \in \jNeighbors^k$
    and $A \in \kSubAlg_{W,n}$ measurable.
    It is easy to see that $\subAlgD \subseteq \kSubAlg_{W, n+1}$:
    We have $\kSubAlg_{W, n} \subseteq \kSubAlg_{W, n+1}$ by definition of $\kSubAlg_{W, n+1}$.
    Moreover,
    for $\bm{N} \in \jNeighbors^k$ and $A \in \Borel(\kSubAlg_{W,n})$,
    the function $\allOne_A$ is $\kSubAlg_{W,n}$-measurable and, hence
    by definition of $\kSubAlg_{W, n + 1}$, the function
    $\OperatorN \allOne_A$ is then $\kSubAlg_{W, n+1}$ measurable.
    It remains to prove that $\kSubAlg_{W, n + 1} \subseteq \subAlgD$, i.e.,
    that $\kSubAlg_{W,n} \subseteq \subAlgD$ and
    $\OperatorN(\LTwoProdKOf{\kSubAlg_{W, n}}) \subseteq \LTwoProdKOf{\subAlgD}$ for every $\bm{N} \in \jNeighbors^k$.
    We have $\kSubAlg_{W,n} \subseteq \subAlgD$ by definition of $\subAlgD$.
    Let $\bm{N} \in \jNeighbors^k$.
    We have
    $\OperatorN \allOne_{A} \in L^2(X^k, \subAlgD, \mu^{\otimes k})$
    for $A \in \kSubAlg_{W,n}$ by definition of $\subAlgD$.
    Since the linear hull of $\Set{\allOne_A}_{A \in \kSubAlg_{W,n}}$
    is dense in subspace $L^2(X^k, \kSubAlg_{W,n}, \mu^{\otimes k})$
    and since $\LTwoProdKOf{\subAlgD}$ is closed,
    linearity and continuity of $\OperatorN$
    then yield that
    $\OperatorN(\LTwoProdKOf{\kSubAlg_{W, n}}) \subseteq \LTwoProdKOf{\subAlgD}$.

    (\ref{le:algebraSequence:TAInvariant}):
    Let $n \in \N \cup \Set{\infty}$ and $\bm{A} \in \adjGraphs^k$.
    We have $\kSubAlg_{W,0} \subseteq \kSubAlg_{W,n}$, which means that
    the function $\OperatorFromTo{\bm{A}}{W} \allOne_{X^k}$
    is $\kSubAlg_{W,n}$-measurable.
    Then, the claim follows as $\OperatorFromTo{\bm{A}}{W}$
    is a multiplication operator, cf.\ the proof of
    (\ref{le:algebraSequence:zeroAltDef}) and (\ref{le:algebraSequence:minimumTA}).

    (\ref{le:algebraSequence:TInvariant}):
    We first show that $\kInftySubAlg_W \subseteq \subAlg$ for every
    $\naturalKFamilyOperatorsShortW$-invariant sub-$\sigma$-algebra $\subAlg \in \subAlgsk$.
    We have $\left\langle \Set{\emptyset, X^k} \right\rangle \subseteq \subAlg$
    and, hence, $\kSubAlg_{W,0} = \OperatorFamilyFromTo{\adjGraphs^k}{W}(\left\langle \Set{\emptyset, X^k} \right\rangle) \subseteq \OperatorFamilyFromTo{\adjGraphs^k}{W}(\subAlg) \subseteq \subAlg$.
    From there on, induction yields
    $\kSubAlg_{W, n + 1} = \OperatorFamilyFrom{\jNeighbors^k}(\kSubAlg_{W, n})
    \subseteq \OperatorFamilyFrom{\jNeighbors^k}(\subAlg)
    \subseteq \subAlg$
    for every $n \in \N$.
    Hence, $\kInftySubAlg_W \subseteq \subAlg$.

    It remains to prove that $\kInftySubAlg_W$ is $\naturalKFamilyOperatorsShortW$-invariant.
    By (\ref{le:algebraSequence:TAInvariant}), it suffices to show that
    that $\kInftySubAlg_W$ is $\OperatorN$-invariant for $\bm{N} \in \jNeighbors^k$.
    This is essentially Proposition~5.13 of \cite{GrebikRocha2021}:
    We first show that
    $\OperatorN \allOne_{A} \in L^2(X^k, \kInftySubAlg_W, \mu^{\otimes k})$
    for $A \in \kInftySubAlg_W$.
    To this end, note that
    $\bigcup_{n \in \N} \kSubAlg_{W, n}$ is an algebra and the $\sigma$-algebra
    generated by it is $\kInftySubAlg_W$.
    Hence, from \cite[Theorem $3.1.10$]{Dudley2002}, it easily follows
    that we can approximate every set in $\kInftySubAlg_W$
    by a set in $\bigcup_{n \in \N} \kSubAlg_{W, n}$
    w.r.t.\ the measure of their symmetric difference.
    This implies that, for every $A \in \kInftySubAlg_W$,
    there is a sequence $(A_n)_{n \in \N}$ with $A_n \in \kSubAlg_{W,n}$
    such that $\allOne_{A_n} \rightarrow \allOne_A$ in $\LTwoProdK$.
    Let $\bm{N} \in \jNeighbors^k$.
    By continuity of $\OperatorN$,
    we have $\OperatorN \allOne_{A_n} \rightarrow \OperatorN \allOne_A$.
    Note that, for $n \in \N$, we have
    $\OperatorN \allOne_{A_n} \in L^2(X^k, \kSubAlg_{W, n + 1}, \mu^{\otimes k})
    \subseteq L^2(X^k, \kInftySubAlg_W, \mu^{\otimes k})$,
    which is a closed subspace by \Cref{cl:conditionalExpectation}.
    Hence, $\OperatorN \allOne_A \in L^2(X^k, \kInftySubAlg_W, \mu^{\otimes k})$.
    Since the linear hull of $\Set{\allOne_A}_{A \in \kInftySubAlg_W}$
    is dense in the closed subspace $L^2(X^k, \kInftySubAlg_W, \mu^{\otimes k})$,
    linearity and continuity of $\OperatorN$
    then yields that $L^2(X^k, \kInftySubAlg_W, \mu^{\otimes k})$
    is $\OperatorN$-invariant.

    (\ref{le:algebraSequence:permutation}):
    First, recall that $\Borel^{\otimes k}$ is permutation invariant.
    Moreover, if $\subAlg \in \subAlgsk$, then $\pi(\subAlg) \in \subAlgsk$
    for every permutation $\pi \colon [k] \to [k]$.
    This implies that, if $\mathcal{X} \subseteq \Borel^{\otimes k}$
    is a set with $\pi(\mathcal{X}) \subseteq \mathcal{X}$
    for every permutation $\pi \colon [k] \to [k]$,
    then $\langle \mathcal{X} \rangle$ is permutation invariant.
    Hence, $\langle \Set{\emptyset, X^k} \rangle$ is permutation invariant,
    and it suffices to show that,
    for a permutation-invariant sub-$\sigma$-algebra $\subAlg \in \subAlgsk$,
    both $\OperatorFamilyFromTo{\adjGraphs^k}{W}(\subAlg)$
    and $\OperatorFamilyFrom{\jNeighbors^k}(\subAlg)$
    are permutation-invariant.
    Then, induction yields that $\kSubAlg_{W, n}$ is permutation invariant
    for every $n \in \N$ and, hence, also $\kSubAlg_W$ since
    $\pi\left(\bigcup_{n \in \N} \kSubAlg_{W, n}\right)
    = \bigcup_{n \in \N} \pi(\kSubAlg_{W, n})
    \subseteq \bigcup_{n \in \N} \kSubAlg_{W, n}$
    for every permutation $\pi \colon [k] \to [k]$.

    It remains to show that,
    for a permutation-invariant sub-$\sigma$-algebra $\subAlg \in \subAlgsk$,
    both $\OperatorFamilyFromTo{\adjGraphs^k}{W}(\subAlg)$
    and $\OperatorFamilyFrom{\jNeighbors^k}(\subAlg)$
    are permutation-invariant.
    We prove the statement for $\OperatorFamilyFromTo{\adjGraphs^k}{W}(\subAlg)$;
    the proof for $\OperatorFamilyFrom{\jNeighbors^k}(\subAlg)$ is analogous.
    To this end, we show that, for an arbitrary sub-$\sigma$-algebra $\subAlg \in \subAlgsk$,
    we have
    \begin{equation}
        \pi(\OperatorFamilyFromTo{\adjGraphs^k}{W}(\subAlg))
        = \OperatorFamilyFromTo{\adjGraphs^k}{W}(\pi(\subAlg))
        \label{eq:adjOperatorFamilyPermInvariant}
    \end{equation}
    for every permutation $\pi \colon [k] \to [k]$.
    Then, if $\subAlg$ is permutation invariant,
    we get
$\pi(\OperatorFamilyFromTo{\adjGraphs^k}{W}(\subAlg))
        = \OperatorFamilyFromTo{\adjGraphs^k}{W}(\pi(\subAlg))
        = \OperatorFamilyFromTo{\adjGraphs^k}{W}(\subAlg)$
for every permutation $\pi \colon [k] \to [k]$.

    To prove \Cref{eq:adjOperatorFamilyPermInvariant}, let $\pi \colon [k] \to [k]$
    be a permutation
    and observe that
$T_\pi \circ \OperatorFromTo{\ijAdjacencyij}{W} \circ T_{\pi^{-1}}
        = \OperatorFromTo{\adjacencyGraphOf{k}{\pi(i)\pi(j)}}{W}$
for all $i \neq j \in [k]$.
    As a side note, the analogous observation for
    $\OperatorFamilyFrom{\jNeighbors^k}(\subAlg)$
    is
    $T_\pi \circ \OperatorFromTo{\jNeighborj}{W} \circ T_{\pi^{-1}}
    = \OperatorFromTo{\neighborGraphOf{k}{\pi(j)}}{W}$ for every $j \in [k]$.
    We get that
    \begin{align*}
        \OperatorFromTo{\adjacencyGraphOf{k}{ij}}{W}(L^2(X^k, \pi(\subAlg), \mu^{\otimes k}))
        &= \OperatorFromTo{\adjacencyGraphOf{k}{ij}}{W}(T_{\pi^{-1}}(L^2(X^k, \subAlg, \mu^{\otimes k})))\\
        &= T_{\pi^{-1}}(\OperatorFromTo{\adjacencyGraphOf{k}{\pi(i)\pi(j)}}{W}(L^2(X^k, \subAlg, \mu^{\otimes k}))).
    \end{align*}
    Let $\subAlgD \in \subAlgsk$. Then, we have
    \begin{align*}
        &\OperatorFromTo{\ijAdjacencyij}{W}(L^2(X^k, \pi(\subAlg), \mu^{\otimes k}))
        \subseteq L^2(X^k, \subAlgD, \mu^{\otimes k})\\
        \iff {} &
        \OperatorFromTo{\adjacencyGraphOf{k}{\pi(i)\pi(j)}}{W}(L^2(X^k, \subAlg, \mu^{\otimes k}))
        \subseteq T_\pi(L^2(X^k, \subAlgD, \mu^{\otimes k}))\\
        \iff {} &
        \OperatorFromTo{\adjacencyGraphOf{k}{\pi(i)\pi(j)}}{W}(L^2(X^k, \subAlg, \mu^{\otimes k}))
        \subseteq L^2(X^k, \pi^{-1}(\subAlgD), \mu^{\otimes k}).
    \end{align*}
    As the mapping $\ijAdjacencyij \mapsto \adjacencyGraphOf{k}{\pi(i)\pi(j)}$
    is a permutation of $\adjGraphs^k$ and
    we also have
    $\subAlgD \supseteq \pi(\subAlg)
    \iff \pi^{-1}(\subAlgD) \supseteq \subAlg$,
    this implies \Cref{eq:adjOperatorFamilyPermInvariant}.
\end{proof}

\subsection{Weisfeiler-Leman Measures and Distributions}
\label{subsec:degreeMeasures}

Before defining the mapping $\owlk_W \colon X^k \to \Mk$,
we have to define the space $\Mk$,
which can be seen as the space of all colors
used by oblivious $k$-WL.
We state some facts regarding spaces of measures first:
For a separable metrizable space $(X, \Topology)$,
let $\probMeas(X)$ denote the set of all Borel
probability measures on $X$.
Let $C_b(X)$ denote the set of bounded continuous
real-valued functions on $X$.
We endow $\probMeas(X)$
with the topology generated by the maps
$\mu \mapsto \int f \dmu$ for $f \in C_b(X)$.
This leads to the notion of \emph{weak convergence of measures}
of which the \textit{Portmanteau Theorem} gives many equivalent
characterizations \cite[Theorem $17.20$]{Kechris1995}.
We only use that, for $(\mu_i)_{i \in \N}$ with $\mu_i \in \probMeas(X)$
and $\mu \in \probMeas(X)$,
we have $\mu_i \rightarrow \mu$
if and only if
\begin{equation*}
    \int f d \mu_i \rightarrow \int f d \mu
\end{equation*}
for every $f \in C_b(X)$,
where we may replace $C_b(X)$ by a dense subset,
i.e., a subset that
is dense for $d_{\sup}(f, g) \coloneqq \sup \lvert f - g \rvert$.
If $(X, \Topology)$ is compact, which is the case for the spaces we define,
then $C_b(X) = C(X)$,
where $C(X)$
denotes the set of continuous real-valued functions on $X$.
The Borel $\sigma$-algebra $\Borel(\probMeas(X))$
is generated by
the maps $\mu \mapsto \mu(A)$ for $A \in \Borel(X)$
and also by the maps $\mu \mapsto \int f \dmu$
for bounded Borel real-valued functions $f$ \cite[Theorem 17.24]{Kechris1995}.
If $(X, \Topology)$ is Polish, then so is $P(X)$ \cite[Theorem 17.23]{Kechris1995},
which means that,
if $(X, \Borel)$ is a standard Borel space,
then so is
$(\probMeas(X), \Borel(\probMeas(X)))$.
We note that every compact metrizable space
$K = (X, \Topology)$ is separable \cite[Proposition $4.6$]{Kechris1995},
which means that $(X, \Borel)$ is a standard Borel space,
where $\Borel$ is the Borel $\sigma$-algebra generated by $\Topology$.
Additionally, in the case of such a $K$,
the topological space $\probMeas(X)$
is again compact metrizable \cite[Theorem 17.22]{Kechris1995}.

We are ready to define the space $\Mk$.
One should pay attention to the connection
to oblivious $k$-WL, cf.\ \Cref{sec:kWLGraphons}:
Here, $\iPk_0 = [0,1]^{\binom{[k]}{2}}$ is the space
of possible \enquote{edge weights} of a tuple $\bar{x} \in X^k$,
generalizing possible atomic types.
Moreover, oblivious $k$-WL defines $k$
multisets of colors in every refinement,
which results in $k$ probability measures on the previous space $\Mk_n$
in the following definition,
where we recall that
$f_* \mu$ denotes the push-forward of $\mu$ via $f$.

\begin{definition}[The Spaces $\Mk$ and $\Pk$]
    Let $k \ge 1$.
    Let $\iPk_0 \coloneqq [0,1]^{\binom{[k]}{2}}$ and inductively define
    $\Mk_n \coloneqq \prod_{i \le n} \iPk_i$
    and
    $\iPk_{n+1} \coloneqq \left(\probMeas\left(\Mk_n\right)\right)^{k}$
    for every $n \in \N$.
    Let $\Mk \coloneqq \Mk_\infty \coloneqq \prod_{n \in \N} \iPk_i$ and,
    for $n \le m \le \infty$,
    let
    $p_{m,n} \colon \Mk_m \to \Mk_n$ be the natural projection, i.e.,
    the restriction to the first $n$ components.
    Finally, define
    \begin{equation*}
        \Pk \coloneqq \Set{\alpha \in \Mk \mid (\alpha_{n+1})_{j} = (p_{n+1,n})_* (\alpha_{n+2})_{j} \text{ for all }  j \in [k], n \in \N}.
    \end{equation*}
\end{definition}

As a product of a sequence of metrizable compact spaces,
$\Mk$ is metrizable \cite[Proposition $2.4.4$]{Dudley2002}
and also compact by Tychonoff's Theorem \cite[Theorem $2.2.8$]{Dudley2002}.
Moreover, as $\Mk$ is a product of a sequence of second-countable spaces,
the Borel $\sigma$-algebra of $\Mk$
and the product of the Borel $\sigma$-algebras of its factors are the same,
cf.\ \Cref{sec:standardBorelSpaces}.

Consider the requirement
$(\alpha_{n+1})_{j} = (p_{n+1,n})_* (\alpha_{n+2})_{j}$
in the definition of $\Pk$,
and
note that
$\alpha_{n+1} \in \iPk_{n+1} = \left(\probMeas\left(\Mk_n\right)\right)^{k}$
and $\alpha_{n+2} \in \iPk_{n+2} = \left(\probMeas\left(\Mk_{n+1}\right)\right)^{k}$
for $\alpha \in \Mk$, i.e.,
$\Pk$ is well-defined.
This requirement intuitively expresses that $\alpha_{n+2}$,
which can be thought of as a coloring
after $n+2$ refinement rounds,
is consistent with $\alpha_{n+1}$ for every $n \in \N$,
but it does not require that $\alpha_0$ is consistent with $\alpha_1$.
One could add the additional consistency condition that,
for $ij \in \binom{[k]}{2}$ and $u \notin ij$,
the push-forward of $(\alpha_1)_{u}$ via the projection to component $ij$ is the Dirac
measure of $(\alpha_0)_{ij}$, but this would introduce an inconsistency
in the case $k = 2$ where there is no such $u$.
For simplicity, we just leave this out;
it does not cause any problems for us.

In terms of graphs,
an element $(\alpha_0, \alpha_1, \dots)$ of $\Mk$ can be thought of
as a sequence of unfoldings of a graph, cf.\ \cite{Dell2018},
of heights $0, 1, 2, \dots$.
These unfoldings, however, do not have to be related in any way.
The subspace $\Pk$ contains these sequences
where each unfolding is a continuation of the
previous one.
These sequences can also be viewed as a single, infinite
unfolding:
By the Kolmogorov Consistency Theorem \cite[Exercise $17.16$]{Kechris1995},
for all $\alpha \in \Pk$ and $j \in [k]$,
there is a unique measure $\mu^\alpha_{j} \in \probMeas(\Mk)$ such that
$(p_{\infty, n})_* \mu^\alpha_{j} = (\alpha_{n+1})_{j}$
for every $n \in \N$.
Moreover, one can verify that this mapping
$\alpha \mapsto \mu^\alpha_{j}$ is continuous,
cf.\ \cite[Claim $6.2$]{GrebikRocha2021}.

\begin{lemma}
    \label{le:continuity}
    $\Pk$ is closed in $\Mk$ and
    $\Pk \to \probMeas(\Mk),\, \alpha \mapsto \mu^\alpha_j$
    is continuous for every $j \in [k]$.
\end{lemma}
\begin{proof}
    To prove that $\Pk$ is closed,
    let $\alpha_i \rightarrow \alpha$
    with $\alpha_i \in \Pk$ for every $i \in \N$ and $\alpha \in \Mk$.
    Let $j \in [k]$ and $n \in \N$.
    By definition of the product topology, we have
    $((\alpha_i)_{n+2})_{j} \rightarrow (\alpha_{n+2})_j$, which yields
    \begin{align*}
        \int\limits_{\Mk_n} f \,d ((\alpha_i)_{n+1})_j
        \stackrel{\alpha_i \in \Pk}{=} \int\limits_{\Mk_n} f \,d (p_{n+1,n})_* ((\alpha_i)_{n+2})_{j}
        &= \int\limits_{\Mk_{n+1}} f \circ p_{n+1,n} \,d ((\alpha_i)_{n+2})_{j}\\
        &\rightarrow \int\limits_{\Mk_{n+1}} f \circ p_{n+1,n} \,d (\alpha_{n+2})_{j}\\
        &= \int\limits_{\Mk_n} f \,d (p_{n+1,n})_*(\alpha_{n+2})_{j}
    \end{align*}
    for every $f \in C(\Mk_n)$.
    Therefore,
    $((\alpha_i)_{n+1})_{j} \rightarrow (p_{n+1,n})_*(\alpha_{n+2})_{j}$.
    Since also
    $((\alpha_i)_{n+1})_{j} \rightarrow (\alpha_{n+1})_j$
    and the metrizable space $\probMeas(\Mk_n)$ is Hausdorff,
    we get
    $(\alpha_{n+1})_j = (p_{n+1,n})_*(\alpha_{n+2})_{j}$.
    Hence, $\alpha \in \Pk$.

    Let $j \in [k]$.
    Let $\alpha_i \rightarrow \alpha$
    with $\alpha_i \in \Pk$ for every $i \in \N$ and $\alpha \in \Pk$.
    To prove that $\mu^{\alpha_i}_j \rightarrow \mu^\alpha_j$,
    we observe that
    \begin{align*}
        \int\limits_{\Mk} f \circ p_{\infty, n} \, d\mu^{\alpha_i}_j
        = \int\limits_{\Mk_n} f \, d(p_{\infty, n})_* \mu^{\alpha_i}_j
        = \int\limits_{\Mk_n} f \, d((\alpha_i)_{n+1})_j
        \rightarrow
        \int\limits_{\Mk_n} f \, d(\alpha_{n+1})_j
        &= \int\limits_{\Mk_n} f \, d(p_{\infty, n})_* \mu^{\alpha}_j\\
        &= \int\limits_{\Mk} f \circ p_{\infty, n} \, d\mu^{\alpha}_j
    \end{align*}
    for every $n \in \N$ and every $f \in C(\Mk_n)$.
    This already proves the claim as the set
    $\bigcup_{n \in \N} C(\Mk_n) \circ p_{\infty, n}$
    is uniformly dense in $C(\Mk)$
    by the Stone-Weierstrass Theorem \cite[Theorem $2.4.11$]{Dudley2002};
    in particular, this set separates points by the definition
    of the product topology and the fact that every metrizable space
    is completely Hausdorff.
\end{proof}

\Cref{le:continuity} implies that $\Pk \in \Borel(\Mk)$
and that
$\Pk \to \R,\, \alpha \mapsto \int f \, d \mu^\alpha_j$ is measurable
for every bounded measurable real-valued function $f$ on $\Mk$
and every $j \in [k]$, cf.\
the definition of $\probMeas(\Mk)$.
This justifies the following definition of a $k$-WL distribution ($k$-WLD),
which intuitively generalizes the concept of a multiset of colors
with the additional constraints that, first,
the non-consistent sequences $\alpha \in \Mk$
have measure zero and, second, it satisfies a variant
of the Tonelli-Fubini Theorem w.r.t.\
the measures given by the mappings
$\Pk \to \probMeas(\Mk),\, \alpha \mapsto \mu^\alpha_j$.
This definition only become fully clear in the next subsections:
we will show that every graphon $W$ has a natural $k$-WLD $\nu^k_W$ associated
with it that satisfies both conditions and that these
conditions guarantee the existence of certain operators
associated with the $k$-WLD.

\begin{definition}
    \label{def:kWLD}
    Let $k \ge 1$.
    A measure $\nu \in \probMeas(\Mk)$ is called a
    \textit{$k$-Weisfeiler-Leman distribution ($k$-WLD)} if
    \begin{enumerate}
        \item $\nu(\Pk) = 1$ and
        \item $\int\limits_{\Mk} f \, d \nu = \int\limits_{\Mk} \Big(\int\limits_{\Mk} f \, d\mu^\alpha_j\Big) \, d \nu(\alpha)$ for all bounded measurable $f \colon \Mk \to \R$, $j \in [k]$.
    \end{enumerate}
\end{definition}

\subsection{The Mapping \texorpdfstring{$\owlk_W$}{owlkW}}
\label{subsec:mappingOwl}

Having defined the compact metrizable space $\Mk$,
we can finally
define the mapping $\owlk_W \colon X^k \to \Mk$
and the $k$-WL distribution $\nu^k_W$
for a graphon $W$.
To this end, let us first recall that
oblivious $k$-WL for a graph $G$ initially colors a $k$-tuple
$\bar{v} \in V(G)^k$
by its atomic type,
which includes the information of
which vertices in $\bar{v}$
are equal and which are connected by an edge.
In our case, this becomes somewhat simpler
since we do not have to deal with the case
that entries of a $k$-tuple $\bar{x} \in X^k$
are equal;
if our standard Borel space is atom free,
such diagonal sets have measure zero in the product space
and do not matter.
Hence, we only include the information $W(x_i, x_j)$
for every $ij \in \binom{[k]}{2}$.
Notice the connection to the operators
$\OperatorFamilyFromTo{\adjGraphs^k}{W}$:
by definition, we have
$(\OperatorFromTo{\ijAdjacencyij}{W} f) (\bar{x}) = W(x_i, x_j) \cdot f(\bar{x})$
for every $f \in \LTwoProdK$
and $\mu^{\otimes k}$-almost every $\bar{x} \in X^k$.

Let us also take a look at the
substitution operation in the refinement rounds
of oblivious $k$-WL.
Fix $\bar{x} \in X^k$ and $j \in [k]$.
Define
$\bar{x}[/j] \coloneqq (x_1, \dots, x_{j-1}, x_{j+1}, \dots, x_k) \in X^{k-1}$
to be the tuple obtained from $\bar{x}$ by removing the $j$th component,
and for $y \in X$, also
$\bar{x}[y/j] \coloneqq (x_1, \dots, x_{j-1}, y, x_{j+1}, \dots, x_k) \in X^k$,
which is the tuple obtained from $\bar{x}$ by replacing the $j$th component
by $y$.
The preimage of a set $A \subseteq X^k$ under the map
$\jsection \colon X \to X^k,\, y \mapsto \bar{x}[y/j]$
is
\begin{equation*}
    \jsection^{-1}(A) = \Set{y \in X \mid \bar{x}[y/j] \in A} \eqqcolon A_{\bar{x}[/j]},
\end{equation*}
which we call the \textit{section of $A$ determined by $\bar{x}[/j]$}.
Note that, technically, $A_{\bar{x}[/j]}$ also depends on $j$
and not only on the $(k-1)$-tuple
$\bar{x}[/j] \in X^{k-1}$,
but we nevertheless stick to this notation.
The mapping $\jsection$ is measurable, i.e.,
we have $A_{\bar{x}[/j]} \in \Borel$
for every $A \in \Borel^{\otimes k}$ \cite[Theorem $18.1$ (i)]{Billingsley1995}.
If we let $p_j \colon X^k \to X$ denote the projection to the $j$th component,
which is measurable by definition of $\Borel^{\otimes k}$,
then, the mapping $\jsection \circ p_j \colon X^k \to X^k,
\bar{y} \mapsto \bar{x}[y_j/j]$
is measurable as the composition of measurable functions
and we have
$(\jnsection)_*\mu^{\otimes k} = \jsection_* \mu$.
To see the connection to
the operators $\OperatorFamilyFrom{\jNeighbors^k}$,
note that the definition of
$\OperatorFrom{\jNeighborj}$
yields that
\begin{align}
    \label{eq:OperatorNJSection}
    (\OperatorFrom{\jNeighborj} f) (\bar{x})
    = \int\limits_{X} f(\bar{x}[y/j]) \dmu(y)
    = \int\limits_{X} f \circ \jsection \dmu
    = \int\limits_{X^k} f d\, (\jsection_* \mu)
\end{align}
for every $f \in \LTwoProdK$
and $\mu^{\otimes k}$-almost every $\bar{x} \in X^k$.

\begin{definition}[The Mapping $\owlk_W$]
    \label{def:owlk}
    Let $k \ge 1$
    and $W \colon X \times X \to [0,1]$ be a graphon.
    Define $\owlk_{W,0} \colon X^k \to \Mk_0$ by
    \begin{equation*}
        \owlk_{W,0}(\bar{x}) \coloneqq \big(W(x_i, x_j)\big)_{ij \in \binom{[k]}{2}}
    \end{equation*}
    for every $\bar{x} \in X^k$.
    Inductively define $\owlk_{W, n+1} \colon X^k \to \Mk_{n+1}$ by
    \begin{equation*}
        \owlk_{W,n+1}(\bar{x}) \coloneqq \left(\owlk_{W,n}(\bar{x}),\Big(\left(\owlk_{W,n} \circ \jsection\right)_*\mu\Big)_{j \in [k]}\right)
\end{equation*}
    for every $\bar{x} \in X^k$.
    Then, let $\owlk_W = \owlk_{W,\infty} \colon X^k \to \Mk$ be the mapping
    defined by
$(\owlk_W(\bar{x}))_n \coloneqq (\owlk_{W, \infty}(\bar{x}))_n \coloneqq (\owlk_{W,n}(\bar{x}))_n$
for all $n \in \N$, $\bar{x} \in X^k$.
    Finally, let
$\nu^k_W \coloneqq {\owlk_W}_{*} \mu^{\otimes k} \in \probMeas(\Mk)$
be the push-forward of $\mu^{\otimes k}$ via $\owlk_W$.
\end{definition}

An immediate consequence of \Cref{def:owlk}
is the following lemma. In particular, we use it to
prove that the mapping
$\owlk_{W,n}$ is measurable
for every $n \in \NInfty$, which actually is needed for
everything in \Cref{def:owlk} to be well defined.
\begin{lemma}
    \label{le:ikPreImageProjection}
    Let $k \ge 1$
    and $W \colon X \times X \to [0,1]$ be a graphon.
    Then,
    ${\owlk_{W,m}}^{\invSpace-1}(p_{m,n}^{-1}(A)) = {\owlk_{W,n}}^{\invSpace-1}(A)$
    for all $1 \le n < m \le \infty$ and every $A \in \Borel(\Mk_n)$.
\end{lemma}
\begin{proof}
    Let $A \in \Borel(\Mk_n)$.
    We have
    \begin{align*}
        {\owlk_{W,m}}^{\invSpace-1}(p_{m,n}^{-1}(A))
        &= \Set{\bar{x} \in X^k \mid \left(\left(\owlk_{W,m}(\bar{x})\right)_1, \dots, \left(\owlk_{W,m}(\bar{x})\right)_n\right) \in A}\\
        &= \Set{\bar{x} \in X^k \mid \left(\left(\owlk_{W,n}(\bar{x})\right)_1, \dots, \left(\owlk_{W,n}(\bar{x})\right)_n\right) \in A}\\
        &= {\owlk_{W,n}}^{\invSpace-1}(A)
    \end{align*}
    by definition of $\owlk_{W,m}$ and $\owlk_{W,n}$.
\end{proof}
\Cref{le:minAlgebraIKWMeasurable}
states not only that $\owlk_{W,n}$ is measurable
but also that the minimum $\mu^{\otimes k}$-relatively complete
sub-$\sigma$-algebra that makes it measurable is given by $\kSubAlg_{W,n}$,
cf.\ \cite[Proposition $6.6$]{GrebikRocha2021}.

\begin{lemma}
    \label{le:minAlgebraIKWMeasurable}
    Let $k \ge 1$
    and $W \colon X \times X \to [0,1]$ be a graphon.
    For $n \in \NInfty$,
    \begin{equation*}
        \kSubAlg_{W,n} = \left\langle \Set{{\owlk_{W,n}}^{\invSpace-1}(A) \mid A \in \Borel(\Mk_n)} \right\rangle.
    \end{equation*}
\end{lemma}
\begin{proof}
    Let $\subAlgD_n \coloneqq \langle \{{\owlk_{W,n}}^{\invSpace-1}(A) \mid A \in \Borel(\Mk_n)\} \rangle$.
    First, we prove $\kSubAlg_{W,n} = \subAlgD_n$ for every $n \in \N$
    by induction on $n$.
    For the induction basis $n = 0$, we have
    \begin{align*}
        \subAlgD_0
        = \left\langle \Set{{\owlk_{W,0}}^{\invSpace-1}(A) \mid A \in \Borel(\Mk_0)} \right\rangle
        = \left\langle \Set{{\owlk_{W,0}}^{\invSpace-1}(A) \mid A \in \Borel([0,1]^{\binom{[k]}{2}})} \right\rangle
    \end{align*}
    The Borel $\sigma$-algebra $\Borel([0,1]^{\binom{[k]}{2}})$
    is generated by the sets of the form $\prod_{ij \in \binom{[k]}{2}} A_{ij}$
    where $A_{ij} \in \Borel([0,1])$ and $A_{ij} = [0,1]$ for all but at most one $ij$ \cite[Section $10.B$]{Kechris1995}.
    Since it suffices to check measurability of a function for a generating set \cite[Theorem $4.1.6$]{Dudley2002},
    we may replace $\Borel([0,1]^{\binom{[k]}{2}})$
    by a generating set in the definition of $\subAlgD_0$, which yields that
    \begin{align*}
        \subAlgD_0
        &= \left\langle \Set{{\owlk_{W,0}}^{\invSpace-1}(A) \mid A \in \Borel([0,1]^{\binom{[k]}{2}})} \right\rangle\\
        &= \left\langle \Set{ \Set{\bar{x} \in X^k \mid (W(x_i,x_j))_{ij \in \binom{[k]}{2}} \in A} \mid A \in \Borel([0,1]^{\binom{[k]}{2}})} \right\rangle\\
        &= \left\langle \Set{ \Set{\bar{x} \in X^k \mid W(x_i,x_j) \in A} \mid A \in \Borel([0,1]),\, ij \in {\textstyle \binom{[k]}{2}}} \right\rangle\\
        &= \left\langle {\textstyle \bigcup_{ij \in \binom{[k]}{2}} }\Set{ \Set{\bar{x} \in X^k \mid W(x_i,x_j) \in A} \mid A \in \Borel([0,1])} \right\rangle\\
        &= \left\langle {\textstyle \bigcup_{\bm{A} \in \adjGraphs^k}} \Set{ (\OperatorAW \allOne_{X^k})^{-1}(A) \mid A \in \Borel([0,1])} \right\rangle\\
&= \kSubAlg_{W,0}.\tag{\Cref{le:algebraSequence} (\ref{le:algebraSequence:zeroAltDef})}
    \end{align*}

    For the inductive step, let $n \in \N$.
    We have to prove that $\kSubAlg_{W,n+1} = \subAlgD_{n+1}$.
    Recall that we have $\Mk_{n+1} = \Mk_n \times \left(\probMeas\left(\Mk_n\right)\right)^{k}$
    by definition
    and that the Borel $\sigma$-algebra $\Borel(\probMeas\left(\Mk_n\right))$
    is generated by the maps
    $\mu \mapsto \mu(A)$ for $A \in \Borel(\Mk_n)$ \cite[Theorem $17.24$]{Kechris1995}.
    Hence, by definition of the product $\sigma$-algebra
    and since it suffices to check measurability of a function for a generating set \cite[Theorem $4.1.6$]{Dudley2002},
    $\Borel(\Mk_{n+1})$
    is the smallest $\sigma$-algebra
    containing $\Set{p_{n+1, n}^{-1}(A) \mid A \in \Borel(\Mk_n)}$
    and making the maps
    \begin{equation*}
        \Mk_{n+1} \ni \alpha \mapsto ((\alpha)_{n+1})_j (A)
    \end{equation*}
    for $A \in \Borel(\Mk_n)$ and $j \in [k]$ measurable.
    Again by \cite[Theorem $4.1.6$]{Dudley2002},
    this means that $\subAlgD_{n+1}$
    is the smallest $\mu^{\otimes k}$-relatively complete
    sub-$\sigma$-algebra of $\Borel^{\otimes k}$
    containing
    \begin{equation*}
        \Set{{\owlk_{W, n+1}}^{\invSpace-1}(p_{n+1,n}^{-1}(A)) \mid A \in \Borel(\Mk_n)}
        =
        \Set{{\owlk_{W, n}}^{\invSpace-1}(A) \mid A \in \Borel(\Mk_n)}
    \end{equation*}
    and making the maps
    \begin{align*}
        X^k \ni \bar{x}
        \mapsto ((\owlk_{W, n+1}(\bar{x}))_{n+1})_j (A)
        &=\Big(\left(\owlk_{W,n} \circ \jsection\right)_*\mu\Big)(A)\\
        &=\int\limits_{\Mk_n} \allOne_{A} d\,\left(\owlk_{W,n} \circ \jsection\right)_*\mu\\
        &=\int\limits_{X^k} \allOne_{A} \circ \owlk_{W,n} d\, \jsection_*\mu\\
        &=(\OperatorFrom{\jNeighborj} \allOne_{A} \circ \owlk_{W,n}) (\bar{x})
    \end{align*}
    for $A \in \Borel(\Mk_n)$ and $j \in [k]$ measurable,
    where the equalities hold
    $\mu^{\otimes k}$-almost everywhere,
    cf.\ also \Cref{eq:OperatorNJSection}.

    To see that $\subAlgD_{n+1} \subseteq \kSubAlg_{W, n+1}$,
    we verify that $\kSubAlg_{W, n+1}$ contains the aforementioned sets
    and that the aforementioned maps are measurable for it.
    We have
    \begin{equation*}
        \Set{{\owlk_{W, n}}^{\invSpace-1}(A) \mid A \in \Borel(\Mk_n)}
        \stackrel{\text{def.}}{\subseteq} \subAlgD_n
        \stackrel{\text{IH}}{\subseteq} \kSubAlg_{W, n}
        \stackrel{\text{def.}}{\subseteq} \kSubAlg_{W, n + 1}.
    \end{equation*}
    By the induction hypothesis, $\owlk_{W,n}$ is $\kSubAlg_{W,n}$-measurable,
    and since $A \in \Borel(\Mk_n)$, so is
    $\allOne_{A} \circ \owlk_{W,n}$.
    Hence, by definition of $\kSubAlg_{W, n+1}$,
    $\OperatorFrom{\jNeighborj} \allOne_{A} \circ \owlk_{W,n}$
    is $\kSubAlg_{W, n+1}$-measurable,
    which is just what we wanted to prove.

    It remains to verify that $\kSubAlg_{W, n+1} \subseteq \subAlgD_{n+1}$.
    By \Cref{le:algebraSequence} (\ref{le:algebraSequence:nAltDef}),
    it suffices to prove that $\subAlgD_{n+1}$
    contains $\kSubAlg_{W, n}$ and
    makes the functions $\OperatorN \allOne_A$
    for $\bm{N} \in \jNeighbors^k$ and $A \in \kSubAlg_{W,n}$ measurable.
    We have
    \begin{equation*}
        \kSubAlg_{W,n}
        \stackrel{\text{IH}}{\subseteq} \subAlgD_n
        = \left\langle \Set{{\owlk_{W,n}}^{\invSpace-1}(A) \mid A \in \Borel(\Mk_n)} \right\rangle
        \subseteq \subAlgD_{n+1}.
    \end{equation*}
    Let $A \in \kSubAlg_{W,n}$.
    By the induction hypothesis, we have $A \in \subAlgD_{n}$.
    Since the preimage of a $\sigma$-algebra is a $\sigma$-algebra,
    we have $A = {\owlk_{W,n}}^{\invSpace-1}(B) \triangle Z$
    for some $B \in \Borel(\Mk_n)$ and $Z \in \Borel^{\otimes k}$ with $\mu^{\otimes k}(Z) = 0$.
    Then, $\bar{x} \in A \iff \owlk_{W,n}(\bar{x}) \in B$
    for every $\bar{x} \notin Z$, i.e.,
    $\allOne_{B} \circ \owlk_{W,n} = \allOne_A$, where the equality holds
    $\mu^{\otimes k}$-almost everywhere.
    Let $j \in [k]$.
    We know that $\subAlgD_{n+1}$ makes the map
    $\OperatorFrom{\jNeighborj} \allOne_{B} \circ \owlk_{W,n}
    = \OperatorFrom{\jNeighborj} \allOne_{A}$ measurable,
    but this is already what we wanted to show.

    It remains to prove that
    \begin{equation*}
        \kSubAlg_{W} = \left\langle \Set{{\owlk_{W}}^{\invSpaceSmall-1}(A) \mid A \in \Borel(\Mk)} \right\rangle,
    \end{equation*}
    where, by definition,
    we have
    $\kSubAlg_W = \langle \bigcup_{n \in \N} \kSubAlg_{W,n} \rangle$.
    It is easy to see that the Borel $\sigma$-algebra $\Borel(\Mk)$
    is generated by the projections $p_{\infty, n}$.
    Hence, by \cite[Theorem $4.1.6$]{Dudley2002},
    \begin{align*}
        \kSubAlg_W
        = \left\langle \bigcup_{n \in \N} \kSubAlg_{W,n} \right\rangle
        &= \left\langle \bigcup_{n \in \N} \Set{{\owlk_{W,n}}^{\invSpace-1}(A) \mid A \in \Borel(\Mk_n)} \right\rangle\\
        &= \left\langle \Set{{\owlk_{W,n}}^{\invSpace-1}(A) \mid n \in \N,\, A \in \Borel(\Mk_n)} \right\rangle\\
        &= \left\langle \Set{{\owlk_{W}}^{\invSpaceSmall-1}(p_{\infty, n}^{-1}(A)) \mid n \in \N,\, A \in \Borel(\Mk_n)} \right\rangle\\
        &= \left\langle \Set{{\owlk_{W}}^{\invSpaceSmall-1}(A) \mid A \in \Borel(\Mk)} \right\rangle.
    \end{align*}
\end{proof}

By \Cref{le:minAlgebraIKWMeasurable},
$\kInftySubAlg_W$ is the minimum
$\mu^{\otimes k}$-relatively complete sub-$\sigma$-algebra that
makes $\owlk_W$ measurable.
Hence $\owlk_W \colon X^k \to \Mk$
is a measurable and measure-preserving
mapping from
the measure space
$(X^k, \Borel^{\otimes k}, \mu^{\otimes k})$
to $(\Mk, \Borel(\Mk), \nu^k_W)$
and we can
consider the Koopman operator
$T_{\owlk_W} \colon \LTwoMkNuKW \to \LTwoProdK$ of $\owlk_W$.
\begin{corollary}
    \label{le:ikWMarkovIso}
    \label{le:fullIkWMarkovIso}
    Let $k \ge 1$
    and $W \colon X \times X \to [0,1]$ be a graphon.
    Then,
    \begin{enumerate}
        \item the Koopman operator
            $T_{\owlk_W} \colon \LTwoMkNuKW \to \LTwoProdK$ of $\owlk_W$
            is a Markov embedding onto
            $L^2(X^k, \kSubAlg_W, \mu^{\otimes k})$, \label{le:ikWMarkovIso:KoopmanIkW}
        \item
            the operator $S_{\kSubAlg_W} \colon \LTwoProdK \to \LTwoProdKQuoOf{\kSubAlg_W}$
            becomes a Markov isomorphism when restricted to
            $L^2(X^k, \kSubAlg_W, \mu^{\otimes k})$, and \label{le:ikWMarkovIso:quotientOp}
        \item $R^k_W \coloneqq S_{\kSubAlg_W} \circ T_{\owlk_W} \colon \LTwoMkNuKW \to \LTwoProdKQuoOf{\kSubAlg_W}$ is a Markov isomorphism. \label{le:ikWMarkovIso:iso}
    \end{enumerate}
\end{corollary}
\begin{proof}
    (\ref{le:ikWMarkovIso:KoopmanIkW}):
    Consider the measure spaces $(X^k, \Borel^{\otimes k}, \mu^{\otimes k})$
    and $(\Mk, \Borel(\Mk), \nu^k_W)$.
    The mapping $\owlk_W \colon X^k \to \Mk$ is measurable
    by \Cref{le:minAlgebraIKWMeasurable}
    and measure-preserving by definition of $\nu^k_W$.
    Hence, its Koopman operator
    $T_{\owlk_W} \colon \LTwoMkNuKW \to \LTwoProdK$
    is a Markov embedding \cite[Theorem $7.20$]{EisnerEtAl2015}.
    By \Cref{co:quotientSpaceUnique},
    it is an isometry onto
    $L^2(X^k, \kSubAlg_W, \mu^{\otimes k})$.

    (\ref{le:ikWMarkovIso:quotientOp}):
    By \Cref{th:quotientSpaces} (\ref{th:quotientSpaces:isometry}),
    the operator $S_{\kSubAlg_W} \colon \LTwoProdK \to \LTwoProdKQuoOf{\kSubAlg_W}$
    becomes a Markov isomorphism when restricted to
    $L^2(X^k, \kSubAlg_W, \mu^{\otimes k})$.

    (\ref{le:ikWMarkovIso:iso}):
    Follows immediately from (\ref{le:ikWMarkovIso:KoopmanIkW})
    and (\ref{le:ikWMarkovIso:quotientOp}).
\end{proof}

It remains to verify that $\nu^k_W$ is in fact a $k$-WLD.
The following lemma can also be seen as a justification
of the definition of a $k$-WLD.
In particular, it shows that
Tonelli-Fubini-like requirement in \Cref{def:kWLD}
actually stems from the Tonelli-Fubini Theorem.
In other words, the definition of a $k$-WLD
is chosen such that it captures the essential
properties of $\nu^k_W$ that
make it possible to define
the analogue of $\naturalKFamilyOperatorsShortW$
on the space $\LTwoMkNuKW$.
In the next section,
we define these operators on $\LTwoMkNu$
for an arbitrary $k$-WLD $\nu$.

\begin{lemma}
    \label{le:nuKWIskWLD}
    Let $k \ge 1$
    and $W \colon X \times X \to [0,1]$ be a graphon.
    Then,
    \begin{enumerate}
        \item $\owlk_W(X^k) \subseteq \Pk$,\label{le:nuKWIskWLD:subsetPk}
        \item $\mu_j^{\owlk_W(\bar{x})} = (\owlk_{W} \circ \jsection)_*\mu$
            for all $j \in [k]$, $\bar{x} \in X^k$, and \label{le:nuKWIskWLD:jNeighbors}
        \item $\nu^k_W$ is a $k$-WLD. \label{le:nuKWIskWLD:kWLD}
    \end{enumerate}
\end{lemma}
\begin{proof}
    (\ref{le:nuKWIskWLD:subsetPk}):
    Let $\bar{x} \in X^k$.
    For $j \in [k]$ and $n \in \N$, we have
    \begin{align*}
        (p_{n+1,n})_*((\owlk_W(\bar{x}))_{n+2})_j
        &= (p_{n+1,n})_*((\owlk_{W,n+2}(\bar{x}))_{n+2})_j \tag{Definition $\owlk_W$}\\
        &= (p_{n+1,n})_*((\owlk_{W,n+1} \circ \jsection)_* \mu) \tag{Definition $\owlk_{W,n+2}$}\\
        &= (\owlk_{W,n} \circ \jsection)_* \mu\\
        &= ((\owlk_{W,n+1}(\bar{x}))_{n+1})_j \tag{Definition $\owlk_{W,n+1}$}\\
        &= ((\owlk_{W}(\bar{x}))_{n+1})_j. \tag{Definition $\owlk_{W}$}
    \end{align*}
    Hence, $\owlk_W(\bar{x}) \in \Pk$.

    (\ref{le:nuKWIskWLD:jNeighbors}):
    For $n \in \N$ and $A \in \Borel(\Mk_n)$, we have
    \begin{align*}
        \mu_j^{\owlk_W(\bar{x})}(p_{\infty,n}^{-1}(A))
        = (p_{\infty, n})_* \mu_j^{\owlk_W(\bar{x})}(A)
        &= ((\owlk_W(\bar{x}))_{n+1})_j(A) \tag{Definition $\mu_j^{\owlk_W(\bar{x})}$}\\
        &= ((\owlk_{W, n+1}(\bar{x}))_{n+1})_j(A) \tag{Definition $\owlk_W$}\\
        &= (\owlk_{W, n} \circ \jsection)_* \mu (A) \tag{Definition $\owlk_{W,n+1}$}\\
        &= (\owlk_{W} \circ \jsection)_* \mu (p_{\infty,n}^{-1}(A)).
    \end{align*}
    That is, $\mu_j^{\owlk_W(\bar{x})}$
    and $(\owlk_{W} \circ \jsection)_*\mu$
    both are probability measures that
    agree on the set
    $\bigcup_{n \in \N} \Set{p_{\infty,n}^{-1}(A) \mid A \in \Borel(\Mk_n)}$,
    which generates $\Borel(\Mk)$.
    By the $\pi$-$\lambda$ Theorem \cite[Theorem $10.1$ iii)]{Kechris1995},
    they agree on all of $\Borel(\Mk)$.

    (\ref{le:nuKWIskWLD:kWLD}):
    By (\ref{le:nuKWIskWLD:subsetPk}), we have
    $\nu^k_W(\Pk) = \mu^{\otimes k}({\owlk_W}^{-1}(\Pk)) = \mu^{\otimes k}(X^k) = 1$.
    Let $j \in [k]$.
    Let $f \colon \Mk \to \R$ be bounded and measurable.
    We have
    \begin{align*}
        \int\limits_{\Mk} f \, d \nu^k_W
        \stackrel{\text{def.\ $\nu^k_W$}}{=} \int\limits_{\Mk} f \, d {\owlk_W}_* \mu^{\otimes k}
        &= \int\limits_{X^k} f \circ \owlk_W \dmuk\\
        &\stackrel{\mathclap{\text{T.-F.}}}{=} \int\limits_{X^{k-1}} \vast(\int\limits_X f \circ \owlk_W(\bar{x}[y/j]) \,d\mu(y)\vast)\,d\mu^{\otimes k-1}(\bar{x}[/j]) \tag{$x_j \in X$ arb.}\\
        &= \int\limits_{X^{k}} \vast(\int\limits_X f \circ \owlk_W(\bar{x}[y/j]) \,d\mu(y)\vast)\,d\mu^{\otimes k}(\bar{x})\\
        &= \int\limits_{X^{k}} \vast(\,\int\limits_{\Mk} f \,d (\owlk_W \circ \jsection)_* \mu\vast)\,d\mu^{\otimes k}(\bar{x})\\
        &\stackrel{\mathclap{(\ref{le:nuKWIskWLD:jNeighbors})}}{=} \int\limits_{X^{k}} \vast(\,\int\limits_{\Mk} f \,d \mu_j^{\owlk_W(\bar{x})}\vast)\,d\mu^{\otimes k}(\bar{x}) \\
        &\stackrel{\mathclap{\text{push-f.}}}{=} \; \int\limits_{\Mk} \vast(\,\int\limits_{\Mk} f \,d \mu_j^\alpha\vast)\,d{\owlk_W}_*\mu^{\otimes k}(\alpha)\\
        &\stackrel{\mathclap{{\text{def.\ $\nu^k_W$}}}}{=} \; \int\limits_{\Mk} \vast(\,\int\limits_{\Mk} f \,d \mu_j^\alpha\vast)\,d \nu^k_W(\alpha).
    \end{align*}
\end{proof}

\subsection{Operators and Weisfeiler-Leman Measures}
\label{sec:operatorsOnMeasures}

For a graphon $W$, the operator
$R^k_W \colon \LTwoMkNuKW \to \LTwoProdKQuoOf{\kSubAlg_W}$ is a Markov isomorphism
by \Cref{le:ikWMarkovIso}.
Hence, if $U$ is another graphon with $\nu^k_U = \nu^k_W$, then
$R^k_U \circ (R^k_W)^*$
is a Markov isomorphism from
$\LTwoProdKQuoOf{\kSubAlg_W}$
to $\LTwoProdKQuoOf{\kSubAlg_U}$.
However, we still lack that this Markov isomorphism
\enquote{maps} the family $\naturalKFamilyOperatorsShortW / \kSubAlg_W$
to $\naturalKFamilyOperatorsShortU / \kSubAlg_U$.
To close this gap, we show that
we can define a family $\TT_{\nu^k_W}$ of
operators on $\LTwoMkNuKW$ such that $R^k_W$
\enquote{maps} $\TT_{\nu^k_W}$
to $\naturalKFamilyOperatorsShortW / \kSubAlg_W$.
This replaces the graphon
$\MM \times \MM \to [0,1]$ defined
by Greb\'ik and Rocha \cite{GrebikRocha2021}.
Let us begin with operators
for neighbor graphs as this is the interesting case;
in particular,
it shows why we have the Tonelli-Fubini-like requirement
in the definition of a $k$-WLD.

\begin{lemma}
    \label{le:operatorNWellDefined}
    Let $k \ge 1$, and let $\nu \in \probMeas(\Mk)$ be a $k$-WLD.
    Let $j \in [k]$.
    Setting
    \begin{equation*}
        (\OperatorFromTo{\jNeighborj}{\nu} f) (\alpha) \coloneqq \int_{\Mk} f \, d \mu^\alpha_{j}
    \end{equation*}
    for all $f \in \LInftyMkNu$, $\alpha \in \Mk$
    defines an $L^\infty$-contraction
    that uniquely extends to an $L^2$-contraction
    $\LTwoMkNu \to \LTwoMkNu$.
\end{lemma}
\begin{proof}
    We show that the definition yields a well-defined contraction
    $\OperatorFromTo{\jNeighborj}{\nu}$ on $\LInftyMkNu$
    such that
    $\normT{\OperatorFromTo{\jNeighborj}{\nu} f} \le \normT{f}$
    for every $f \in \LInftyMkNu$.
    Then, $\OperatorFromTo{\jNeighborj}{\nu}$
    uniquely extends to a contraction
    on $\LTwoMkNu$
    since $\LInftyMkNu$ is dense in $\LTwoMkNu$.

    The definition of a $k$-WLD immediately yields that,
    if $A \in \Borel(\Mk)$ with $\nu(A) = 0$,
    then $\mu^\alpha_j(A) = 0$ for $\nu$-almost every $\alpha \in \Mk$.
    Hence, if a property holds $\nu$-almost everywhere,
    it holds $\mu^\alpha_j$-almost everywhere for $\nu$-almost every $\alpha \in \Mk$.
    Let $f \in \LLInftyMkNu$.
    Then, $\lvert f \rvert \le \normI{f}$ $\nu$-almost everywhere, and hence,
    $\lvert f \rvert \le \normI{f}$ holds $\mu^\alpha_j$-almost everywhere
    for $\nu$-almost every $\alpha \in \Mk$.
    Thus, for $\nu$-almost every $\alpha \in \Mk$, we have
    \begin{equation*}
        \Big\lvert \int\limits_{\Mk} f \, d \mu^\alpha_j \Big\rvert
        \le \int\limits_{\Mk} \lvert f \rvert \, d \mu^\alpha_j
        \le \int\limits_{\Mk} \normI{f} \, d \mu^\alpha_j
        = \normI{f},
    \end{equation*}
    which yields that
    $\normI{\OperatorFromTo{\jNeighborj}{\nu} f} \le \normI{f}$.
    In particular,
    if $f,g \in \LLInftyMkNu$ are equal $\nu$-almost everywhere,
    then
    \begin{align*}
        \normI{\OperatorFromTo{\jNeighborj}{\nu} f - \OperatorFromTo{\jNeighborj}{\nu} g}
        = \normI{\OperatorFromTo{\jNeighborj}{\nu} (f - g)}
        \le \normI{f-g}
        = 0,
    \end{align*}
    that is, $\OperatorFromTo{\jNeighborj}{\nu} f$ and $\OperatorFromTo{\jNeighborj}{\nu} g$
    are equal $\nu$-almost everywhere.
    Here we used that the mapping $\OperatorFromTo{\jNeighborj}{\nu}$ is linear,
    which follows directly from the linearity of the integral.
    Recall that
    $\Pk \to \R, \alpha \mapsto \int f \mu^\alpha_j$
    is measurable for every bounded measurable $\R$-valued function $f$ on $\Mk$
    by \Cref{le:continuity} and the definition of $\probMeas(\Pk)$.
    Since $\Pk \in \Borel(\Mk)$ by \Cref{le:continuity} and $\nu(\Pk) = 1$,
    this combined with the previous considerations yields that
    $\OperatorFromTo{\jNeighborj}{\nu}$ is a well-defined
    mapping $\LInftyMkNu \to \LInftyMkNu$.

    It remains to show that
    $\normT{\OperatorFromTo{\jNeighborj}{\nu} f} \le \normT{f}$
    for every $f \in \LInftyMkNu$.
    We have
    \begin{align*}
        \normT{\OperatorFromTo{\jNeighborj}{\nu} f}^2
        = \int\limits_{\Mk} \notsovast( \,\int\limits_{\Mk} f \, d \mu^\alpha_j \notsovast)^2 \, d \nu(\alpha)
        \;&\stackrel{\mathclap{\text{C.-S.}}}{\le} \int\limits_{\Mk} \notsovast(\,\int\limits_{\Mk} f^2 \, d \mu^\alpha_j \notsovast)\, d \nu(\alpha) \tag{Cauchy-Schwarz}\\
        &= \int\limits_{\Mk} f^2 \, d \nu \tag{def.\ $k$-WLD}\\
        &= \normT{f}^2
    \end{align*}
    Note that we again used that
    $\lvert f \rvert \le \normI{f}$ holds $\mu^\alpha_j$-almost everywhere
    for $\nu$-almost every $\alpha \in \Mk$
    in order to apply the Cauchy-Schwarz inequality.
\end{proof}

The following lemma states that
\Cref{le:operatorNWellDefined}
is indeed the right definition.
\begin{lemma}
    \label{le:DIDMOperatorN}
    Let $k \ge 1$
    and $W \colon X \times X \to [0,1]$ be a graphon.
    For every $\bm{N} \in \jNeighbors^k$,
    \begin{multicols}{2}
    \begin{enumerate}
        \item $T_{\bm{N}} \circ T_{\owlk_W} = T_{\owlk_W} \circ \OperatorNNuW$, \label{le:DIDMOperatorN:normalOp}
        \item $(T_{\bm{N}})_{\kSubAlg_W} \circ T_{\owlk_W} = T_{\owlk_W} \circ \OperatorNNuW$, and \label{le:DIDMOperatorN:expOp}
        \item $T_{\bm{N}} / \kSubAlg_W \circ R^k_W = R^k_W \circ \OperatorNNuW$. \label{le:DIDMOperatorN:iso}
    \end{enumerate}
    \end{multicols}
\end{lemma}
\begin{proof}
    (\ref{le:DIDMOperatorN:normalOp}):
    Let $j \in [k]$.
    We have
    \begin{align*}
        (T_{\jNeighborj} \circ T_{\owlk_W} f) (\bar{x})
        = (T_{\jNeighborj} f \circ \owlk_W) (\bar{x})
        &= \int\limits_X f \circ \owlk_W(\bar{x}[y/j]) \, d\mu(y) \tag{def.}\\
        &= \int\limits_{\Mk} f \, d (\owlk_W \circ \jsection)_* \mu\\
        &= \int\limits_{\Mk} f \, d \mu_j^{\owlk_W(\bar{x})} \tag{\Cref{le:nuKWIskWLD} (\ref{le:nuKWIskWLD:jNeighbors})}\\
        &= (\OperatorFromTo{\jNeighborj}{\nu^k_W} f) (\owlk_W(\bar{x})) \tag{def.\ $\OperatorFromTo{\jNeighborj}{\nu^k_W}$}\\
        &= (T_{\owlk_W} \circ \OperatorFromTo{\jNeighborj}{\nu^k_W} f) (\bar{x}) \tag{def.}
    \end{align*}
    for $\mu^{\otimes k}$-almost every $\bar{x} \in X^k$
    and every $f \in \LInftyMkNuKW$.
    This already proves the claim
    as $\LInftyMkNuKW$ is dense in $\LTwoMkNuKW$.

    (\ref{le:DIDMOperatorN:expOp}):
    We have
    \begin{align*}
        (T_{\bm{N}})_{\kSubAlg_W} \circ T_{\owlk_W}
        &= T_{\bm{N}} \circ \ExpVal{\kSubAlg_W} \circ T_{\owlk_W} \tag{\Cref{le:expOperator} (\ref{le:expOperator:invariant}) and \Cref{le:algebraSequence} (\ref{le:algebraSequence:TInvariant})}\\
        &= T_{\bm{N}} \circ T_{\owlk_W} \tag{cf.\ \Cref{le:ikWMarkovIso}}\\
        &= T_{\owlk_W} \circ \OperatorNNuW. \tag{(\ref{le:DIDMOperatorN:normalOp})}
    \end{align*}

    (\ref{le:DIDMOperatorN:iso}):
    We have
    \begin{align*}
        T_{\bm{N}} / \kSubAlg_W \circ R^k_W
        &=S_{\kSubAlg_W} \circ T_{\bm{N}} \circ I_{\kSubAlg_W} \circ S_{\kSubAlg_W} \circ T_{\owlk_W} \tag{def.}\\
        &= S_{\kSubAlg_W} \circ \ExpVal{\kSubAlg_W} \circ T_{\bm{N}} \circ \ExpVal{\kSubAlg_W} \circ T_{\owlk_W} \tag{\Cref{th:quotientSpaces} (\ref{th:quotientSpaces:SExpVal}) and (\ref{th:quotientSpaces:ExpVal})}\\
        &= S_{\kSubAlg_W} \circ (T_{\bm{N}})_{\kSubAlg_W} \circ T_{\owlk_W} \tag{def.}\\
        &= S_{\kSubAlg_W} \circ T_{\owlk_W} \circ \OperatorNNuW \tag{(\ref{le:DIDMOperatorN:expOp})}\\
        &= R^k_W \circ \OperatorNNuW. \tag{def.}
    \end{align*}
\end{proof}

Defining the operators
for adjacency graphs is much simpler.
Intuitively, every $\alpha \in \Mk$
contains the values $W(x_i, x_j)$ for every
$ij \in \binom{[k]}{2}$ at position $0$.

\begin{lemma}
    \label{le:operatorAWellDefined}
    Let $k \ge 1$, and let $\nu \in \probMeas(\Mk)$ be a $k$-WLD.
    Let $ij \in \binom{[k]}{2}$.
    Setting
    \begin{equation*}
        (\OperatorFromTo{\ijAdjacencyij}{\nu} f) (\alpha) \coloneqq (\alpha_0)_{ij} \cdot f(\alpha)
    \end{equation*}
    for all $f \in \LTwoMkNu$, $\alpha \in \Mk$
    defines an $L^\infty$- and $L^2$-contraction $\LTwoMkNu \to \LTwoMkNu$.
\end{lemma}
\begin{proof}
    The mapping $\alpha \mapsto (\alpha_0)_{ij}$ is measurable
    by definition of the product $\sigma$-algebra.
    Hence,
    $\OperatorFromTo{\ijAdjacencyij}{\nu} f$
    for $f \in \LTwoMkNu$
    is measurable as the product of measurable functions.
    Moreover,
    by definition of $\Mk$,
    the function
    $\alpha \mapsto (\alpha_0)_{ij}$
    is bounded by $1$,
    which immediately yields that
    $\normT{\OperatorFromTo{\ijAdjacencyij}{\nu} f} \le \normT{f}$
    for $f \in \LTwoMkNu$
    and
    $\normI{\OperatorFromTo{\ijAdjacencyij}{\nu} f} \le \normI{f}$
    for $f \in \LInftyMkNu$.
    Moreover, $\OperatorFromTo{\ijAdjacencyij}{\nu}$ is linear
    as a multiplication operator.
\end{proof}

Analogously to \Cref{le:DIDMOperatorN},
one can verify that
\Cref{le:operatorAWellDefined}
is in fact the right definition.

\begin{lemma}
    \label{le:DIDMOperatorA}
    Let $k \ge 1$
    and $W \colon X \times X \to [0,1]$ be a graphon.
    For every $\bm{A} \in \adjGraphs^k$,
    \begin{multicols}{2}
    \begin{enumerate}
        \item $\OperatorAW \circ T_{\owlk_W} = T_{\owlk_W} \circ \OperatorANuW$, \label{le:DIDMOperatorA:normalOp}
        \item $(\OperatorAW)_{\kSubAlg_W} \circ T_{\owlk_W} = T_{\owlk_W} \circ \OperatorANuW$, and\label{le:DIDMOperatorA:expOp}
        \item $\OperatorAW / \kSubAlg_W \circ R^k_W = R^k_W \circ \OperatorANuW$. \label{le:DIDMOperatorA:iso}
    \end{enumerate}
    \end{multicols}
\end{lemma}
\begin{proof}
    (\ref{le:DIDMOperatorA:normalOp}):
    Let $ij \in \binom{[k]}{2}$.
    We have
    \begin{align*}
        (\OperatorFromTo{\ijAdjacencyij}{W} \circ T_{\owlk_W} f) (\bar{x})
        &= (\OperatorFromTo{\ijAdjacencyij}{W} f \circ \owlk_W) (\bar{x})\\
        &= W(x_i, x_j) \cdot (f \circ \owlk_W) (\bar{x}) \tag{def.}\\
        &= ((\owlk_W(\bar{x}))_0)_{ij} \cdot (f \circ \owlk_W) (\bar{x}) \tag{def.\ $\owlk_W$}\\
        &= (\OperatorFromTo{\ijAdjacencyij}{\nu^k_W} f) (\owlk_W(\bar{x})) \tag{def.\ $\OperatorFromTo{\ijAdjacencyij}{\nu^k_W}$}\\
        &= (T_{\owlk_W} \circ \OperatorFromTo{\ijAdjacencyij}{\nu^k_W} f) (\bar{x}) \tag{def.}
    \end{align*}
    for $\mu^{\otimes k}$-almost every $\bar{x} \in X^k$
    and every $f \in \LTwoMkNuKW$.

    (\ref{le:DIDMOperatorA:expOp}) and
    (\ref{le:DIDMOperatorA:iso}):
    Analogous to the proof of
    (\ref{le:DIDMOperatorN:expOp}) and
    (\ref{le:DIDMOperatorN:iso})
    of \Cref{le:DIDMOperatorN}, respectively.
\end{proof}

For a $k$-WLD $\nu \in \probMeas(\Mk)$,
we define the family of $L^\infty$-contractions
$\TT_\nu \coloneqq (\OperatorFromTo{\bm{F}}{\nu})_{\bm{F} \in \adjNeiGraphs^k}$.
\Cref{le:DIDMOperatorN}~(\ref{le:DIDMOperatorN:iso}) and
\Cref{le:DIDMOperatorA}~(\ref{le:DIDMOperatorA:iso}) can then be rephrased as the following corollary.

\begin{corollary}
    \label{co:DIDMOperatorFamily}
    Let $k \ge 1$
    and $W \colon X \times X \to [0,1]$ be a graphon.
    Then, $\naturalKFamilyOperatorsShortW/\kSubAlg_W \circ R^k_W = R^k_W \circ \TT_{\nu^k_W}$.
\end{corollary}

Recall \Cref{def:homFunctionFamily}, i.e.,
the homomorphism density of a term
in a family of $L^\infty$-contractions.
In particular, this definition
applies to the family $\TT_{\nu^k_W}$
of the $k$-WLD $\nu^k_W$
of a graphon $W$.
\Cref{le:markovEmbeddingPreservesHomomorphisms} with
the previous corollary yields
that $\TT_{\nu^k_W}$ and $\naturalKFamilyOperatorsShortW / \kSubAlg_W$
give us the same homomorphism densities (and also functions),
which are just the original homomorphism densities in $W$.

\begin{corollary}
    \label{le:operatorsOnMkPreserveHomomorphisms}
    Let $k \ge 1$.
    Let $W \colon X \times X \to [0,1]$ be a graphon
    and $\subAlg \in \subAlgsk$ be $W$-invariant.
    Then,
    $t(\term, \TT_{\nu^k_W}) = t(\graphOfTerm, W)$
    for every $\term \in \treeClosureK$.
\end{corollary}
\begin{proof}
    By \Cref{le:invariantQuotientPreservesHomomorphisms},
    we have
    $t(\term, \naturalKFamilyOperatorsShortW \!/ \kSubAlg_W) = t(\graphOfTerm, W)$
    since $\kSubAlg_W$ is $W$-invariant by \Cref{le:algebraSequence}.
    By \Cref{co:DIDMOperatorFamily},
    we have
    $\naturalKFamilyOperatorsShortW/\kSubAlg_W \circ R^k_W = R^k_W \circ \TT_{\nu^k_W}$,
    where $R^k_W$ is a Markov isomorphism by \Cref{le:ikWMarkovIso}.
    Then, \Cref{le:markovEmbeddingPreservesHomomorphisms}
    yields $t(\term, \naturalKFamilyOperatorsShortW/\kSubAlg_W) = t(\term, \TT_{\nu^k_W})$.
\end{proof}

The mapping $\owlk_W \colon X^k \to \Mk$
assigns an element of $\Mk$ to a $k$-tuple
of elements from $X$.
The order of these elements from $X$ in the tuple
has no deeper meaning and should, intuitively, not matter.
While we already have defined what it means for an operator
and a sub-$\sigma$-algebra to be permutation invariant,
we now also formalize this for $k$-WLDs, and in particular,
show that the $k$-WLD $\nu^k_W$ of a graphon $W$,
which is defined via $\owlk$, is permutation invariant.
This allows us to show that, in \Cref{th:WLForGraphons},
Markov operators and Markov isomorphisms can always be assumed
to be permutation invariant and, hence, confirms the intuition
that the order of the elements in a $k$-tuple from $X$ does not matter.
This becomes important for
the connection to linear equations in \Cref{sec:opHierarchies},
where this assumption of permutation invariance is always implicitly present.

A permutation $\pi \colon [k] \to [k]$
extends to a measurable bijection $\pi \colon \Mk \to \Mk$ as follows:
We obtain a measurable bijection $\pi \colon P^k_0 \to P^k_0$
by setting
$\pi((y_{ij})_{ij}) \coloneqq (y_{\pi(i) \pi(j)})_{ij}$
for $(y_{ij})_{ij} \in [0,1]^{\binom{[k]}{2}}$.
From there on, $\pi$ inductively extends to a measurable
bijection $\pi \colon \Mk_n \to \Mk_n$
by component-wise application
and, then, to a measurable bijection
$\pi \colon P^k_{n+1} \to P^k_{n+1}$ by setting
$\pi((\mu_j)_{j \in [k]}) = (\pi_* \mu_{\pi(j)})_{j \in [k]}$
for every $(\mu_j)_{j \in [k]} \in P^k_{n+1}$.
Finally, we obtain
the measurable
bijection $\pi \colon \Mk_n \to \Mk_n$
by component-wise application.

Let $\nu \in \probMeas(\Mk)$ be a $k$-WLD
and $\pi \colon [k] \to [k]$ be a permutation.
By definition of $\pi_*\nu$, the extension $\pi \colon \Mk \to \Mk$
is a measure-preserving map from
$(\Mk, \Borel(\Mk), \nu)$ to $(\Mk, \Borel(\Mk), \pi_*\nu)$ by definition.
The $k$-WLD $\nu$ is called \textit{$\pi$-invariant}
if $\pi_* \nu = \nu$,
in which case
we can view the Koopman operator of $\pi$
as an operator
$\OperatorFromTo{\pi}{\nu} \colon \LTwoMkNu \to \LTwoMkNu$.
The notation $\OperatorFromTo{\pi}{\nu}$
avoids confusion with the Koopman operator of $\pi$
when viewing it as a map $X^k \to X^k$, which we denote just by $T_\pi$.
We call a $k$-WLD $\nu \in \probMeas(\Mk)$ \textit{permutation-invariant} if it is
$\pi$-invariant for every permutation $\pi \colon [k] \to [k]$.
Then \Cref{le:DIDMOperatorP}
yields that the $k$-WLD $\nu^k_W$ of a graphon $W$ is permutation invariant.

\begin{lemma}
    \label{le:DIDMOperatorP}
    Let $k \ge 1$
    and $W \colon X \times X \to [0,1]$ be a graphon.
    For every permutation $\pi \colon [k] \to [k]$,
    \begin{multicols}{2}
    \begin{enumerate}
        \item $\pi \circ \owlk_W = \owlk_W \circ \pi$, \label{le:DIDMOperatorP:commutes}
        \item $\nu^k_W$ is $\pi$-invariant, \label{le:DIDMOperatorP:invariant}
        \item $T_{\pi} \circ T_{\owlk_W} = T_{\owlk_W} \circ \OperatorPNuW$, \label{le:DIDMOperatorP:normalOp}
        \item $(T_{\pi})_{\kSubAlg_W} \circ T_{\owlk_W} = T_{\owlk_W} \circ \OperatorPNuW$, and\label{le:DIDMOperatorP:expOp}
        \item $T_{\pi} / \kSubAlg_W \circ S_{\kSubAlg_W} \circ T_{\owlk_W} = S_{\kSubAlg_W} \circ T_{\owlk_W} \circ \OperatorPNuW$.\label{le:DIDMOperatorP:iso}
    \end{enumerate}
    \end{multicols}
\end{lemma}
\begin{proof}
    (\ref{le:DIDMOperatorP:commutes}):
    We prove that
    $\pi \circ \owlk_{W,n} = \owlk_{W, n} \circ \pi$
    by induction on $n \in \N$.
    This yields
    $(\pi \circ \owlk_W (\bar{x}))_n = (\owlk_W \circ \pi (\bar{x}))_n$
    for every $\bar{x} \in X^k$
    by induction on $n \in \N$,
    which then implies the claim.
    For the base case, we have
    \begin{align*}
        \pi(\owlk_{W,0} (\bar{x}))
        = \left((\owlk_{W,0}(\bar{x}))_{\pi(i)\pi(j)}\right)_{ij \in \binom{[k]}{2}}
        = \left(W(x_{\pi(i)},x_{\pi(j)})\right)_{ij \in \binom{[k]}{2}}
        = \owlk_{W,0}(\pi(\bar{x}))
    \end{align*}
    for every $\bar{x} \in X^k$.
    For the inductive step,
    the induction hypothesis yields that
    $(\pi (\owlk_{W,n+1} (\bar{x})))_{i} = (\owlk_{W,n+1}(\pi(\bar{x})))_{i}$
    for every $\bar{x} \in X^k$ and every $i \le n$.
    Moreover, we have
    \begin{align*}
        (\pi (\owlk_{W,n+1} (\bar{x})))_{n+1}
        &=\left(\pi_*((\owlk_{W,n+1}(\bar{x}))_{n+1})_{\pi(j)}\right)_{j \in [k]}\\
        &=\left(\pi_*\left(\left(\owlk_{W,n} \circ \bar{x}[\cdot/\pi(j)]\right)_*\mu\right)\right)_{j \in [k]}\\
        &=\left(\left(\pi \circ \owlk_{W,n} \circ \bar{x}[\cdot/\pi(j)]\right)_*\mu\right)_{j \in [k]}\\
&=\left(\left(\pi \circ \owlk_{W,n} \circ \pi^{-1}\circ \pi(\bar{x})[\cdot/j])\right)_*\mu\right)_{j \in [k]}\\
        &=\left(\left(\owlk_{W,n} \circ \pi(\bar{x})[\cdot/j])\right)_*\mu\right)_{j \in [k]} \tag{IH}\\
        &= (\owlk_{W,n+1}(\pi(\bar{x})))_{n+1}
    \end{align*}
    for every $\bar{x} \in X^k$.

    (\ref{le:DIDMOperatorP:invariant}):
    We have
    \begin{align*}
        \pi_* \nu^k_W
        = \pi_* ({\owlk_W}_* \mu^{\otimes k})
        = (\pi \circ \owlk_W)_* \mu^{\otimes k}
        \stackrel{(\ref{le:DIDMOperatorP:commutes})}{=} (\owlk_W \circ \pi)_* \mu^{\otimes k}
        = {\owlk_W}_* (\pi_* \mu^{\otimes k})
        &= {\owlk_W}_* \mu^{\otimes k}\\
        &= \nu^k_W.
    \end{align*}

    (\ref{le:DIDMOperatorP:normalOp}):
    We have
    \begin{equation*}
        T_{\pi} \circ T_{\owlk_W} f
        = f \circ \owlk_W \circ \pi
        \stackrel{(\ref{le:DIDMOperatorP:commutes})}{=} f \circ \pi \circ \owlk_W
        = T_{\owlk_W} \circ \OperatorPNuW f
    \end{equation*}
    for every $f \in \LTwoMkNuKW$.

    (\ref{le:DIDMOperatorP:expOp}) and (\ref{le:DIDMOperatorP:iso}):
    Analogous to the proof of
    (\ref{le:DIDMOperatorN:expOp}) and
    (\ref{le:DIDMOperatorN:iso})
    of \Cref{le:DIDMOperatorN}, respectively.
\end{proof}

\subsection{Homomorphism Functions and Weisfeiler-Leman Measures}
\label{sec:WLIndistinguishability}

For the proof of \Cref{th:kWLGraphons},
\Cref{le:operatorsOnMkPreserveHomomorphisms}
allows us to get from $k$-WLDs to homomorphism densities,
but getting to the other characterizations
from there is arguably the most involved part of the proof.
As Greb\'ik and Rocha have shown \cite{GrebikRocha2021},
the key tool needed for this is the Stone-Weierstrass Theorem:
It yields that the set of homomorphism functions on $\Mk$,
which is yet to be defined,
is dense in the set $C(\Mk)$ of continuous functions on $\Mk$.
This then means that
equal homomorphism densities already imply equal $k$-WLDs.

To apply the Stone-Weierstrass Theorem,
we have to define
the homomorphism function
of a term on the set $\Mk$.
Recall that an $\alpha \in \Mk$ is a sequence
$\alpha = (\alpha_0, \alpha_1, \alpha_2, \dots)$
that, intuitively, corresponds to a sequence
of unfoldings of heights $0, 1, 2, \dots$ of a graphon.
However, as the components $\alpha_0, \alpha_1, \alpha_2$
do not have to be consistent, cf.\
the definition of $\Pk$,
using different components may lead to different functions.
Hence, we define a whole set of functions for a single term
by considering all ways in which we may use the components
to define a homomorphism function.
We could avoid this by defining homomorphism functions
just on $\Pk$ instead of $\Mk$;
this, however, would complicate things further down the road,
which is why we just accept this small inconvenience.
Note the similarity between the following definition
and the operators defined in the previous section.
\begin{definition}
    \label{def:setOfHomFuns}
    Let $k \ge 1$.
    For every term $\term \in \treeClosureK$
    and every $n \in \N$ with $n \ge h(\term)$,
    we inductively define the set
    $F^\term_n$ of functions $\Mk_n \to [0,1]$
    as the smallest set such that
    \begin{enumerate}
        \item
            $\allOne_{\Mk_n} \in F_n^\identityGraph$,
        \item
            $\alpha \mapsto (\alpha_0)_{ij} \cdot f(\alpha) \in F_n^{\ijAdjacencyij \circ \term}$
            for every $f \in F_n^{\term}$,
        \item
            $\alpha \mapsto \int_{\Mk_{n}} f \, d(\alpha_{n+1})_{j}
            \in F_{n+1}^{\jNeighborj \circ \term}$
            for every $f\in F_n^{\term}$ and every $j \in [k]$,
        \item
            $f_1 \cdot f_2 \in F_n^{\term_1 \blProd \term_2}$
            for all $f_1 \in F_n^{\term_1}, f_2 \in F_n^{\term_2}$, and
        \item
            $f \circ p_{n, m} \in F_n^{\term}$
            for every $f \in F_m^{\term}$ and every $m \in \N$ with $n > m \ge h(\term)$.
    \end{enumerate}
    Moreover, for every term $\term \in \treeClosureK$, define the set $F^\term$ of
    functions $\Mk \to [0,1]$
    by
    \begin{equation*}
        F^\term \coloneqq F^\term_\infty \coloneqq \bigcup_{h(\term) \le n < \infty} F^\term_n \circ p_{\infty, n}.
    \end{equation*}
\end{definition}

With a simple induction, one can verify that
for every term $\term \in \treeClosureK$
and every $n \in \NInfty$ with $n \ge h(\term)$,
the set $F^\term_n$ is non-empty and
all functions in it are
well-defined
and continuous.
Recall that,
for a
term $\term \in \treeClosure{\adjNeiGraphs^k}$
and a $k$-WLD $\nu \in \probMeas(\Mk)$,
the operators $\TT_{\nu}$
already define the homomorphism function
$\homFunctionT{\TT_{\nu}} \in \LInftyMkNu$
by \Cref{def:homFunctionFamily}.
Note that the $k$-WLD $\nu$ satisfying $\nu(\Pk) = 1$
is the reason why we only have this single function $\homFunctionT{\TT_{\nu}}$.
Then, it should come at no surprise that this single function
is equal to all of the previously defined functions $\nu$-almost everywhere.

\begin{lemma}
    \label{le:allHomFunsEqual}
    Let $k \ge 1$ and
    $\nu \in \probMeas(\Mk)$ be a $k$-WLD.
    Let $\term \in \treeClosure{\adjNeiGraphs^k}$ be a term
    and $n \in \N$ with $n \ge h(\term)$.
    Then,
    every function in $F_n^\term \circ p_{\infty, n}$
    is equal to $\homFunctionT{\TT_{\nu}}$ $\nu$-almost everywhere.
\end{lemma}
\begin{proof}
    We prove the statement by induction on $\term$ and $n$.
    For the base case, we have
    $\allOne_{\Mk_n} \circ p_{\infty, n} = \allOne_{\Mk} = \homFunction{\identityGraph}{\TT_\nu}$
    $\nu$-almost everywhere.
    For the inductive step, first consider
    $\alpha \mapsto (\alpha_0)_{ij} \cdot f(\alpha) \in F_n^{\ijAdjacencyij \circ \term}$
    for an $f \in F_n^{\term}$,
    where we have
    \begin{align*}
(\alpha_0)_{ij} \cdot f(p_{\infty,n}(\alpha))
        = (\OperatorFromTo{\ijAdjacencyij}{\nu} f \circ p_{\infty, n})(\alpha)
        \stackrel{\text{IH}}{=} (\OperatorFromTo{\ijAdjacencyij}{\nu} \homFunctionT{\TT_\nu})(\alpha)
        = \homFunction{\ijAdjacencyij \circ \term}{\nu}(\alpha)
    \end{align*}
    for $\nu$-almost every $\alpha \in \Mk$.
    Next, consider
    $\alpha \mapsto \int_{\Mk_{n}} f \, d(\alpha_{n+1})_{j}
    \in F_{n+1}^{\jNeighborj \circ \term}$
    for an $f\in F_n^{\term}$ and a $j \in [k]$.
    Since $\nu$ is a $k$-WLD, we have $\nu(\Pk) = 1$,
    which yields that
    \begin{align*}
        \int_{\Mk_{n}} f \, d(\alpha_{n+1})_{j}
        = \int_{\Mk_{n}} f \, d (p_{\infty, n})_* \mu^\alpha_j
        = \int_{\Mk} f \circ p_{\infty,n} \, d \mu^\alpha_j
        &= (\OperatorFromTo{\jNeighborj}{\nu} f \circ p_{\infty,n})(\alpha)\\
        &\stackrel{\mathclap{\text{IH}}}{=} (\OperatorFromTo{\jNeighborj}{\nu}  \homFunction{\term}{\nu})(\alpha)\\
        &= \homFunction{\jNeighborj \circ \term}{\nu}(\alpha)
    \end{align*}
    for $\nu$-almost every $\alpha \in \Mk$.
    For the product
    $f_1 \cdot f_2 \in F_n^{\term_1 \blProd \term_2}$
    of $f_1 \in F_n^{\term_1}, f_2 \in F_n^{\term_2}$, we have
    \begin{align*}
        (f_1 \cdot f_2) \circ p_{\infty, n}
        = (f_1 \circ p_{\infty, n}) \cdot (f_2 \circ p_{\infty, n})
        &\stackrel{\mathclap{\text{IH}}}{=} \homFunction{\term_1}{\TT_\nu} \cdot \homFunction{\term_2}{\TT_\nu}
        =\homFunction{\term_1 \cdot \term_2}{\TT_\nu}
    \end{align*}
    $\nu$-almost everywhere.
    Finally, consider
    $f \circ p_{n, m} \in F_n^{\term}$
    for $f \in F_m^{\term}$ and $m \in \N$ with $n > m \ge h(\term)$.
    Then,
$f \circ p_{n,m} \circ p_{\infty, n}
        = f \circ p_{\infty, m}
        = \homFunctionT{\TT_\nu}$
holds $\nu$-almost everywhere
    by the inductive hypothesis.
\end{proof}

\Cref{le:operatorsOnMkPreserveHomomorphisms}
yields the following corollary to the previous lemma.

\begin{corollary}
    \label{co:homDensitykWLDIntegral}
    Let $k \ge 1$ and
    $W \colon X \times X \to [0,1]$ be a graphon.
    For every term $\term \in \treeClosure{\adjNeiGraphs^k}$
    and every function $f \in F^\term$,
    we have
    \begin{equation*}
        t(\graphOfTerm, W) = \int_{\Mk} f \nu^k_W.
    \end{equation*}
\end{corollary}

For every $n \in \NInfty$, define
$\Tk_n \coloneqq \bigcup_{{\term \in \treeClosureK,\\ h(\term) \le n}} F_n^\term$
and abbreviate $\Tk \coloneqq \Tk_\infty$.
By induction, we can use the Stone-Weierstrass Theorem and the Portmanteau
Theorem to show that the Stone-Weierstrass Theorem is actually
applicable to all of these sets and, in particular,
to $\Tk$, cf.\ \cite[Proposition $7.5$]{GrebikRocha2021}.

\begin{lemma}
    \label{le:tkSeparatesPoints}
    Let $k \ge 1$.
    For every $n \in \NInfty$,
    the set $\Tk_n$ is closed under multiplication, contains $\allOne_{\Mk_n}$,
    and separates points of $\Mk_n$.
\end{lemma}
\begin{proof}
    First, consider the case that $n \in \N$.
    We trivially have $\allOne_{\Mk_n} \in F_n^\identityGraph \subseteq \Tk_n$
    by definition.
    Moreover, for $f_1, f_2 \in \Tk_n$, there are terms $\term_1, \term_2 \in \treeClosureK$
    with $h(\term_1) \le n$ and $h(\term_2) \le n$.
    such that $f_1 \in F_n^{\term_1}$ and $f_2 \in F_n^{\term_2}$.
    Then, $f_1 \cdot f_2 \in F_n^{\term_1 \cdot \term_2} \subseteq \Tk_n$
    as $h(\term_1 \cdot \term_2) = \max \Set{h(\term_1), h(\term_2)} \le n$.
    We prove that $\Tk_n$ separates points of $\Mk_n$ by induction on $n$.
    For the base case $n = 0$,
    let $\beta \neq \gamma \in \Mk_0$.
    Then, there is an $ij \in \binom{[k]}{2}$ such that $\beta_{ij} \neq \gamma_{ij}$,
    and the function $\alpha \mapsto (\alpha_0)_{ij} \in F_0^{\ijAdjacencyij \circ \identityGraph}$ separates $\beta$ and $\gamma$.

    For the inductive step, assume that $\Tk_n$ separates points of $\Mk_n$.
    Let $\beta \neq \gamma \in \Mk_{n+1}$.
    If there is an $0 \le m \le n$ such that $\beta_m \neq \gamma_m$,
    then $p_{n+1, n}(\beta) \neq p_{n+1, n}(\gamma) \in \Mk_n$.
    Hence, by the inductive hypothesis,
    there is an $f \in \Tk_n$ such that $f(p_{n+1, n}(\beta)) \neq f(p_{n+1,n}(\gamma))$.
    By definition, $f \in F_n^\term$ for some term $\term \in \treeClosureK$ with $h(\term) \le n$.
    Therefore, $f \circ p_{n+1, n} \in F_{n+1}^\term \subseteq \Tk_{n+1}$
    is a function that separates $\beta$ and $\gamma$.

    For the remaining case, assume that $\beta_{n+1} \neq \gamma_{n+1}$.
    Then, there is a $j \in {[k]}$ such that
    $(\beta_{n+1})_{j} \neq (\gamma_{n+1})_{j}$.
    By the inductive hypothesis
    and the Stone-Weierstrass Theorem \cite[Theorem $2.4.11$]{Dudley2002},
    the linear hull of $\Tk_n$ is uniformly dense in $C(\Mk_n)$.
    Since $\Mk_n$ is Hausdorff, it then follows from the
    Portmanteau Theorem \cite[Theorem $17.20$]{Kechris1995}
    that there is an $f \in \Tk_n$ such that
    $\int_{\Mk_n} f \, d (\beta_{n+1})_{j}
    \neq \int_{\Mk_n} f \, d (\gamma_{n+1})_{j}$.
    By definition, $f \in F_n^\term$ for some term $\term \in \treeClosureK$ with $h(\term) \le n$.
    Then, $\alpha \mapsto \int_{\Mk_{n}} f \, d(\alpha_{n+1})_{j}
    \in F_{n+1}^{\jNeighborj \circ \term} \subseteq \Tk_{n+1}$
    is a function that separates $\beta$ and $\gamma$.

    Having proven the statement for every $n \in \N$, one can also easily see
    that it holds in the case $n = \infty$ from the definitions, cf.\ also
    the first case of the induction.
\end{proof}

A final application of the Stone-Weierstrass Theorem
and the Portmanteau Theorem yields that,
for a sequence of graphons, convergence of their $k$-WLDs
is equivalent to convergence of treewidth $k-1$
multigraph homomorphism densities.
\begin{lemma}
    \label{le:kWLDConvergence}
    Let $k \ge 1$.
    Let $(W_n)_{n}$ and $W \colon X \times X \to [0,1]$
    be a sequence of graphons and a graphon, respectively.
    Then, $\nu^k_{W_n} \rightarrow \nu^k_W$ if and only if
    $t(F, W_n) \rightarrow t(F, W)$
    for every multigraph~$F$ of treewidth
    at most $k - 1$.
\end{lemma}
\begin{proof}
    Note that
    the linear hull of $\Tk$
    is uniformly dense in $C(\Mk)$ by
    \Cref{le:tkSeparatesPoints}
    and the Stone-Weierstrass Theorem \cite[Theorem $2.4.11$]{Dudley2002}.
    Hence, we have
    \begin{align*}
        \nu^k_{W_n} \rightarrow \nu^k_W
        &\iff \int_{\Mk} f \, d \nu^k_{W_n} \rightarrow \int_{\Mk} f \, d \nu^k_W \text{ for every $f \in C(\Mk)$} \tag{Portmanteau Theorem}\\&\iff \int_{\Mk} f \, d \nu^k_{W_n} \rightarrow \int_{\Mk} f \, d \nu^k_W \text{ for every $f$ in the linear hull of $\Tk$}\\
        &\iff \int_{\Mk} f \, d \nu^k_{W_n} \rightarrow \int_{\Mk} f \, d \nu^k_W \text{ for every $f \in \Tk$} \tag{Linearity of the integral}\\
        &\iff t(\graphOfTerm, W_n) \rightarrow t(\graphOfTerm, W) \text{ for every $\term \in \treeClosureK$} \tag{\Cref{co:homDensitykWLDIntegral}}\\
        &\iff t(F, W_n) \rightarrow t(F, W) \text{ for every multigraph $F$ of tw.\ $\le k - 1$} \tag{\Cref{le:graphsOfTerms}}.
    \end{align*}
\end{proof}

Since $\Mk$ is Hausdorff, this in particular means
that the $k$-WLDs of two graphons
are equal if and only if
their homomorphism densities are.
\begin{corollary}
    \label{le:distinguishingTerm}
    Let $k \ge 1$
    and $U, W \colon X \times X \to [0,1]$ be graphons,
    Then, $\nu^k_U = \nu^k_W$ if and only if
    $t(F, U) = t(F, W)$
    for every multigraph~$F$ of treewidth at most $k-1$.
\end{corollary}

\subsection{Proof of Theorem~\ref{th:WLForGraphons}}
\label{sec:mainProof}

We are finally ready to prove \Cref{th:WLForGraphons}.
The majority of the work is already done and, at this point,
it is just about putting all the previous results together.
For easier readability, we restate the theorem here.
\maintheorem*
\begin{proof}
    (\ref{th:WLForGraphons:homomorphisms}) $\implies$ (\ref{th:WLForGraphons:DIDM}):
    This is just \Cref{le:distinguishingTerm}.

    (\ref{th:WLForGraphons:DIDM}) $\implies$ (\ref{th:WLForGraphons:minInvSubalgebra}):
    Let $R \coloneqq R^k_U \circ (R^k_W)^*$.
    By the assumption, $R$ is well defined,
    and by \Cref{le:ikWMarkovIso},
    it is a Markov isomorphism
    as the composition of two Markov isomorphisms.
    By \Cref{co:DIDMOperatorFamily}, we have
    \begin{align*}
        \naturalKFamilyOperatorsShortU/\kSubAlg_U\circ R= \naturalKFamilyOperatorsShortU/\kSubAlg_U\circ R^k_U \circ (R^k_W)^*= R^k_U \circ \TT^k_{\nu^k_U} \circ (R^k_W)^*
        &= R^k_U \circ \TT^k_{\nu^k_W} \circ (R^k_W)^*\\
        &= R^k_U \circ (R^k_W)^* \circ \naturalKFamilyOperatorsShortW/\kSubAlg_W\\
        &= R \circ \naturalKFamilyOperatorsShortW/\kSubAlg_W.
    \end{align*}
    Similarly, \Cref{le:DIDMOperatorP}
    yields that $R$ is permutation invariant.

    (\ref{th:WLForGraphons:minInvSubalgebra}) $\implies$ (\ref{th:WLForGraphons:markovOperator}):
    Set $S \coloneqq I_{\kSubAlg_U} \circ R \circ S_{\kSubAlg_W}$,
    which is a Markov operator as the composition of Markov operators.
    By \Cref{le:algebraSequence} (\ref{le:algebraSequence:TInvariant}), $\kSubAlg_U$ and $\kSubAlg_W$
    are $\naturalKFamilyOperatorsShortU$-
    and $\naturalKFamilyOperatorsShortW$-invariant, respectively.
    Hence,
    \begin{align*}
        \naturalKFamilyOperatorsShortU \circ S= \naturalKFamilyOperatorsShortU \circ I_{\kSubAlg_U} \circ R \circ S_{\kSubAlg_W}= I_{\kSubAlg_U} \circ \naturalKFamilyOperatorsShortU/\kSubAlg_U \circ R \circ S_{\kSubAlg_W}&= I_{\kSubAlg_U} \circ R \circ \naturalKFamilyOperatorsShortW/\kSubAlg_W \circ S_{\kSubAlg_W}\\
        &= I_{\kSubAlg_U} \circ R \circ S_{\kSubAlg_W} \circ \naturalKFamilyOperatorsShortW\\
        &= S \circ \naturalKFamilyOperatorsShortW.
    \end{align*}
    by \Cref{le:quoOperator} (\ref{le:quoOperator:invariantI}) and (\ref{le:quoOperator:invariant}).
    In a similar fashion,
    \Cref{le:algebraSequence} (\ref{le:algebraSequence:permutation}) implies that,
    if $R$ is permutation invariant, then so is $S$.

    (\ref{th:WLForGraphons:markovOperator}) $\implies$ (\ref{th:WLForGraphons:markovIsomorphism}):
    Follows immediately from \Cref{th:markovOperatorToIsomorphism}.

    (\ref{th:WLForGraphons:markovIsomorphism}) $\implies$ (\ref{th:WLForGraphons:homomorphisms}):
    We have
$t(\graphOfTerm, U)= t(\term, \naturalKFamilyOperatorsShortU \!/ \subAlgC)= t(\term, \naturalKFamilyOperatorsShortW \!/ \subAlgD)= t(\graphOfTerm, W)$,
for every $\term \in \treeClosureK$
    by
    \Cref{le:invariantQuotientPreservesHomomorphisms} and
    \Cref{le:markovEmbeddingPreservesHomomorphisms}.
Then, \Cref{le:graphsOfTerms} yields the claim.
\end{proof}

\subsection{Measure Hierarchies}
\label{sec:measureHierarchies}

\Cref{th:kWLGraphons} implies that
the sequence $\nu^1_W, \nu^2_W, \dots$
of $k$-WLDs of a graphon $W$
characterizes $W$ up to weak isomorphism since
every graph has some finite treewidth.
Let us explore this a bit more in depth
by combining all these $k$-WLDs into a single measure.

For $1 \le \ell \le k < \infty$, we define the projection $p^{k, \ell}$
from $\Mk$ to $\MOf{\ell}$
as follows:
First, inductively define the function $p^{k, \ell} \colon P^k_n \to P^\ell_n$
by defining $p^{k, \ell} \colon P^k_0 \to P^\ell_0$
by
$p^{k, \ell}((w_{ij})_{ij \in \binom{[k]}{2}}) \coloneqq
    (w_{ij})_{ij \in \binom{[\ell]}{2}}$
and,
for the inductive step,
by defining $p^{k, \ell} \colon P^k_{n+1} \to P^\ell_{n+1}$
by
$p^{k, \ell}((\nu_j)_{j \in [k]}) \coloneqq
    ({p^{k,\ell}}_*\nu_j)_{j \in [\ell]}$.
This is well-defined as every $p^{k, \ell}$ is continuous.
Second,
the function $p^{k, \ell} \colon P^k_n \to P^\ell_n$
directly extends to a function
$p^{k, \ell} \colon \Mk_n \to \MOf{\ell}_n$
by applying it component wise.
Finally, by then applying this function component-wise,
$p^{k, \ell}$ extends to a continuous function $p^{k, \ell} \colon \Mk \to \MOf{\ell}$.

Consider the \textit{inverse limit} of the spaces $\Mk$
and the projections $p^{k+1,k}$ for $k \ge 1$
defined by
\begin{equation*}
    \MInfty \coloneqq \Bigl\{(\alpha^k)_{k \ge 1} \in \prod_{k \ge 1} \Mk \,\Big|\,p^{k+1,k}(\alpha^{k+1}) = \alpha^{k} \text{ for every } k \ge 1\Bigr\}
\end{equation*}
with the $\sigma$-algebra $\Borel(\MInfty)$ generated by
the projections $p^{\infty,k} \colon \MInfty \to \Mk, \alpha \mapsto \alpha^k$ for every $k \ge 1$.
Note that this notation is justified as $\MInfty$ is again a standard Borel space \cite[Exercise $17.16$]{Kechris1995}.
As a product of a sequence of metrizable compact spaces,
$\prod_{k \ge 1} \Mk$ is metrizable \cite[Proposition $2.4.4$]{Dudley2002}
and also compact by Tychonoff's Theorem \cite[Theorem $2.2.8$]{Dudley2002}.
Since $p^{k+1,k}$ is continuous, this implies that $\MInfty$ is closed and,
hence, a metrizable compact space.
Let
\begin{equation*}
    \OWLOf{} \coloneqq \Bigl\{(\nu^k)_{k \ge 1} \in \prod_{k \ge 1} \WLk \,\Big|\, \nu^k = {p^{k+1,k}}_* \nu^{k+1} \text{ for every } k \ge 1\Bigr\},
\end{equation*}
where $\OWLk$ denotes the set of all $k$-WLDs.
Then, by the Kolmogorov Consistency Theorem \cite[Exercise $17.16$]{Kechris1995},
for every $\nu \in \OWLOf{}$,
there is a unique $\nu^\infty \in \probMeas(\MInfty)$ such that
${p^{\infty, k}}_* \nu^\infty = \nu^k$
for every $k \ge 1$.
\begin{lemma}
    \label{le:inftyWLDContinuity}
    Let $(\nu_n)_n$ be a sequence with $\nu_n \in \OWLOf{}$ and $\nu \in \OWLOf{}$.
    Then, $\nu_n^\infty \rightarrow \nu^\infty$
    if and only if $\nu_n^k \rightarrow \nu^k$ for every $k \ge 1$.
\end{lemma}
\begin{proof}
    The set $\bigcup_{1 \le k < \infty} C(\Mk) \circ p^{\infty, k}$
    is uniformly dense in $C(\MInfty)$ by
    the Stone-Weierstrass Theorem \cite[Theorem $2.4.11$]{Dudley2002},
    cf.\ also the proof of \Cref{le:continuity}.
    Hence, we have
    \begin{align*}
        \nu^\infty_n \rightarrow \nu^\infty
        &\iff \int_{\MInfty} f \, d \nu^\infty_n \rightarrow \int_{\MInfty} f \, d \nu^\infty \text{ for every $f \in C(\MInfty)$} \tag{Portmanteau Theorem}\\&\iff \int_{\MInfty} f \circ p^{\infty, k} \, d \nu^\infty_n \rightarrow \int_{\MInfty} f \circ p^{\infty, k} \, d \nu^\infty \text{ for all $k \ge 1$, $f \in C(\Mk)$}\\
        &\iff \int_{\MInfty} f \, d {p^{\infty, k}}_* \nu^\infty_n \rightarrow \int_{\MInfty} f \, d {p^{\infty, k}}_* \nu^\infty \text{ for all $k \ge 1$, $f \in C(\Mk)$}\\
        &\iff \int_{\MInfty} f \, d \nu^k_n\rightarrow \int_{\MInfty} f \, d \nu^k \text{ for all $k \ge 1$, $f \in C(\Mk)$}\\
        &\iff \nu^k_n \rightarrow \nu^k \text{ for every $k \ge 1$}. \tag{Portmanteau Theorem}\end{align*}
\end{proof}

One can show that, for every graphon $W \colon X \times X \to [0,1]$,
the sequence $(\nu^k_W)_{k \ge 1}$ of its $k$-WLDs
is in $\OWLOf{}$ and, hence,
yields a measure $\nu^\infty_W \in \probMeas(\MInfty)$.
Together, \Cref{le:kWLDConvergence} and
\Cref{le:inftyWLDContinuity} imply that
these measures induce the same topology
on the space of graphons
as multigraph homomorphism densities;
note that this topology is different from
the one induced by simple graph homomorphism densities,
cf.\ \cite[Exercise $10.26$]{Lovasz2012} or \cite[Lemma C.$2$]{Janson2013}.
\begin{corollary}
    \label{le:inftyWLDConvergence}
    Let $(W_n)_{n}$ and $W \colon X \times X \to [0,1]$
    be a sequence of graphons and a graphon, respectively.
    Then, the following are equivalent:
    \begin{enumerate}
        \item
            $\nu^\infty_{W_n} \rightarrow \nu^\infty_W$.
            \label{le:inftyWLDConvergence:inftyWLD}
        \item
            $t(F, W_n) \rightarrow t(F, W)$
            for every multigraph~$F$.
            \label{le:inftyWLDConvergence:multigraphs}
    \end{enumerate}
\end{corollary}

While simple graph and multigraph homomorphism densities
yield different topologies, they do not make a difference
for weak isomorphism:
two graphons
are weakly isomorphic
if and only if they have the same
multigraph homomorphism densities \cite[Corollary $10.36$]{Lovasz2012}.
Since $\measOne(\MInfty)$ is Hausdorff,
this yields the following corollary.

\begin{corollary}
    \label{co:weaklyIsomorphic}
    Let $U, W \colon X \times X \to [0,1]$ be graphons.
    Then,
    $\nu^\infty_U = \nu^\infty_W$
    if and only if
    $U$ and $W$ are weakly isomorphic.
\end{corollary}

\subsection{Linear Equations}
\label{sec:opHierarchies}

In this section, we elaborate the connection between
Characterization~(\ref{th:WLForGraphons:markovOperator}) of \Cref{th:WLForGraphons}
and the system $\Lkk(G, H)$ of linear equations mentioned in the introduction.
To this end, we first describe $\Lkk(G, H)$:
$G$ and $H$ are graphs and $\Lkk(G, H)$
has a variable $X_\pi$ for every set
$\pi \subseteq V(G) \times V(H)$
of size $\lvert \pi \rvert \le k$.
Such a set $\pi$ is called a \textit{partial isomorphism}
if the mapping it induces is injective and
preserves adjacency and non-adjacency.
Then, $\Lkk(G, H)$ is given by the following equations:
\begin{equation*}
    \Lkk(G, H) \colon
    \begin{cases}
        \displaystyle\sum_{v \in V(G)} X_{\pi \cup \{(v,w)\}} = X_\pi
        &\begin{aligned}[t]
            &\text{for every $\pi \subseteq V(G) \times V(H)$ of size}\\
            &\text{$\lvert \pi \rvert \le k-1$ and every $w \in V(H)$}
        \end{aligned}\\
        \displaystyle\sum_{w \in V(H)} X_{\pi \cup \{(v,w)\}} = X_\pi
        &\begin{aligned}[t]
            &\text{for every $\pi \subseteq V(G) \times V(H)$ of size}\\
            &\text{$\lvert \pi \rvert \le k-1$ and every $v \in V(G)$}
        \end{aligned}\\
        X_\emptyset = 1\\
        X_\pi = 0
        &\begin{aligned}[t]
            &\text{for every $\pi \subseteq V(G) \times V(H)$ of size $\lvert \pi \rvert \le k$}\\
            &\text{that is not a partial isomorphism}
        \end{aligned}
    \end{cases}
\end{equation*}
The graphs $G$ and $H$ are not distinguished by oblivious $k$-WL
if and only if $\Lkk(G,H)$ has a non-negative real solution.
The equivalence to precisely this system of linear equations is from
\cite{Dell2018},
although it is already implicit in earlier work
\cite{ImmermanLander1990, AtseriasManeva2013, GroheOtto2015}.

$\Lkk(G, H)$ is much closer related to
Characterization~(\ref{th:WLForGraphons:markovOperator}) of \Cref{th:WLForGraphons}
than it might seem at first glance:
The variables of $\Lkk(G, H)$
can be interpreted
as permutation-invariant
matrices on $V(G)^1 \times V(H)^1, \dots, V(G)^k \times V(H)^k$.
In \Cref{th:WLForGraphons},
instead of permutation-invariant operators on all spaces
$L^2(X^1, \mu^{\otimes 1}), \dots, \LTwoProdK$,
we only have a single permutation-invariant Markov operator
$S$ on $\LTwoProdK$.
In general, for an operator $S$ on $\LTwoProdK$,
defining
\begin{equation*}
    \stepDown \coloneqq \OperatorFrom{\forgetGraphOf{k}{k}} \circ S \circ \OperatorFrom{\introduceGraphOf{k}{k}}\end{equation*}
yields an operator on $\LTwoProdKM$.
It is easy to see
that
$(\stepDown)^* = \stepDownOf{S^*}$
since the adjoint of a forget graph is the corresponding introduce graph
and vice versa.
Moreover, if $S$ is permutation-invariant, this definition
is independent of the specific pair of forget and introduce graphs, i.e.,
we have
$\stepDown = \OperatorFrom{\jForgetj} \circ S \circ \OperatorFrom{\jIntroducej}$
for every $j \in [k]$
since
$\OperatorFrom{\jForget{k}} \circ T_{(k \dots j)} = \OperatorFrom{\jForgetj}$
and
$T_{(j \dots k)} \circ \OperatorFrom{\jIntroduce{k}} = \OperatorFrom{\jIntroducej}$.

\begin{lemma}
    \label{le:markovOperatorStepDown}
    Let $k \ge 1$ and
    $S$ be a
    permutation-invariant Markov operator on $\LTwoProdK$.
    Then, $\stepDown$ is a permutation-invariant Markov operator.
    Moreover,
    if $\OperatorFrom{\neighborGraphOf{k}{k}} \circ S = S \circ \OperatorFrom{\neighborGraphOf{k}{k}}$, then
    \begin{multicols}{2}
    \begin{enumerate}
        \item $S \circ \OperatorFrom{\introduceGraphOf{k}{k}} = \OperatorFrom{\introduceGraphOf{k}{k}} \circ \stepDown$,
            \label{le:markovOperatorStepDown:commutesI}
        \item $\OperatorFrom{\forgetGraphOf{k}{k}} \circ S = \stepDown \circ \OperatorFrom{\forgetGraphOf{k}{k}}$, and
            \label{le:markovOperatorStepDown:commutesF}
        \item $\OperatorFrom{\neighborGraphOf{k-1}{k-1}} \circ \stepDown = \stepDown \circ \OperatorFrom{\neighborGraphOf{k-1}{k-1}}$.
            \label{le:markovOperatorStepDown:commutesN}
    \end{enumerate}
    \end{multicols}
\end{lemma}
\begin{proof}
    First note that
    \begin{equation*}
        \stepDown \allOne_{X^{k-1}}= (\OperatorFrom{\forgetGraphOf{k}{k}} \circ S \circ \OperatorFrom{\introduceGraphOf{k}{k}}) \allOne_{X^{k-1}}= (\OperatorFrom{\forgetGraphOf{k}{k}} \circ S)\allOne_{X^{k}}= \OperatorFrom{\forgetGraphOf{k}{k}}\allOne_{X^{k}}= \allOne_{X^{k-1}},
    \end{equation*}
    where the last equality holds since $\mu$ is a probability measure.
    Since $S^*$ is also a Markov operator,
    we also obtain $(\stepDown)^* \allOne_{X^{k-1}} = \stepDownOf{S^*} \allOne_{X^{k-1}} = \allOne_{X^{k-1}}$.
    Let $f \in \LTwoProdKM$ with $f \ge 0$.
    Then, $\OperatorFrom{\introduceGraphOf{k}{k}} f = f \otimes \allOne_X \ge 0$, and hence,
    $(S \circ \OperatorFrom{\introduceGraphOf{k}{k}}) f \ge 0$.
    Therefore, also $\stepDown f = (\OperatorFrom{\forgetGraphOf{k}{k}} \circ S \circ \OperatorFrom{\introduceGraphOf{k}{k}}) f \ge 0$.
    Hence, $\stepDown$ is a Markov operator.
    For a permutation $\pi \colon [k-1] \to [k-1]$, we define
    the permutation $\pi' \colon [k] \to [k]$
    by $\pi'(i) \coloneqq \pi(i)$ for $i \in [k-1]$ and $\pi'(k) \coloneqq k$.
    Then,
    \begin{align*}
        T_\pi \circ \stepDown = T_\pi \circ \OperatorFrom{\jForget{k}} \circ S \circ \OperatorFrom{\jIntroduce{k}}= \OperatorFrom{\jForget{k}} \circ T_{\pi'} \circ S \circ \OperatorFrom{\jIntroduce{k}}&= \OperatorFrom{\jForget{k}} \circ S \circ T_{\pi'} \circ \OperatorFrom{\jIntroduce{k}}\\
        &= \OperatorFrom{\jForget{k}} \circ S \circ \OperatorFrom{\jIntroduce{k}} \circ T_{\pi}\\
        &= \stepDown \circ T_{\pi}.
    \end{align*}
    Hence, $\stepDown$ is permutation invariant.
    Now, assume that $\OperatorFrom{\neighborGraphOf{k}{k}} \circ S = S \circ \OperatorFrom{\neighborGraphOf{k}{k}}$.
    Then,
    \begin{align*}
        \OperatorFrom{\introduceGraphOf{k}{k}} \circ \stepDown = \OperatorFrom{\introduceGraphOf{k}{k}} \circ \OperatorFrom{\forgetGraphOf{k}{k}} \circ S \circ \OperatorFrom{\introduceGraphOf{k}{k}}= \OperatorFrom{\neighborGraphOf{k}{k}} \circ S \circ \OperatorFrom{\introduceGraphOf{k}{k}}&= S \circ \OperatorFrom{\neighborGraphOf{k}{k}} \circ \OperatorFrom{\introduceGraphOf{k}{k}}\\
        &= S \circ \OperatorFrom{\introduceGraphOf{k}{k}} \circ \OperatorFrom{\forgetGraphOf{k}{k}} \circ \OperatorFrom{\introduceGraphOf{k}{k}}\\
        &= S \circ \OperatorFrom{\introduceGraphOf{k}{k}},
    \end{align*}
    where the last equality holds since $\mu$ is a probability measure.
    Then, we also obtain \ref{le:markovOperatorStepDown:commutesF}
    by considering $S^*$ and $\stepDownOf{S^*}$
    and then taking adjoints.
    Finally,
    note that the permutation invariance of $S$ yields that we also have
    $\OperatorFrom{\neighborGraphOf{k}{k-1}} \circ S = S \circ \OperatorFrom{\neighborGraphOf{k}{k-1}}$.
    Moreover, observe that
    $\neighborGraphOf{k-1}{k-1} \circ \forgetGraphOf{k}{k}
    = \forgetGraphOf{k}{k} \circ \neighborGraphOf{k}{k-1}$.
Hence,
    \begin{align*}
        \OperatorFrom{\neighborGraphOf{k-1}{k-1}} \circ \stepDown =  \OperatorFrom{\neighborGraphOf{k-1}{k-1}} \circ \OperatorFrom{\forgetGraphOf{k}{k}} \circ S \circ \OperatorFrom{\introduceGraphOf{k}{k}}=  \OperatorFrom{\forgetGraphOf{k}{k}} \circ \OperatorFrom{\neighborGraphOf{k}{k-1}} \circ S \circ \OperatorFrom{\introduceGraphOf{k}{k}}&=  \OperatorFrom{\forgetGraphOf{k}{k}} \circ S \circ \OperatorFrom{\neighborGraphOf{k}{k-1}} \circ \OperatorFrom{\introduceGraphOf{k}{k}}\\
        &=  \OperatorFrom{\forgetGraphOf{k}{k}} \circ S \circ \OperatorFrom{\introduceGraphOf{k}{k}} \circ \OperatorFrom{\neighborGraphOf{k-1}{k-1}}\\
        &= \stepDown \circ \OperatorFrom{\neighborGraphOf{k-1}{k-1}}.
    \end{align*}
\end{proof}

Given a permutation-invariant Markov operator
$S$ on $\LTwoProdK$,
repeated applications of \Cref{le:markovOperatorStepDown} yield
a sequence $S_0, \dots, S_k$
of permutation-invariant Markov operators $S_i$ on $\LTwoProdI$
by letting $S_k \coloneqq S$ and $S_{i-1} \coloneqq \stepDownOf{S_i}$ for $i \in [k]$,
which we call the \textit{operator hierarchy defined by $S$}.
The following lemma shows that the equation
$T_{\neighborGraphOf{k}{k}} \circ S = S \circ T_{\neighborGraphOf{k}{k}}$
of \Cref{th:WLForGraphons}
is just a way of formulating graphon analogues for
the first three equations in the definition of $\Lkk$.
\begin{lemma}
    \label{le:operatorHierarchy}
    If $S$ is a permutation-invariant Markov operator
    on $\LTwoProdK$
    such that $T_{\neighborGraphOf{k}{k}} \circ S = S \circ T_{\neighborGraphOf{k}{k}}$,
    then its operator hierarchy satisfies
    \begin{enumerate}
        \item $S_{i}(f \otimes \allOne_X) = S_{i-1}(f) \otimes \allOne_X$
              for every $f \in \LTwoProdIM$ and every $i \in [k]$,
            \label{le:markovOpStepDownSequence:one}
        \item $S^*_{i}(f \otimes \allOne_X) = S^*_{i-1}(f) \otimes \allOne_X$
              for every $f \in \LTwoProdIM$ and every $i \in [k]$,
            \label{le:markovOpStepDownSequence:adjointOne}
        \item $S_0$ is the identity operator, and
            \label{le:markovOpStepDownSequence:identity}
        \item $S_i \ge 0$ for every $i \in [k]$.
            \label{le:markovOpStepDownSequence:positive}
    \end{enumerate}
    Conversely,
    if $S_0, \dots, S_k$ is a sequence of permutation-invariant operators
    $S_i$ on $\LTwoProdI$ satisfying the four conditions above,
    then $S_0, \dots, S_k$ are Markov operators
    satisfying
    $\OperatorFrom{\neighborGraphOf{i}{i}} \circ S_i= S_i \circ \OperatorFrom{\neighborGraphOf{i}{i}}$.
\end{lemma}
\begin{proof}
    The forward direction follows inductively from \Cref{le:markovOperatorStepDown}.
    For backward direction,
    note that,
    by definition of $\introduceGraphOf{i}{i}$,
    the first condition just states that
    $S_i \circ \OperatorFrom{\introduceGraphOf{i}{i}} = \OperatorFrom{\introduceGraphOf{i}{i}}  \circ S_{i-1}$;
    the second condition is the analogous statement for forget graphs.
    This immediately yields the backward direction.
\end{proof}

As a final remark,
we note that in addition to \Cref{le:markovOperatorStepDown},
one can also easily prove that, if
$\OperatorFromTo{\adjacencyGraphOf{k}{12}}{U} \circ S = S \circ \OperatorFromTo{\adjacencyGraphOf{k}{12}}{W}$
holds for graphons
$U,W \colon X \times X \to [0,1]$ and $k \ge 3$, then
we also have
$\OperatorFromTo{\adjacencyGraphOf{k-1}{12}}{U} \circ \stepDown = \stepDown \circ \OperatorFromTo{\adjacencyGraphOf{k-1}{12}}{W}$.
This inductively extends to operator hierarchies,
and it is easy to see that
this requirement corresponds to the fourth equation
in the definition of $\Lkk$ (concerning partial isomorphisms);
we are just missing the injectivity that a partial isomorphism requires,
which is not important
as long as our standard Borel space is atom-free.
Together with \Cref{le:operatorHierarchy},
this shows how to restore the characterization
of oblivious $k$-WL indistinguishability by $\Lkk$
for graphs $G$ and $H$ from \Cref{th:WLForGraphons}.
 
\section{Simple Weisfeiler-Leman Indistinguishability}
\label{sec:simpleWL}

\Cref{th:kWLGraphons} shows that
oblivious $k$-WL corresponds to bounded treewidth
\textit{multi}graph homomorphism densities.
The reason for this are the atomic types used by $k$-WL, or more
accurately in our setting,
the adjacency graphs
since
subsequent applications of the same adjacency graph $\ijAdjacencyij$
to a term result in parallel edges.
This cannot be prevented by simply disallowing such subsequent applications:
for the application of the Stone-Weierstrass Theorem
in the proof of \Cref{th:kWLGraphons},
it is crucial that the set $\Tk$ of
homomorphism functions is closed under multiplications.
However, to achieve this closure under multiplications, we close the set of terms
under Schur products, which may introduce
parallel edges if we have edges between input vertices,
cf.\ \Cref{fig:parallelEdges}.
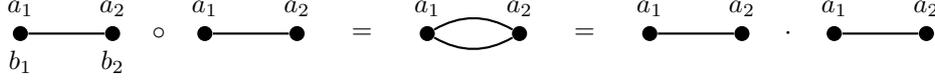
To prevent this way of introducing parallel edges,
we have to prevent edges from being added between
input vertices in the first place.

In \Cref{sec:obliviousSimpleWL}, we show
how \Cref{th:kWLGraphons} and its proof
have to be adapted for simple graph homomorphism densities.
To this end, we introduce \textit{simple (oblivious) $k$-WL}.
Not surprisingly, the definitions become more similar to
color refinement and the ones
of Greb\'ik and Rocha \cite{GrebikRocha2021}.
For the sake of brevity, we only
include proofs that significantly differ from their counterpart in \Cref{sec:WL}.
We also briefly show how
\textit{simple non-oblivious $k$-WL}
can be defined
in \Cref{sec:nonObliviousSimpleWL}.

\subsection{Simple Oblivious Weisfeiler-Leman}
\label{sec:obliviousSimpleWL}

\begin{figure}
    \centering
    \begin{tikzpicture}
		\node[vertex, label={90:$a_1$}, label={270:$b_1$}] (N11) {};
        \node[vertex, fill=white, draw=white, right = of N11] (N12g) {};
		\node[vertex, label={90:$a_2$}, above = 0.1cm of N12g] (N12a) {};
		\node[vertex, label={270:$b_2$}, below = 0.1cm of N12g] (N12b) {};
        \node[left = 0.5cm of N11, scale = 1.0] (NLabel){$\adjNeiGraphOf{2}{2,\{1\}}\colon$};
        \path[draw, thick] (N12b) edge (N11);

		\node[vertex, label={90:$a_1$}, label={270:$b_1$}, right = 5.0cm of N12g] (N1) {};
        \node[vertex, fill=white, draw=white, right = of N1] (N2g) {};
		\node[vertex, label={90:$a_2$}, above = 0.1cm of N2g] (N2a) {};
		\node[vertex, label={270:$b_2$}, below = 0.1cm of N2g] (N2b) {};
		\node[vertex, label={90:$a_3$}, label={270:$b_3$}, right = of N2g] (N3) {};
        \node[left = 0.5cm of N1, scale = 1.0] (NLabel){$\adjNeiGraphOf{3}{2, \{1,3\}}\colon$};
        \path[draw, thick] (N2b) edge (N1);
        \path[draw, thick] (N2b) edge (N3);
    \end{tikzpicture}
    \caption{The graphs $\adjNeiGraphOf{2}{2,\{1\}}$ and $\adjNeiGraphOf{3}{2, \{1,3\}}$.}
\end{figure}

To prevent edges from being added between input vertices,
we only allow certain combinations
of adjacency and neighbor graphs;
after a sequence of adjacency graphs connecting a vertex $j$ to other vertices,
we immediately follow up with a $j$-neighbor graph.
Formally,
for every $(j, V)$ in the set
$S^k \coloneqq \Set{(j, V) \mid j \in [k] \text{ and } V \subseteq [k] \setminus \{j\}}$,
define
the bi-labeled graph
\begin{equation*}
    \jAdjNei \coloneqq \jNeighborj \circ \bigcirc_{i \in V} \ijAdjacencyij
     \in \graphsOf{k,k}.
\end{equation*}
We note that this is well-defined since the composition
of adjacency graphs is associative.
Let
$\simpleAdjNeisK \coloneqq \Set{\jAdjNei \mid
(j,V) \in S^k}
\subseteq \graphsOf{k,k}$
be the set of all these bi-labeled graphs.
We have to be a bit cautious as, in general, these graphs
are not symmetric
and, hence, their graphon operators are not self-adjoint;
in general, the set $\simpleAdjNeisK$ is not even closed under transposition.
Note that, by definition, the $\jAdjNei$-graphon operator
of a graphon $W$ is given by
\begin{equation*}
    (\OperatorFromTo{\jAdjNei}{W} f) (\bar{x}) = \int\limits_{X} \notsovast(\prod_{i \in V} W(x_i, y) \notsovast) \cdot f \circ \bar{x}[y/j] \dmu(y)
\end{equation*}
for $\mu^{\otimes k}$-almost every $\bar{x} \in X^k$.
Analogously to \Cref{le:graphsOfTerms},
one can observe that the underlying graphs of $\graphOfTerm$
for terms $\term \in \simpleTreeClosureK$
are, up to isolated vertices,
precisely the simple graphs of treewidth at most $k-1$.
Basically, when constructing a term from a nice tree decomposition,
we just add all edges that are missing for a vertex whenever that vertex is forgotten.
We note that we do not miss any edges this way, i.e.,
every edge of the original graph is present in the term and
added at some point,
since
the bag at the root node of a nice tree decomposition is the empty set.

For the sake of brevity, we write
$\naturalSKFam \coloneqq \OperatorFamilyFromTo{\simpleAdjNeisK}{W}$
for a graphon $W$.
Then, define $\skSubAlg_{W, n} \in \subAlgsk$ for every $n \in \N$ by setting
$\skSubAlg_{W,0} \coloneqq \left\langle \Set{\emptyset, X^k} \right\rangle$,
$\skSubAlg_{W, n + 1} \coloneqq \naturalSKFam(\skSubAlg_{W, n}) \text{ for every } n \in \N$, and
finally,
$\skSubAlg_W \coloneqq \skSubAlg_{W, \infty} \coloneqq \langle \bigcup_{n \in \N} \skSubAlg_{W, n} \rangle$.
Then, analogously to \Cref{le:algebraSequence},
one can show that $\skSubAlg_W$ is permutation-invariant
and the
minimum $\naturalSKFam$-invariant
$\mu^{\otimes k}$-relatively complete
sub-$\sigma$-algebra of $\Borel^{\otimes k}$.
We now deviate a bit from the definition of $W$-invariance
and call a $\mu^{\otimes k}$-relatively complete sub-$\sigma$-algebra $\subAlg \in \subAlgsk$
\textit{simply $W$-invariant}
if $\subAlg$ is invariant
for every operator in the family $(\naturalSKFam)_{\skSubAlg_W}$,
i.e.,
$\subAlg$ is $(\OperatorFromTo{\bm{F}}{W})_{\skSubAlg_W}$-invariant
for every $\bm{F} \in \simpleAdjNeisK$.
The reason for this is that, since $\naturalSKFam$
is not closed under taking adjoints,
$\skSubAlg_W$ might not be invariant under these adjoints.
In contrast, $\skSubAlg_W$ is trivially both
$(\naturalSKFam)_{\skSubAlg_W}$-invariant
and $(\naturalSKFam)^*_{\skSubAlg_W}$-invariant.
In fact, it is easy to see that
$\skSubAlg_W$ is also the minimum simply $W$-invariant
$\mu^{\otimes k}$-relatively complete sub-$\sigma$-algebra of $\Borel^{\otimes k}$.

For a separable metrizable space $(X, \Topology)$,
let $\measOne(X)$ denote the set of all measures
of total mass at most $1$.
We endow $\measOne(X)$ with a topology analogously to $\probMeas(X)$,
i.e., with the topology generated by the maps
$\mu \mapsto \int f \dmu$ for $f \in C_b(X)$.
Then, for measures that all have the same total mass, the Portmanteau Theorem
is still applicable as we can scale them to have total mass of one.
Let $\iPsk_0 \coloneqq \Set{1}$ be the one-point space and inductively define
\begin{equation*}
    \Msk_n \coloneqq \prod_{i \le n} \iPsk_i \text{ and }\iPsk_{n+1} \coloneqq \left(\measOne\left(\Msk_n\right)\right)^{S^k}
\end{equation*}
for every $n \in \N$.
Let $\Msk \coloneqq \Msk_\infty \coloneqq \prod_{n \in \N} \iPsk_i$ and,
for $n \le m \le \infty$,
let
$p_{m,n} \colon \Msk_m \to \Msk_n$ be the natural projection.
Finally, define
\begin{equation*}
    \Psk \coloneqq \Set{\alpha \in \Msk \mid (\alpha_{n+1})_{(j,V)} = (p_{n+1,n})_* (\alpha_{n+2})_{(j,V)} \text{ for all }  (j,V) \in S^k, n \in \N}.
\end{equation*}

By the Kolmogorov Consistency Theorem \cite[Exercise $17.16$]{Kechris1995},
for all $\alpha \in \Pk$ and $(j,V) \in S^k$,
there is a unique measure $\mu^\alpha_{j,V} \in \probMeas(\Mk)$ such that
$(p_{\infty, n})_* \mu^\alpha_{j,V} = (\alpha_{n+1})_{(j,V)}$
for every $n \in \N$.
Analogously to \Cref{le:continuity},
the set
$\Psk$ is closed in $\Msk$ and, for every $(j,V) \in S^k$,
the mapping $\Psk \to \probMeas(\Msk), \alpha \mapsto \mu^\alpha_{j,V}$
is continuous.
To adapt the definition of a $k$-WLD, we add
a third requirement of absolute continuity
and Radon-Nikodym derivatives, cf.\ the definition
of distributions over iterated degree measures \cite{GrebikRocha2021}.
\begin{definition}
    Let $k \ge 1$.
    A measure $\nu \in \probMeas(\Msk)$ is called a \textit{simple $k$-Weisfeiler-Leman distribution (simple $k$-WLD)} if
    \begin{enumerate}
        \item $\nu(\Psk) = 1$,
        \item $\int_{\Msk} f \, d \nu = \int_{\Msk} \Big(\int_{\Msk} f \, d\mu^\alpha_{j,\emptyset}\Big) \, d \nu(\alpha)$ for all bounded measurable $f \colon \Msk \to \R$, $j \in [k]$, and
        \item $\mu^\alpha_{j,V} \acle \mu^\alpha_{j,\emptyset}$ and $0 \le \frac{d\mu^\alpha_{j,V}}{d\mu^\alpha_{j,\emptyset}} \le 1$ for $\nu$-almost every $\alpha \in \Msk$ and every $(j,V)\in S^k$.
    \end{enumerate}
\end{definition}

Let $W \colon X \times X \to [0,1]$ be a graphon.
Define $\isk_{W,0} \colon X^k \to \Msk_0$ by
$\isk_{W,0}(\bar{x}) \coloneqq 1$
for every $\bar{x} \in X^k$.
Inductively define $\isk_{W, n+1} \colon X^k \to \Msk_{n+1}$ by
\begin{equation*}
    \isk_{W,n+1}(\bar{x}) \coloneqq \left(\isk_{W,n}(\bar{x}),\notsovast(
        A \mapsto \int_{{\isk_{W,n}}^{-1}(A)_{\bar{x}[/j]}} \prod_{i \in V} W(x_i, y) \dmu(y)
    \notsovast)_{(j,V) \in S^k}\right)
\end{equation*}
for every $\bar{x} \in X^k$.
Then, let $\isk_W = \isk_{W,\infty} \colon X^k \to \Msk$ be the mapping
defined by
$(\isk_W(\bar{x}))_n \coloneqq (\isk_{W, \infty}(\bar{x}))_n \coloneqq (\isk_{W,n}(\bar{x}))_n$
for all $n \in \N$, $\bar{x} \in X^k$.
Finally, let
$\nusk_W \coloneqq {\isk_W}_{*} \mu^{\otimes k} \in \probMeas(\Msk)$
be the push-forward of $\mu^{\otimes k}$ via $\isk_W$.
Analogously to
\Cref{le:minAlgebraIKWMeasurable},
one can show that
\begin{equation*}
    \skSubAlg_{W,n} = \left\langle \Set{{\isk_{W,n}}^{\invSpace-1}(A) \mid A \in \Borel(\Msk_n)} \right\rangle
\end{equation*}
for $n \in \NInfty$.
Defining
$\RskW \coloneqq S_{\skSubAlg_W} \circ T_{\isk_W}$
yields a Markov isomorphism
from $\LTwoMskNuskW$
to $\LTwoProdKQuoOf{\skSubAlg_W}$,
cf.\ \Cref{le:ikWMarkovIso}.
Let us explicitly state the adaptation of \Cref{le:nuKWIskWLD}
since the proof requires some additional work.
\begin{lemma}
    \label{le:nuSKWIsskWLD}
    Let $k \ge 1$
    and $W \colon X \times X \to [0,1]$ be a graphon.
    Then,
    \begin{enumerate}
        \item $\isk_W(X^k) \subseteq \Psk$, and \label{le:nuSKWIsskWLD:subsetPsk}
        \item $\mu_{j, \emptyset}^{\isk_W(\bar{x})} = (\isk_{W} \circ \jsection)_*\mu$
            for all $j \in [k]$, $\bar{x} \in X^k$, \label{le:nuSKWIsskWLD:jNeighbors}
        \item $\nusk_W$ is a simple $k$-WLD. \label{le:nuSKWIsskWLD:skWLD}
    \end{enumerate}
\end{lemma}
\begin{proof}
    The proof of (\ref{le:nuSKWIsskWLD:subsetPsk}) is analogous to
    \Cref{le:nuKWIskWLD} (\ref{le:nuKWIskWLD:subsetPk}).
    For (\ref{le:nuSKWIsskWLD:jNeighbors}), observe that
    $\mu_{j, \emptyset}^{\isk_W(\bar{x})}$ is a probability measure.
    Then, the proof is analogous to \Cref{le:nuKWIskWLD} (\ref{le:nuKWIskWLD:jNeighbors}).
    For (\ref{le:nuSKWIsskWLD:skWLD}),
    we get $\nusk_W(\Psk) = 1$
    and
    $\int_{\Msk} f \,d \nusk = \int_{\Msk} \Bigl(\int_{\Msk} f \,d \mu^\alpha_{j, \emptyset} \Bigr) \, d \nusk_W(\alpha)$
    for every bounded measurable $f \colon \Msk \to \R$ and every $j \in [k]$
    as in the proof of
    \Cref{le:nuKWIskWLD} (\ref{le:nuKWIskWLD:kWLD}).

    Let $(j, V) \in S^k$.
    Let $\bar{x} \in X^k$
    and let
    \begin{equation*}
        \subAlg \coloneqq \Bigl\langle\Set{\jsection^{-1}\big({\isk_W}^{\!-1}(A)\big) \mid A \in \Borel(\Msk)}\Bigr\rangle
    \end{equation*}
    be the minimum $\mu$-relatively complete sub-$\sigma$-algebra that
    makes $\isk_W \circ \jsection$ measurable.
    Then, $\E(y \mapsto \prod_{i \in V} W(x_i, y) \mid \subAlg) \in \LTwoSub$ and hence,
    by \Cref{co:quotientSpaceUnique},
    there is a measurable function $g \colon X \to \R$
    such that $\E(y \mapsto \prod_{i \in V} W(x_i, y) \mid \subAlg) = g \circ \isk_W \circ \jsection$ $\mu$-almost everywhere.
    Note that $0 \le g \le 1$ holds $\mu$-almost everywhere.
    For every $n \in \N$ and every  $A \in \Borel(\Msk_n)$, we have
    \begin{align*}
        \mu_{j,V}^{\isk_W(\bar{x})}(p_{\infty, n}^{-1}(A))
        &= (p_{\infty, n})_* \mu_{j,V}^{\isk_W(\bar{x})}(A)\\
        &= ((\isk_W(\bar{x}))_{n+1})_{(j,V)}(A)\\
        &= ((\isk_{W,n+1}(\bar{x}))_{n+1})_{(j,V)}(A)\\
        &= \int_{{\isk_{W,n}}^{\invSpace-1}(A)_{\bar{x}[/j]}} \prod_{i \in V} W(x_i, y) \dmu(y)\\
        &= \int_{\jsection^{-1} ({\isk_{W}}^{\!-1}(p_{\infty,n}^{-1}(A)))} \prod_{i \in V} W(x_i, y) \dmu(y)\\
        &= \int_{\jsection^{-1} ({\isk_{W}}^{\!-1}(p_{\infty,n}^{-1}(A)))} \E \left(y \mapsto \prod_{i \in V} W(x_i, y) \mid \subAlg \right) \dmu \tag{\Cref{cl:conditionalExpectation}}\\
        &= \int_{\jsection^{-1} ({\isk_{W}}^{\!-1}(p_{\infty,n}^{-1}(A)))} g \circ \isk_W \circ \jsection \dmu\\
        &= \int_{p_{\infty,n}^{-1}(A)} g\,d(\isk_W \circ \jsection )_*\mu\\
        &= \int_{p_{\infty,n}^{-1}(A)} g\,d\mu_{j, \emptyset}^{\isk_W(\bar{x})}.
    \end{align*}
    Since
    $\bigcup_{n \in \N} \Set{p_{\infty,n}^{-1}(A) \mid A \in \Borel(\Msk_n)}$
    generates $\Borel(\Msk)$,
    the $\pi$-$\lambda$ Theorem \cite[Theorem $10.1$ iii)]{Kechris1995} yields that
    $\mu_{j,V}^{\isk_W(\bar{x})}(A)
    = \int_{A} g\,d\mu_{j, \emptyset}^{\isk_W(\bar{x})}$
    for every $A \in \Borel(\Msk)$.
    Therefore,
    $\mu^\alpha_{j,V} \acle \mu^\alpha_{j,\emptyset}$ and $0 \le \frac{d\mu^\alpha_{j,V}}{d\mu^\alpha_{j,\emptyset}} \le 1$ for every $\alpha \in \isk_W(X^k)$.
    By definition of $\nusk_W$, this holds $\nusk_W$-almost everywhere.
    Hence, $\nusk_W$ is a simple $k$-WLD.
\end{proof}

Let $\nu \in \probMeas(\Msk)$ be a simple $k$-WLD and
$(j,V) \in S^k$.
By definition of a $k$-WLD, we have
$0 \le \frac{d \mu^\alpha_{j,V}}{d \mu^\alpha_{j, \emptyset}} \le 1$
for $\nu$-almost every $\alpha \in \Msk$.
Hence, analogously to
\Cref{le:operatorNWellDefined},
one can show that
setting
\begin{equation*}
    (\OperatorFromTo{\jAdjNei}{\nu} f) (\alpha) \coloneqq \int_{\Msk} \frac{d \mu^\alpha_{j,V}}{d \mu^\alpha_{j, \emptyset}} \cdot f \, d \mu^\alpha_{j, \emptyset}
    = \int_{\Msk} f \, d \mu^\alpha_{j, V}
\end{equation*}
for all $f \in \LInftyMskNu$, $\alpha \in \Msk$
defines an $L^\infty$-contraction
that uniquely extends to an $L^2$-contraction
$\LTwoMskNu \to \LTwoMskNu$.

\begin{lemma}
    \label{le:SWLDOperatorS}
    Let $k \ge 1$
    and $W \colon X \times X \to [0,1]$ be a graphon.
    For every $\bm{S} \in \simpleAdjNeisK$,
    \begin{multicols}{2}
    \begin{enumerate}
        \item $\OperatorFromTo{\bm{S}}{W} \circ T_{\isk_W} = T_{\isk_W} \circ \OperatorSNuW$, \label{le:SWLDOperatorS:normalOp}
        \item $(\OperatorFromTo{\bm{S}}{W})_{\kSubAlg_W} \!\circ T_{\isk_W} \!= T_{\isk_W} \!\circ \OperatorSNuW$, and \label{le:SWLDOperatorS:expOp}
        \item $\OperatorFromTo{\bm{S}}{W} / \kSubAlg_W \circ \RskW = \RskW \circ \OperatorSNuW$. \label{le:SWLDOperatorS:iso}
    \end{enumerate}
    \end{multicols}
\end{lemma}
\begin{proof}
    Let $(j, V) \in S^k$ such that $\bm{S} = \jAdjNei$.
    For $\bar{x} \in X^k$, let
    $\subAlg_{\bar{x}}$ denote the minimum $\mu$-relatively complete
    sub-$\sigma$-algebra that makes $\isk_W \circ \jsection$ measurable.
    As seen in the proof of \Cref{le:nuSKWIsskWLD},
    we have
    \begin{equation*}
        \E(y \mapsto \prod_{i \in V} W(x_i, y) \mid \subAlg_{\bar{x}}) = \frac{d\mu_{j,V}^{\isk_W(\bar{x})}}{d\mu_{j, \emptyset}^{\isk_W(\bar{x})}} \circ \isk_W \circ \jsection
    \end{equation*}
    $\mu$-almost everywhere.
    Then, we have
    \begin{align*}
        (T_{\isk_W} \circ \OperatorSNuW f) (\bar{x})
        &= \int_{\Msk} \frac{d \mu^{\isk_W(\bar{x})}_{j,V}}{d \mu^{\isk_W(\bar{x})}_{j, \emptyset}} \cdot f\, d (\isk_W \circ \jsection)_* \mu \tag{Definition and \Cref{le:nuSKWIsskWLD} (\ref{le:nuKWIskWLD:jNeighbors})}\\
        &= \int_{X} \E(y \mapsto \prod_{i \in V} W(x_i, y) \mid \subAlg_{\bar{x}}) \cdot f \circ \isk_W \circ \jsection \dmu\\
        &= \int_{X} \prod_{i \in V} W(x_i, y)\cdot \E(f \circ \isk_W \circ \jsection \mid \subAlg_{\bar{x}})(y)  \dmu(y) \tag{\Cref{cl:conditionalExpectation}}\\
        &= \int_{X} \prod_{i \in V} W(x_i, y)\cdot f \circ \isk_W \circ \bar{x}[y/j]  \dmu(y)\\
        &= (\OperatorFromTo{\bm{S}}{W} \circ T_{\isk_W} f) (\bar{x})
    \end{align*}
    for every $f \in \LInftyMskNu$ and $\mu^{\otimes k}$-almost every
    $\bar{x} \in X^k$.
    As $\LInftyMskNuskW$ is dense in $\LTwoMskNuskW$, this implies (\ref{le:SWLDOperatorS:normalOp}).
    From there on, (\ref{le:SWLDOperatorS:expOp}) and (\ref{le:SWLDOperatorS:iso})
    are analogous to \Cref{le:DIDMOperatorN} (\ref{le:DIDMOperatorN:expOp})
    and (\ref{le:DIDMOperatorN:iso}), respectively.
\end{proof}

For $k \ge 1$ and a simple $k$-WL distribution $\nu \in \probMeas(\Msk)$, let
$\TT_\nu \coloneqq (\OperatorFromTo{\bm{S}}{\nu})_{\bm{S} \in \simpleAdjNeisK}$.
Then, for a graphon $W \colon X \times X \to [0,1]$,
we have
\begin{align*}
    &\naturalSKFam/\skSubAlg_W \circ \RskW = \RskW \circ \TTsk_{\nusk_W}&
    &\text{and}&
    &{\naturalSKFam}^*/\skSubAlg_W \circ \RskW = \RskW \circ \TTsk_{\nusk_W}^*,&
\end{align*}
where the first equation is just
\Cref{le:SWLDOperatorS}
and the second equation follows from the first
since $\Rsk$ is a Markov isomorphism.
As before, a permutation $\pi \colon [k] \to [k]$
naturally extends to a measurable bijection $\pi \colon \Msk \to \Msk$,
and the $\pi$-invariance, and more general the permutation invariance,
of a simple $k$-WLD can be defined
analogously to \Cref{sec:operatorsOnMeasures}.
The analogous result to \Cref{le:DIDMOperatorP} holds as well;
in particular, $\nusk_W$ is permutation invariant for a graphon $W$.
Let $\subAlg \in \subAlgsk$ be simply $W$-invariant;
recall that this definition is a bit artificial
as it means that
$\subAlg$ is $(\naturalSKFam)_{\skSubAlg_W}$-invariant.
\Cref{le:invariantQuotientPreservesHomomorphisms}
can then be adapted to the also somewhat convoluted statement
that
\begin{equation*}
    t(\term, \TT_{\nusk_W})
    = t(\term, ((\naturalSKFam)_{\skSubAlg_W})_\subAlg)= t(\term, (\naturalSKFam)_{\skSubAlg_W} \!/ \subAlg)= t(\term, \naturalSKFam)= t(\graphOfTerm, W)
\end{equation*}
holds for every $\term \in \simpleTreeClosureK$.
To prove this, one has to apply
\Cref{le:markovEmbeddingPreservesHomomorphisms} twice
this time:
first, to get from $\naturalSKFam$ to $(\naturalSKFam)_{\skSubAlg_W}$ and, second, to get
from there to $((\naturalSKFam)_{\skSubAlg_W})_\subAlg$
and $(\naturalSKFam)_{\skSubAlg_W} \!/ \subAlg$.

For a term $\term \in \simpleTreeClosureK$
and every $n \in \N$ with $n \ge h(\term)$,
the set $F^\term_n$ of functions $\Msk_n \to [0,1]$
is defined similarly to
\Cref{def:setOfHomFuns}.
More precisely,
while we could just use the old definition,
it can actually be simplified
as the distinct cases
for adjacency and neighbor graphs
can be subsumed by the function
\begin{equation*}
    \alpha \mapsto \int_{\Msk_{n}} f \, d(\alpha_{n+1})_{(j,V)}
    \in F_{n+1}^{\jAdjNei \circ \term}
\end{equation*}
for every $f\in F_n^{\term}$ and every $j \in [k]$.
From there, we analogously obtain the set
$F^\term$ of continuous
functions $\Msk \to [0,1]$.
\Cref{le:allHomFunsEqual}
and \Cref{co:homDensitykWLDIntegral}
adapt in a straight-forward fashion.

For every $n \in \NInfty$, define
$\Tsk_n \coloneqq \bigcup_{{\term \in \simpleTreeClosureK,\\ h(\term) \le n}} F_n^\term$
and abbreviate $\Tsk \coloneqq \Tsk_\infty$.
\Cref{le:tkSeparatesPoints} also adapts easily, i.e.,
for every $n \in \NInfty$,
the set $\Tsk_n$ is closed under multiplication, contains $\allOne_{\Msk_n}$,
and separates points of $\Msk_n$.
Here, one has to observe
that the all-one function distinguishes two
measures if their total mass is different,
which means that the Portmanteau Theorem is still applicable in this case.
From there, we obtain the following
analogue to \Cref{le:kWLDConvergence}.

\begin{lemma}
    \label{le:skWLDConvergence}
    Let $k \ge 1$.
    Let $(W_n)_{n}$ and $W \colon X \times X \to [0,1]$
    be a sequence of graphons and a graphon, respectively.
    Then, $\nusk_{W_n} \rightarrow \nusk_W$ if and only if
    $t(F, W_n) \rightarrow t(F, W)$
    for every simple graph~$F$ of treewidth
    at most $k - 1$.
\end{lemma}

Since $\probMeas(\Msk)$
is Hausdorff, this also means that
the simple $k$-WLDs of two graphons are equal if and only if
their treewidth $k-1$ simple graph homomorphism densities are.
With the
Counting Lemma \cite[Lemma $10.23$]{Lovasz2012},
we also obtain the following additional corollary.

\begin{corollary}
    Let $k \ge 1$.
    The mapping $\mathcal{W}_0 \to \probMeas(\Msk), W \mapsto \nusk_W$
    is continuous
    when $\mathcal{W}_0$ is endowed with the cut distance.
\end{corollary}

We note that the same reasoning does not work for multigraphs
as the Counting Lemma does not hold for multigraphs.
Moreover, the above corollary does not hold for multigraphs
since convergence of simple graph homomorphism densities
does not imply convergence of multigraph homomorphism densities,
cf.\ \cite[Exercise $10.26$]{Lovasz2012} or \cite[Lemma C.$2$]{Janson2013}.

Having outlined the necessary changes for simple graphs,
we obtain the following variant of \Cref{th:kWLGraphons}
for simple graph homomorphism densities.
Note the inelegant characterization via Markov operators,
which is quite artificial in this case;
this again stems from the fact that the family $\naturalSKFam$
of operators is not closed under taking adjoints.

\begin{theorem}
    \label{th:simpleWLForGraphons}
    Let $k \ge 1$
    and $U, W \colon X \times X \to [0,1]$ be graphons.
    The following are equivalent:
    \begin{enumerate}
        \item
            $t(F, U) = t(F, W)$ for every simple graph of treewidth at most $k-1$.
            \label{th:simpleWLForGraphons:homomorphisms}
        \item $\nusk_U = \nusk_W$.
            \label{th:simpleWLForGraphons:DIDM}
        \item
            There is a
            (permutation-inv.)
            Markov iso.\
            $R \colon \!\LTwoProdKQuoOf{\skSubAlg_W} \to \LTwoProdKQuoOf{\skSubAlg_U}$
            such that
            $\naturalSKFamU/\skSubAlg_U\circ R= R \circ \naturalSKFam/\skSubAlg_W$.
            \label{th:simpleWLForGraphons:minInvSubalgebra}
        \item
            There is a
            (permutation-inv.)
            Markov operator
            $S \colon \LTwoProdK \to \LTwoProdK$
            such that
            $(\naturalSKFamU)_{\skSubAlg_U} \circ S = S \circ (\naturalSKFam)_{\skSubAlg_W}$
            and
            $S^* \circ (\naturalSKFamU)_{\skSubAlg_U} = (\naturalSKFam)_{\skSubAlg_W} \circ S^*$.
            \label{th:simpleWLForGraphons:markovOperator}
        \item
            There are
            $\mu^{\otimes k}$-rel.\ comp.\ sub-$\sigma$-algebras $\subAlgC$, $\subAlgD$
            of $\Borel^{\otimes k}$
            that are simply $U$-invariant and simply $W$-invariant, respectively,
            and a Markov iso.\
            $R \colon \LTwoProdKQuoD \to \LTwoProdKQuo$
            such that
            $(\naturalSKFamU)_{\skSubAlg_U}/\subAlg \circ R= R \circ (\naturalSKFam)_{\skSubAlg_W}/\subAlgD$.
            \label{th:simpleWLForGraphons:markovIsomorphism}
    \end{enumerate}
\end{theorem}
\begin{proof}
    (\ref{th:simpleWLForGraphons:homomorphisms}) $\implies$ (\ref{th:simpleWLForGraphons:DIDM}):
    Follows from \Cref{le:skWLDConvergence}.

    (\ref{th:simpleWLForGraphons:DIDM}) $\implies$ (\ref{th:simpleWLForGraphons:minInvSubalgebra}):
    Analogous to \Cref{th:WLForGraphons}
    as we have
    both
    $\naturalSKFamU/\skSubAlg_U \circ \Rsk_U = \Rsk_U \circ \TTsk_{\nusk_U}$
    and
    $(\Rsk_W)^* \circ {\naturalSKFam}/\skSubAlg_W = {\TTsk_{\nusk_W}} \circ (\Rsk_W)^*$
    since $\Rsk_W$ is a Markov isomorphism.

    (\ref{th:simpleWLForGraphons:minInvSubalgebra}) $\implies$ (\ref{th:simpleWLForGraphons:markovOperator}):
    Set $S \coloneqq I_{\skSubAlg_U} \circ R \circ S_{\skSubAlg_W}$,
    which is a Markov operator as the composition of Markov operators.
    Then,
    \begin{align*}
        (\naturalSKFamU)_{\skSubAlg_U} \circ S= (\naturalSKFamU)_{\skSubAlg_U} \circ I_{\skSubAlg_U} \circ R \circ S_{\skSubAlg_W}&= I_{\skSubAlg_U} \circ \naturalSKFamU/{\skSubAlg_U} \circ R \circ S_{\skSubAlg_W} \tag{\Cref{le:quoOperator} (\ref{le:quoOperator:expOperatorI})}\\
        &= I_{\skSubAlg_U} \circ R \circ \naturalSKFam/{\skSubAlg_W} \circ S_{\skSubAlg_W}\\
        &= I_{\skSubAlg_U} \circ R \circ S_{\skSubAlg_W} \circ (\naturalSKFam)_{\skSubAlg_W} \tag{\Cref{le:quoOperator} (\ref{le:quoOperator:expOperator})}\\
        &= S \circ (\naturalSKFam)_{\skSubAlg_W}.
    \end{align*}
    Note that we neither used that $\skSubAlg_U$ is $\naturalSKFamU$-invariant
    nor that $\skSubAlg_W$ is $\naturalSKFam$-invariant.
    Since $R$ is a Markov isomorphism,
    we also have ${\naturalSKFamU}^*/\skSubAlg_U \circ R= R \circ {\naturalSKFam}^*/\skSubAlg_W$,
    which means that we obtain
    $({\naturalSKFamU}^*)_{\skSubAlg_U} \circ S = S \circ ({\naturalSKFam}^*)_{\skSubAlg_W}$
    in an analogous fashion.
    This implies the claim.
    Moreover, analogously to \Cref{th:WLForGraphons},
    if $R$ is permutation invariant, then so is $S$.

    (\ref{th:simpleWLForGraphons:markovOperator}) $\implies$ (\ref{th:simpleWLForGraphons:markovIsomorphism}):
    Follows immediately from \Cref{th:markovOperatorToIsomorphism}.

    (\ref{th:simpleWLForGraphons:markovIsomorphism}) $\implies$ (\ref{th:simpleWLForGraphons:homomorphisms}):
    Analogous to \Cref{th:WLForGraphons}.
\end{proof}

Also in this case, it is possible to define the space $\MsInfty$ and, for a graphon $W \colon X \times X \to [0,1]$,
the measure $\nusInfty_W \in \probMeas(\MsInfty)$.
Then, one obtains the following lemma
corresponding to
\Cref{le:inftyWLDConvergence},
where we now have a third characterization in terms of the cut distance,
denoted by
$\cutDist$, cf.\ \cite[Theorem $11.5$]{Lovasz2012}.

\begin{lemma}
    \label{le:inftysWLDConvergence}
    Let $(W_n)_{n}$ and $W \colon X \times X \to [0,1]$
    be a sequence of graphons and a graphon, respectively.
    Then, the following are equivalent:
    \begin{enumerate}
        \item
            $\nusInfty_{W_n} \rightarrow \nusInfty_W$.
            \label{le:inftysWLDConvergence:inftysWLD}
        \item
            $t(F, W_n) \rightarrow t(F, W)$
            for every simple graph~$F$.
            \label{le:inftysWLDConvergence:simplegraphs}
        \item $W_n \xrightarrow{\cutDist} W$.
            \label{le:inftysWLDConvergence:cutdist}
    \end{enumerate}
\end{lemma}

\subsection{Non-Oblivious Simple Weisfeiler-Leman}
\label{sec:nonObliviousSimpleWL}

As mentioned in the introduction, there are two non-equivalent
variants of $k$-WL for graphs:
oblivious $k$-WL and (non-oblivious) $k$-WL, where
$k$-WL is equivalent to oblivious $k+1$-WL
in the sense that
two graphs are distinguished
by $k$-WL if and only if they are distinguished
by oblivious $k+1$-WL \cite[Corollary V.$7$]{Grohe2021}.
(Non-oblivious)
$k$-WL is usually considered in the graph setting since it
needs less memory to achieve the same expressive power,
but the connections of oblivious $k$-WL
to other characterizations are much cleaner.
Examples of this are the system $\Lkk(G, H)$
of linear equations, where the $k$ directly corresponds to the $k$
of oblivious $k$-WL,
the logic $\mathsf{C}^{k}$,
the $k$-variable fragment of first-order logic with counting quantifiers,
and the maximum bag size in a tree decomposition,
although the latter is usually hidden by the fact that one
subtracts one from the maximum bag size in a tree decomposition to get
the width of such a decomposition.
Tree decompositions also give an explanation of the
difference between oblivious $k$-WL and non-oblivious $k$-WL:
intuitively, given a tree decomposition of width $k$,
we may dissect it into parts at bags of size $k+1$ or at bags of size $k$
yielding oblivious $k$-WL and non-oblivious $k$-WL, respectively.

Let us formally define non-oblivious $k$-WL.
Let $G$ be a graph and recall that
the atomic type $\atp_G(\bar{v})$ of a tuple $\bar{v} = (v_1, \dots, v_k) \in V(G)^k$
of vertices of $G$ is the $k \times k$-matrix $A$ with entries
$A_{ij} = 2$ if $v_i = v_j$, $A_{ij} = 1$ if $v_i v_j \in E(G)$, and
$A_{ij} = 0$ otherwise.
Then,
let $\wl^k_{G,0}(\bar{v}) \coloneqq \atp_G(\bar{v})$ and,
for every  $n \ge 0$, define
\begin{equation}
    \wl^k_{G, n+1}(\bar{v}) \coloneqq \left(\wl^k_{G, n}(\bar{v}), \MSet{(\atp_G(\bar{v}w), \big(\wl^k_{G, n}(\bar{v}[w/j])\big)_{j \in [k]}) \mid w \in V(G)}\right)
\end{equation}
for every $\bar{v} \in V(G)^k$.
We say that \textit{$k$-WL does not distinguish graphs $G$ and $H$}
if
$\MSet{\wl^k_{G,n}(\bar{v}) \mid \bar{v} \in V(G)^k}
    = \MSet{\wl^k_{H,n}(\bar{v}) \mid \bar{v} \in V(H)^k}$
for every $n \ge 0$.
Recall that the $j$-neighbor
$\bar{v}[w/j]$ denotes the $k$-tuple obtained from $\bar{v}$
by replacing the $j$th component by $w$.
The colorings computed by $1$-WL and color refinement
induce the same partition and,
in particular, $1$-WL distinguishes two graphs
if and only if color refinement does \cite[Proposition V.$4$]{Grohe2021}.

\begin{figure}
    \centering
    \begin{tikzpicture}
        \node[vertex, label={90:$a_1$}] (D1T) {};
        \node[vertex, below = of D1T, label={270:$b_1$}] (D1B) {};
        \path[draw, thick] (D1T) edge (D1B);

        \node[vertex, right = 1.5cm of D1T, label={90:$a_1$}] (I1T) {};
        \node[vertex, below = of I1T, label={270:$b_1$}] (I1B) {};
        \node[vertex, draw = white, fill = white] (I1G) at ($(I1T)!0.5!(I1B)$) {};

        \node[vertex, right = 1.5cm of I1G, label={90:$a_1$}, label={270:$b_1$}] (D2L) {};
        \node[vertex, right = of D2L] (D2R) {};
        \path[draw, thick] (D2L) edge (D2R);

        \node[vertex, right = 1.5cm of D2R, label={90:$a_1$}, label={270:$b_1$}] (I2L) {};
        \node[vertex, right = of I2L] (I2R) {};
    \end{tikzpicture}
    \caption{The (isomorphism types of) graphs in $\mathcal{F}^{\nonoSimpleS 1}$.}
    \label{fig:nonoSimpleSOne}
\end{figure}

Following the intuition that a tree decomposition of width $k$
can be dissect into parts either at bags of size $k+1$ or at bags of size $k$,
one can adapt the definitions of this section
to obtain a variant of simple $k$-WL akin to non-oblivious $k$-WL.
To this end, recall
the definition of forget, adjacency, and introduce graphs from
\Cref{def:differentgraphs}
and let
$\nonoSimplesK$ to be the set of all bi-labeled graphs
\begin{equation*}
    \bm{F}^{k+1}_{j_1} \circ \bigcirc_{i \in V} \bm{A}^{k+1}_{i j_1}
    \circ \bm{I}^{k+1}_{j_2}
    \in \graphsOf{k,k}
\end{equation*}
for $j_1, j_2 \in [k+1]$, $V \subseteq [k+1] \setminus \{j_1\}$.
Then, a term in $\treeClosure{\simpleAdjNeisKP}$
can be turned into a term in $\nonoSimpleTreeClosureK$
by essentially re-grouping the
introduce and forget graphs.
All definitions and results from this section transfer to
the set $\nonoSimplesK$ and, in particular,
one can obtain
a variant of
\Cref{th:simpleWLForGraphons}
without the mismatch of the $k$ of simple $k$-WL
and the $k$ of the treewidth.
Since it is so similar, however, we do not state it here.

For a last remark, consider the special case of $k = 1$
of fractional isomorphism.
The isomorphism types in $\nonoSimplesKOf{1}$ are shown
in \Cref{fig:nonoSimpleSOne};
they all are symmetric
in this special case.
We note that the graph on the far left is $\bm{A}$,
the edge with one input and one output vertex, cf.\ \Cref{ex:graphonOperators},
which satisfies $\OperatorFromTo{\bm{A}}{W} = T_W$.
Among the bi-labeled graphs in $\nonoSimplesKOf{1}$,
this is actually the only interesting graph since
it is necessary and already sufficient for the construction of all trees.
In other words, we can leave out the other bi-labeled graphs from $\nonoSimplesKOf{1}$.
Then, the resulting characterizations are essentially
\Cref{th:colRefGraphons}, the result of \citeauthor{GrebikRocha2021}.

\subsection*{Acknowledgements}

The author would like to thank Jan Greb\'ik for mentioning the idea
of combining all $k$-WL distributions of a graphon
into a single object to obtain a new characterization
of weak isomorphism, which led to \Cref{sec:measureHierarchies}.
Moreover, the author is grateful to the anonymous reviewers for their many
helpful suggestions and, in particular, for their advice on how to make the
introduction more accessible.

Funded by the European Union (ERC, SymSim, 101054974).
Views and opinions expressed are however those of the author only and
do not necessarily reflect those of the European Union or the European Research
Council. Neither the European Union nor the granting authority can be held
responsible for them.

\newpage
\printbibliography

\end{document}